\documentclass[a4paper,12pt]{amsart}

\usepackage{amssymb}
\usepackage{amsmath}
\usepackage{amsthm}

\usepackage[colorlinks]{hyperref}
\usepackage{hyperref}
\renewcommand\eqref[1]{(\ref{#1})} 

\allowdisplaybreaks \numberwithin{equation}{section}

\theoremstyle{plain}
\newtheorem{theorem}{Theorem}[section]
\newtheorem{prop}[theorem]{Proposition}
\newtheorem{corollary}[theorem]{Corollary}
\newtheorem{lemma}[theorem]{Lemma}
\newtheorem{assump}[theorem]{Assumption}

\theoremstyle{definition}
\newtheorem{defn}[theorem]{Definition}
\newtheorem{rem}[theorem]{Remark}
\newtheorem{example}[theorem]{Example}

\def\ind{{\mathcal I}}
\def\sL{{\star_{L}}}
\def\sLs{\widetilde{\star}_{L}}

\newcommand{\omp}{(\overline{\Omega})}
\newcommand{\efel}{\mathcal{F}_L}
\newcommand{\efela}{\mathcal{F}_{L^*}}

\newcounter{quotecount}
\newcommand{\MyQuote}[2]{\bigskip (#2)\hspace*{1cm}
\vspace{1cm}\addtocounter{quotecount}{1}%
     \parbox{12cm}{\em #1}}

\begin{document}

\title
{Nonharmonic analysis of boundary value problems}

\author[Michael Ruzhansky]{Michael Ruzhansky}
\address{
  Michael Ruzhansky:
  \endgraf
  Department of Mathematics
  \endgraf
  Imperial College London
  \endgraf
  180 Queen's Gate, London, SW7 2AZ
  \endgraf
  United Kingdom
  \endgraf
  {\it E-mail address} {\rm m.ruzhansky@imperial.ac.uk}
  }
\author[Niyaz Tokmagambetov]{Niyaz Tokmagambetov}
\address{
  Niyaz Tokmagambetov:
  \endgraf
    al--Farabi Kazakh National University
  \endgraf
  71 al--Farabi ave., Almaty, 050040
  \endgraf
  Kazakhstan
  \endgraf
  {\it E-mail address} {\rm niyaz.tokmagambetov@gmail.com}
  }

\date{\today}


\thanks{The first
 author was supported in parts by the EPSRC
 grant EP/K039407/1 and by the Leverhulme Grant RPG-2014-02.
 No new data was collected or generated during the course of the research.
 The work was also supported by the MESRK Grant 0773/GF4
 of the
Committee of Science, Ministry of Education and Science of the
Republic of Kazakhstan.
 }

\dedicatory{Dedicated to the memory of Professor Louis Boutet de Monvel (1941--2014)}

\maketitle

\begin{abstract}
In this paper we develop the global symbolic calculus of
pseudo--differential operators
generated by a boundary value problem for a
given (not necessarily self-adjoint or elliptic) differential operator.
For this, we also establish elements of a non-self-adjoint
distribution theory and the corresponding biorthogonal Fourier analysis.
We give applications of the developed analysis to obtain a-priori
estimates for solutions of boundary value problems that are elliptic within
the constructed calculus.
\end{abstract}

\tableofcontents

\section{Introduction}

In this paper we are interested in questions devoted to the global
solvability and further properties of boundary value problems in
$\mathbb R^n$. Given a problem for some pseudo-differential
operator $A$ with fixed boundary conditions in a domain
$\Omega\subset\mathbb R^n$, the main idea for our analysis is to
develop a suitable pseudo-differential calculus in which the given
boundary value problem can be solved and its solution can be
efficiently estimated. Such pseudo-differential calculus is
developed in terms of a `model' operator L with the same boundary
conditions in $\Omega$ for which we can introduce and work with
the global Fourier analysis expressed in terms of its
eigenfunctions. In general, such a model operator L does not have
to be self-adjoint, so we will be working with biorthogonal
systems rather than with an orthornomal basis to take into account a
possible non-self-adjointness. The operator L also does not have to
be elliptic.

Different powerful approaches to boundary value problems for
pseudo-dif\-feren\-tial operators have been already developed, see e.g.
Boutet de Monvel \cite{Boutet:Acta} and many subsequent works
by, among others, the Mazya school (see e.g. \cite{Mazya-Soloviev:bk}),
Melrose school (see e.g. \cite{Mazzeo-Melrose:fibred}),
or Schulze school (see e.g. \cite{Harutyunyan-Schulze:bk}), 
see also approaches in e.g. 
Eskin \cite{Eskin:bk-1981},
Schrohe and Schulze \cite{Schrohe-Schulze:Mellin-1999}, 
Melo, Schick and Schrohe \cite{Melo-Schick-Schrohe:K-2006},
Mitrea and Nistor \cite{Mitrea-Nistor:BVPs-2007},
Plamenevskii \cite{Plamenevskii:BVP-1997},
and references therein.

\smallskip
However, our approach is rather different from all these by being global in nature.
An example of such an approach is the toroidal calculus of
pseudo-differential operators on the torus $\mathbb T^n$ or
of the periodic pseudo-differential operators on $\mathbb R^n$.
A global analysis of pseudo-differential operators on the torus
based on the Fourier series representations of functions
with further applications to the spectral theory was
originated by Agranovich \cite{agran}, with further
developments of its different aspects
by Agranovich \cite{agran2}, Amosov \cite{Amosov},
Elschner \cite{Elschner}, McLean \cite{McLean}, Melo \cite{Me97},
Pr\"ossdorf and Schneider \cite{ProssdorfSchneider}, Saranen and
Wendland \cite{SaranenWendland}, Turunen
and Vainikko \cite{TurunenVainikko}, Vainikko and Lifanov
\cite{VainikkoLifanov1}, and others. However,
most of these papers deal with one-dimensional cases or with
classes of operators rather than with classes of symbols.
A consistent development of the application of the classical
Fourier series techniques in the analysis of
pseudo-differential operators on the torus
was developed by the first author and Turunen
in \cite{Ruzhansjy-Turunen:NFA,Ruzhansky-Turunen-JFAA-torus} and
can be also found in the
monograph \cite{RT}. For further extensions of this periodic analysis to the almost
periodic setting see e.g. Wahlberg \cite{Wahlberg:2009,Wahlberg:2012}.
The
classical Fourier series on a circle $\mathbb T=\mathbb R/\mathbb Z$
can be viewed as a unitary transform in the Hilbert
space $L^{2}(0,1)$ generated by the
operator of differentiation $(-i\frac{d}{dx})$ with periodic boundary conditions,
because the system of exponents $\{\exp(2\pi i\lambda x), \, \lambda\in \mathbb Z\}$
is a system of its eigenfunctions.

The analysis of this paper is the development of such ideas to a more
general setting without assuming that the problem has symmetries.
Instead of the differential operator $(-i\frac{d}{dx})$ in the
space $L^{2}(0,1),$ we consider a differential
operator ${\rm L}$ of order $m$ with smooth coefficients,
in the Hilbert space $L^{2}(\Omega),$ where
$\Omega\subset\mathbb R^{n}$ is an open subset.
We assume that L is equipped with some boundary conditions leading to
a discrete spectrum with its family
of eigenfunctions yielding a (biorthogonal) basis in $L^{2}(\Omega)$.
Moreover, L does not have to be self-adjoint.
General biorthogonal systems have been investigated by Bari \cite{bari}
which is a setting convenient for our constructions; see also
Gelfand \cite{Gelfand:on-Bari}. Similar
(slightly more general but essentially the same)
systems are also called `Hilbert systems' or `quasi-orthogonal systems' by
Bari \cite{bari} and Kac, Salem and Zygmund \cite{Kac-Salem-Zygmund}, respectively.

We then investigate the associated spaces of test
functions, distributions, `convolutions', Fourier transforms,
Sobolev spaces $H^s_{\rm L}(\Omega)$ and $l^p({\rm L})$ spaces on the `dual',
associated to L, and their properties such as the Hausdorff-Young inequality,
interpolation, and duality. A strong characteristic feature of this analysis is that it
is build upon biorthogonal systems rather than more familiar orthonormal bases.
Consequently, we introduce
difference operators acting on Fourier coefficients, and
the subsequent symbolic calculus of
pseudo-differential operators generated by a differential
operator ${\rm L}$. A formula for compositions of
pseudo-differential operators and other elements of the symbolic calculus
are obtained. It is shown that
pseudo-differential operators are bounded on $L^2$ under certain
conditions on their symbols. We also analyse ellipticity and a-priori
estimates for operators within this calculus.

\medskip
The exponential systems $\{e^{2\pi i\lambda x}\}_{\lambda\in\Lambda}$ on $L^2(0,1)$
for a discrete set $\Lambda$ possibly containing $\lambda\not\in \mathbb Z$ have been
considered by Paley and Wiener \cite{Paley-Wiener:book-1934} who called
such systems the {\em nonharmonic Fourier series} to emphasize the distinction with the
usual (harmonic) Fourier series when $\Lambda=\mathbb Z$. For further explanations and
developments of the nonharmonic analysis we refer to survey papers by Sedletskii
\cite{Sedletskii-1,Sedletskii-2} (see also an earlier survey \cite{Sedletskii-UMN}).
The difference between the harmonic and nonharmonic Fourier series in our context
is already exhibited by the operator ${\rm L}=-i\frac{d}{dx}$ in the
space $L^{2}(0,1),$ but with boundary conditions $h y(0)=y(1)$ for a fixed
$h>0$. In this case, the series of eigenfunctions (a building block for our analysis)
is `harmonic' for $h=1$ and `nonharmonic' for $h\not=1$.
In Example \ref{Example1} we explain this further and also complement it with a
number of explicit formulae.

From this point of view, the analysis of pseudo-differential operators on the torus
using the classical exponential bases as in \cite{Ruzhansky-Turunen-JFAA-torus},
or further extensions using representation coefficients on compact Lie groups
as in \cite{RT,Ruzhansky-Turunen:IMRN}, both fall within the realm of
`harmonic' analysis. The latter approach has further, still `harmonic' extensions,
for example for the global analysis of pseudo-differential operators on
the Heisenberg group \cite{Fischer-Ruzhansky:CRAS-Heisenberg},
graded Lie groups 
\cite{Fischer-Ruzhansky:CRAS-lower-bounds,Fischer-Ruzhansky:TM-2014,
Fischer-Ruzhansky:book},
or general type I locally compact groups \cite{Mantoiu-Ruzhansky:type-I}.

In the analysis of the present paper such symmetries are
in general lost, nevertheless we attempt to still mimic the harmonic analysis
constructions but in the new `nonharmonic' setting. Therefore, to also emphasize
such a difference, we may call our analysis the `nonharmonic analysis of
boundary value problems'. In spirit, this is similar to the global pseudo-differential
analysis on closed manifolds as in
\cite{Delgado-Ruzhansky:invariant, Delgado-Ruzhansky:CRAS}
partly based on the `nonharmonic' analysis on compact manifold by Seeley
\cite{see:ex,see:exp}. Such analysis becomes effective in a number of problems,
for example it was recently used in \cite{DR} to produce sharp kernel conditions
for Schatten classes of operators on compact manifolds, and in \cite{Dasgupta-Ruzhansky:TAMS}
to give characterisations of Komatsu-type classes of functions and distributions, in particular
for classes of Gevrey functions and ultradistributions, on a compact manifold, extending the 
characterisation given for analytic functions by Seeley \cite{see:exp}.

The analysis of \cite{Delgado-Ruzhansky:invariant} deals with general compact
manifolds, but is simplified by the facts that there are no boundary conditions,
the operator L is self-adjoint, elliptic and positive, and the considered calculus is that of invariant
operators.

We keep the setting of this paper rather abstract, in particular not relying
on a specific form of boundary conditions of the operator for our analysis.
Certainly, if more information on the operator L and its properties are available,
more conclusions can be drawn. In Section \ref{SEC:L-examples} we give
several examples of operators and boundary conditions. In a somewhat related
setting, the global pseudo-differential analysis based on an elliptic
self-adjoint pseudo-differential operator on a closed manifold has been
recently developed in \cite{Delgado-Ruzhansky:invariant}.

Although in this paper we do not give explicit applications to partial differential equations,
these will appear elsewhere. For example, the analysis developed here could allow one
to treat classes of PDE problems in cylindrical domains of finite length without assuming periodic
boundary conditions on the top and bottom edges of the cylindrical domain, see
e.g. Denk and Nau \cite{Denk-Nau:cylindrical-Edi-2013} for this kind of problems.
Also, in subsequent works we will apply the pseudo-differential analysis developed here
to problems in punctured domains with $\delta$-type potentials, for PDE problems of
the type that appeared in \cite{KNT:Russian-Math-2015,KT:Doklady-2015}.

\vspace{3mm}

Let us formulate the main assumptions of this paper. We will consider
a differential operator ${\rm L}$ of order $m$ with smooth coefficients
on an open set $\Omega\subset\mathbb R^n$ equipped with some boundary
conditions. In order to describe the abstract scheme we will denote the
boundary conditions by (BC) without specifying them further in the general
framework. Concerning the boundary conditions we will assume that

\MyQuote{the boundary conditions {\rm (BC)} are {\em linear},
i.e. they are preserved under linear combinations or, in other words,
the spaces of functions satisfying {\rm (BC)} are linear.}{BC}

\vspace{-5mm}
In this paper we prefer to think of the operator in terms of its boundary conditions instead of domain,
in view of the planned further applications. However, in the paper we may use both points of view.

Later on, once introducing topologies on spaces of functions in the domain of ${\rm L}$, 
we will assume the condition (BC+) that the boundary conditions define a closed space.
In Section \ref{SEC:L-examples} we give different examples
of operators ${\rm L}$ and boundary conditions (BC). 

The assumption (BC) may be reformulated by saying that the domain ${\rm Dom(L)}$ of the operator
${\rm L}$ is linear, and the condition (BC+) by saying that ${\rm Dom(L)}$ and ${\rm Dom}({\rm L}^*)$
are closed in the topologies of $C_{\rm L}^\infty(\overline{\Omega})$ and  $C_{{\rm L}^*}^\infty(\overline{\Omega})$, 
respectively, with the latter spaces and their topologies introduced in Definition \ref{TestFunSp}.

\medskip
Also, we will be working with discrete sets of eigenvalues and eigenfunctions indexed by 
a countable set $\ind$.
However, in different problems it may be more convenient to make different
choices for this set, e.g. $\ind=\mathbb N$ or $\mathbb Z$  or $\mathbb Z^k$, etc.
In order to allow different applications we will be denoting it by $\ind$,
and without loss of generality we
will assume that 
\begin{equation}\label{EQ:ind}
\ind \textrm{ is a subset of } \mathbb Z^{K} \textrm{ for some } K\geq 1.
\end{equation}
For simplicity, one can think of $\ind=\mathbb Z$ or $\ind=\mathbb N\cup\{0\}$
throughout this paper.
Thus, throughout this paper we will be always working in the following setting:

\begin{assump}
\label{Assumption_1}
Let $\Omega\subset \mathbb R^{n}$, $n\geq 1$, be a bounded open set.
By ${\rm L}_\Omega$ we denote a differential operator ${\rm L}$ of
order $m$ with smooth coefficients in $\Omega$, equipped with some linear boundary
conditions {\rm (BC)}.
Assume that ${\rm L}_\Omega$ has a discrete spectrum
$\{\lambda_{\xi}\in\mathbb C: \, \xi\in\ind\}$
on $L^{2}(\Omega)$, where $\ind$ is a countable set as in \eqref{EQ:ind},
and we order the eigenvalues with
the occurring multiplicities in the ascending order:
\begin{equation}\label{EQ: EVOrder}
 |\lambda_{j}|\leq|\lambda_{k}| \quad\textrm{ for } |j|\leq |k|.
\end{equation}
\end{assump}


Let us denote by $u_{\xi}$  the eigenfunction of ${\rm L}_{\Omega}$ corresponding to the
eigenvalue $\lambda_{\xi}$ for each $\xi\in\ind$, so that
\begin{equation}
\label{SpecPr} {\rm L}u_{\xi}=\lambda_{\xi}u_{\xi} \,\,\,\,\,\,
\textrm{ in }  \Omega,\quad \textrm{ for all } \xi\in\ind.
\end{equation}
Here the eigenfunctions $u_\xi$ satisfy the boundary conditions
(BC) discussed earlier. The conjugate spectral problem is
\begin{equation}
\label{ConjSpecPr} {\rm
L^{\ast}}v_{\xi}=\overline{\lambda}_{\xi}v_{\xi}\,\,\,\,\,\,
\textrm{ in } \Omega\quad \textrm{ for all } \xi\in\ind,
\end{equation}
which we equip with the conjugate boundary conditions which we may denote by (BC)$^*$.
This adjoint problem will be denoted by ${\rm L}_\Omega^*$.

Let $\|u_{\xi}\|_{L^{2}}=1$ and $\|v_{\xi}\|_{L^{2}}=1$ for all
$\xi\in\ind.$ Here, we can take biorthogonal systems
$\{u_{\xi}\}_{\xi\in\ind}$ and $\{v_{\xi}\}_{\xi\in\ind}$,
i.e.
\begin{equation}\label{BiorthProp}
(u_{\xi},v_{\eta})_{L^2}=0 \,\,\,\, \hbox{for}
\,\,\,\, \xi\neq\eta, \,\,\,\, \hbox{and} \,\,\,\,
(u_{\xi},v_{\eta})_{L^2}=1 \,\,\,\, \hbox{for} \,\,\,\, \xi=\eta,
\end{equation}
where
$$(f, g)_{L^{2}}:=\int_{\Omega}f(x)\overline{g(x)}dx$$ is the usual
inner product of the
Hilbert space $L^{2}(\Omega)$. From N.K. Bari's work
\cite{bari} it follows that the system $\{u_{\xi}: \,\,\,
\xi\in\ind\}$ is a basis in $L^{2}(\Omega)$ if and only if the
system $\{v_{\xi}: \,\,\, \xi\in\ind\}$ is a basis in
$L^{2}(\Omega)$.
So, from now on we will also assume:

\begin{assump}
\label{Assumption_3}
The system
$\{u_{\xi}: \; \xi\in\ind\}$ is a basis in $L^{2}(\Omega)$, i.e.
for every $f\in L^{2}(\Omega)$ there exists a unique series
$\sum_{\xi\in\ind} a_\xi u_\xi(x)$ that converges to $f$ in  $L^{2}(\Omega)$.
\end{assump}

Therefore, by Bari \cite{bari}, the system $\{v_{\xi}: \,\,\, \xi\in\ind\}$ is also
a basis in $L^{2}(\Omega)$.
Also, Assumption \ref{Assumption_3}  will imply that the spaces
$C^\infty_{{\rm L}}(\overline\Omega)$ and $C^\infty_{{\rm L}^*}(\overline\Omega)$
of test functions introduced in Section \ref{SEC:TD}
are dense in $L^2(\Omega)$.

\medskip
Let us define the weight
\begin{equation}\label{EQ:angle}
\langle\xi\rangle:=(1+|\lambda_{\xi}|^2)^{\frac{1}{2m}},
\end{equation}
which will be instrumental in measuring the growth/decay of Fourier coefficients and of
symbols.
To give its interpretation in terms of the operator analysis, we can define the operator
${\rm L}^{\circ}$ by setting its values on the basis $u_{\xi}$ by
\begin{equation}\label{EQ:Lo-def}
{\rm L}^{\circ} u_{\xi}:=\overline{\lambda_{\xi}} u_{\xi},\quad
\textrm{ for all } \xi\in\ind.
\end{equation}
If L is self-adjoint, we have ${\rm L}^{\circ}={\rm L}^*={\rm L}$.
Consequently, we can informally think of $\langle\xi\rangle$
as of the eigenvalues of the positive (first order) operator
$({\rm I}+{\rm L^\circ\, L})^{\frac{1}{2m}}.$

With a similar definition for $({\rm L}^{*})^{\circ}$, we can observe that
$({\rm L}^{*})^{\circ}=({\rm L}^{\circ})^{*}$.



\medskip
The following technical definition will be useful to single out the case when the eigenfunctions
of both L and ${\rm L}^*$ do not have zeros (WZ stands for `without zeros'):

\begin{defn}
\label{DEF: WZ-system}
The system $\{u_{\xi}: \,\,\, \xi\in\ind\}$ is called a ${\rm WZ}$-system if the functions $u_{\xi}(x), \,
v_{\xi}(x)$ do not have zeros on the domain $\overline{\Omega}$
for all $\xi\in\ind$, and if there exist $C>0$
and $N\geq0$ such that
$$
\inf\limits_{x\in\Omega}|u_{\xi}(x)|\geq
C\langle\xi\rangle^{-N},
$$
$$
\inf\limits_{x\in\Omega}|v_{\xi}(x)|\geq
C\langle\xi\rangle^{-N},
$$
as $|\xi|\to\infty$.
\end{defn}
Here WZ stands for `without zeros'. We note that, in particular, a WZ-system can not be all
real-valued due to orthogonality relations \eqref{BiorthProp}.
Several examples of WZ-systems will be given in Section \ref{SEC:L-examples},
but a typical example is the system $\{e^{2\pi i \lambda x}\}_{\lambda\in\mathbb Z}$ for
${\rm L}=-i\frac{d}{dx}$ on the circle $\mathbb T=\mathbb R/\mathbb Z$.

In the sequel, unless stated otherwise, whenever we will use inverses $u_{\xi}^{-1}$
of the functions $u_{\xi}$,
we will suppose that the system $\{u_{\xi}: \; \xi\in\ind\}$ is a ${\rm WZ}$-system.
However, we will also try to mention explicitly when we make such an additional assumption.

\medskip
The paper is organised as follows.
\begin{itemize}
\item Section \ref{SEC:L-examples}: we give examples of
operators L and of different boundary conditions yielding different types of biorthogonal
systems.
\item Section \ref{SEC:TD}: we introduce elements of the global
theory of distributions $\mathcal D'_{\rm L}(\Omega)$
in $\Omega$ adapted to the boundary value problem ${\rm L}_\Omega$.
\item Section \ref{SEC:FT}: we develop the Fourier transform induced by L, which is the decomposition
of elements of $\mathcal D'_{\rm L}(\Omega)$ with respect to the eigenfunctions of L. Here is a point
when both biorthogonal bases $u_\xi$ and $v_\xi$ came actively into play.
\item Section \ref{SEC:conv}: we introduce L-convolution $\sL$, which is an operation resembling the usual
convolution. Despite the lack of any symmetries in our problem a
number of useful properties of such L-convolution can still be
obtained.
\item Section \ref{SEC:Sobolev}: we introduce the space
$l^2_{\rm L}$ for which the Plancherel identity for the L-Fourier transform holds. Consequently, we
introduce Sobolev spaces $\mathcal H^{s}_{{\rm L}}(\Omega)$ and describe their Fourier images.
\item Section \ref{SEC:lp}: we introduce the spaces $l^{p}({\rm L})$ and $l^{p}({\rm L}^*)$ extending
the spaces $l^2_{\rm L}$ and $l^2_{{\rm L}^*}$ to the $l^p$-setting. We show that these spaces
are interpolation spaces and satisfy the expected duality properties. Moreover, we obtain the
Hausdorff-Young inequality for the L-Fourier transform in these spaces.
\item Section \ref{SEC:Schwartz}: we prove the Schwartz kernel theorem in the distribution spaces
$\mathcal D'_{\rm L}(\Omega)$. This is necessary to set up the subsequent framework of the symbolic
analysis and of the definition of the symbol as the L-Fourier transform of the L-convolution kernel.
\item Section \ref{SEC:quantization}: we introduce difference operators acting on Fourier coefficients
and on symbols. Keeping in mind ideas from the Calder\'on-Zygmund theory, these are defined as
multiplications on the inverse Fourier transform side by functions vanishing at an anticipated singular
support of the integral kernel. Due to the lack of symmetries (as compared e.g. to the cases of the
torus or of compact Lie groups) these difference operators also depend on the points $x$ of the space.
\item Section \ref{SEC:differences}: the notion of difference operators
is used to define H\"ormander type classes
induced by the boundary value problem ${\rm L}_\Omega$ and to develop elements of its
symbolic calculus.
\item Section \ref{SEC:kernels}: we derive some properties of the integral kernels of pseudo-differential operators.
\item Section \ref{SEC:elliptic}: we show that operators that are elliptic in the
constructed symbol classes have both left and right parametrices and provide a formula for it.
\item Section \ref{SEC:embeddings}: we discuss possible Sobolev embedding theorems. In particular,
it seems that in order to have a meaningful collection of embeddings further assumptions on the boundary
value problem ${\rm L}_\Omega$ may be needed.
\item Section \ref{SEC:L2}: we prove a criterion for the $L^2$-boundedness of pseudo-differential operators
in terms of their symbols, and extend it to Sobolev spaces as well. An application is given to
obtain a-priori estimates for solutions to boundary value problems to elliptic operators.
\end{itemize}

Most results, especially those starting from Section \ref{SEC:conv} appear to be new already in the case when
the problem ${\rm L}_\Omega$ is self-adjoint.

\smallskip
The authors would like to thank Julio Delgado for discussions. Applications of the approach of this paper
to Schatten classes and nuclearity properties as well as estimates on eigenvalue asymptotics for boundary value problems
will appear in \cite{Delgado-Ruzhansky-Togmagambetov:nuclear}.

\section{Examples of operators ${\rm L}$ and boundary conditions}
\label{SEC:L-examples}

In this section we give several examples of the operator ${\rm L}$ and of boundary conditions (BC).
The following example shows that among other things, we can extend to the
non-self-adjoint setting the toroidal calculus (with periodic boundary conditions) developed
in \cite{Ruzhansky-Turunen-JFAA-torus}.

\begin{example}\label{Example1}
For $h>0$, let the operator ${\rm O}_{h}^{(1)}$ be given by the expression
$$
{\rm O}_{h}^{(1)}:= -i\frac{\partial}{\partial x}
$$
on the domain $\Omega=(0, 1)$ with the boundary condition
$$
h y(0)=y(1).
$$
In the case $h=1$ we get the operator ${\rm O}_{1}^{(1)}$ with
periodic boundary conditions. In this case
the ${\rm O}_{1}^{(1)}$-pseudo-differential calculus developed in this
paper coincides with the toroidal calculus some aspects of which were
investigated in the works by
Agranovich \cite{agran, agran2}, Turunen and Vainikko
\cite{TurunenVainikko}, and which was then consistently developed
by Ruzhansky and Turunen in \cite{Ruzhansky-Turunen-JFAA-torus}. Thus, of main interest to us
here will be the calculus generated by ${\rm O}_{h}^{(1)}$ with $h\not=1$.

It is easy to check that for $h\not =1$ the operator ${\rm O}_{h}^{(1)}$ is not
self-adjoint. Spectral properties of the operator ${\rm
O}_{h}^{(1)}$ are well-known (see Titchmarsh \cite{titc} and Cartwright \cite{cart}): with $\ind=\mathbb Z$,

\vspace{2mm}

{\bf A.} ${\rm O}_{h}^{(1)}$ has a discrete spectrum and its eigenvalues satisfy
$$
\lambda_{j}=-i\ln h+2j\pi, \ j\in \mathbb{Z}.
$$

\vspace{2mm}

{\bf B.} The system of eigenfunctions
$$
\{u_{j}(x)=h^{x}e^{ 2\pi i x j },\,\, j\in \mathbb{Z}\}
$$
of the operator ${\rm O}_{h}^{(1)}$ is a minimal system in the
space $L^{2}(\Omega),$ and the biorthogonal system to
$\{u_{j}(x)=h^{x}e^{ 2\pi i x j  },\,\, j\in \mathbb{Z}\}$ in
$L^{2}(\Omega)$ is
$$
\{v_{j}(x)=h^{-x} e^{2\pi i x j },\,\, j\in \mathbb{Z}\}.
$$

\vspace{2mm}

{\bf C.} The system of eigenfunctions of the operator ${\rm
O}_{h}^{(1)}$ is a Riesz basis in $L^{2}(\Omega).$
These families also form WZ-systems (without zeros, as
in Definition \ref{DEF: WZ-system}).

\vspace{2mm}

{\bf D.} The resolvent of the operator ${\rm O}_{h}^{(1)}$ is
$$
({\rm O}_{h}^{(1)}-\lambda I)^{-1}f(x)=
\frac{i}{\Delta(\lambda)}e^{i\lambda(x+1)}\int_{0}^{1}e^{-i\lambda
t}f(t)dt
+i e^{i\lambda x}\int_{0}^{x} e^{-i\lambda t}f(t)dt,
$$
where
$$
\Delta(\lambda)=h-e^{i\lambda}.
$$

\vspace{5mm}

The above example fits into our framework once we index the family of eigenvalues and
of the corresponding eigenfunctions by $\ind=\mathbb Z$ which is a choice we made for the (discrete)
index set. In Section \ref{SEC:conv} we will discuss convolutions generated by our
operators ${\rm L}_\Omega$. In this example,
the convolution generated by the operator ${\rm O}_{h}^{(1)}$ has
the following explicit form
$$
(g\star_{{\rm O}_{h}^{(1)}}
f)(x)=\int^{x}_{0}g(x-t)f(t)dt+\frac{1}{h}\int^{1}_{x}g(1+x-t)f(t)dt.
$$
For more details on this particular convolution see \cite{Kanguzhin_Tokmagambetov} and
\cite{Kanguzhin_Tokmagambetov_Tulenov}.
\end{example}

\begin{rem}\label{REM:toroidal}
The toroidal pseudo-differential calculus on all higher dimensional tori $\mathbb T^n$, $n\geq 1$,
as outlined in \cite{RT07} and then consistently
developed in \cite{Ruzhansky-Turunen-JFAA-torus}, can not be covered by the first order differential operator
${\rm O}_{1}^{(1)}$. But in this case we can take the second order operator.
Namely, identifying the torus $\mathbb T^n$ with the cube $[0,1]^n$ with periodic
boundary conditions,
we can take
$
{\rm L}:= \Delta
$
to be the Laplacian with the periodic boundary conditions
on the boundary of  $\Omega=(0,1)^n$. See Example \ref{example-Rn} for further details.
\end{rem}

\begin{example}\label{EX:sigma}
The operator ${\rm L}=i\frac{d}{dt}$ with $\Omega=(-a,a)$ and the boundary condition
$$
\int_{-a}^a y(t) d\sigma(t)=0,\quad {\rm var}\, \sigma(t)<\infty,\quad 0<a<\infty,
$$
has the eigenfunctions in the form $\{\exp(i\lambda_k t)\}_{\lambda_k\in\Lambda},$
where $\Lambda\subset\mathbb C$ is the collection of zeros of the Fourier transform
$\widehat{d\sigma}$ of the measure $d\sigma(t).$
It becomes a biorthogonal system or a Riesz basis under a number of properties of
$\Lambda$, see Sedletskii \cite{Sedletskii-1} for a thorough review of this topic,
see also \cite{Sedletskii-2}.
\end{example}

\begin{example}\label{EX:sigma2}
Combining Example \ref{Example1} and Example \ref{EX:sigma}, we can consider
operator ${\rm L}=-i\frac{d}{dx}$ with $\Omega=(0,1)$ with the domain
$$
{\rm Dom}({\rm L})=\left\{y\in W_2^1[0,1]: ay(0)+by(1)+\int_{0}^1 y(x) q(x) dx=0\right\},
$$
where $a\not=0$, $b\not=0$, and $q\in C^1[0,1]$.
We assume that $$a+b+\int_0^1 q(x) dx=1$$ so that the inverse ${\rm L}^{-1}$ exists and is bounded.
Following \cite{Kanguzhin-Nurahmetov:Kaz-2002}, we have the following properties, with $\ind=\mathbb Z$:

\vspace{2mm}

{\bf A.} The operator ${\rm L}$ has a discrete spectrum and its eigenvalues can be enumerated so that
$$
\lambda_{j}=-i\ln (-\frac{a}{b})+2j\pi+\alpha_j, \ j\in \mathbb{Z},
$$
and for any $\epsilon>0$ we have $$\sum_{j\in\mathbb Z} |\alpha_j|^{1+\epsilon}<\infty.$$
If $m_j$ denotes the multiplicity of the eigenvalue $\lambda_j$, then $m_j=1$ for sufficiently large $j$.
\vspace{2mm}

{\bf B.} The system of extended eigenfunctions
\begin{equation}\label{EQ:exev}
\left\{u_{jk}(x)=\frac{(ix)^k}{k!} e^{ i \lambda_j x }:\quad 0\leq k\leq m_j-1,\; j\in \mathbb{Z}\right\}
\end{equation}
of the operator ${\rm L}$ is a minimal system in the
space $L^{2}(0,1),$ and its biorthogonal system is given by
$$
v_{jk}(x)=\lim_{\lambda\to\lambda_j} \frac{1}{k!} \frac{d^k}{d\lambda^k}
\left(\frac{(\lambda-\lambda_j)^{m_j}}{\Delta(\lambda)}
(ibe^{i\lambda(1-x)}+i\int_x^1 e^{i\lambda(t-x)}q(t) dt)
\right),
$$
$0\leq k\leq m_j-1,\;  j\in \mathbb{Z}$, where
$$\Delta(\lambda)=a+b e^{i\lambda}+\int_0^1  e^{i\lambda x} q(x) dx.$$
The eigenvalues of ${\rm L}$ are determined by the equation $\Delta(\lambda)=0.$
\vspace{2mm}

{\bf C.} The system
$\{u_{jk}\}_{0\leq k\leq m_j-1,\; j\in \mathbb{Z}}$
of extended eigenfunctions \eqref{EQ:exev}
of the operator ${\rm L}$ is a Riesz basis in $L^{2}(0,1).$
Any $f\in {\rm Dom}({\rm L})$ has a decomposition in a uniformly convergent
series of functions in \eqref{EQ:exev}.
Moreover, the eigenfunctions $e^{i\lambda_j x}$ satisfy
\begin{equation}\label{EQ:exl2}
\sum_{j\in\mathbb Z} \|e^{i\lambda_j x}-e^{i 2\pi j x}\|^2_{L^2(0,1)}<\infty.
\end{equation}
In particular, this implies that modulo finitely many elements, the system
\eqref{EQ:exev} is a WZ-system (without zeros, as
in Definition \ref{DEF: WZ-system}).
\end{example}

\begin{example}\label{example-Rn}
We now consider operator ${\rm L}={\rm O}_h^{(n)}$, the analogue of 
Example \ref{Example1} in higher dimensions. 
Let 
$$\Omega:=(0,1)^n \textrm{ and } h>0 \textrm{ i.e. } h=(h_1,\dots,h_n)\in\mathbb R^{n} : h_j>0 \mbox{ for every } j=1,\dots,n.$$ 
The operator ${\rm O}_h^{(n)}$ on $\Omega$ is defined by the differential operator 
\begin{equation}\label{lae1} 
{\rm O}_h^{(n)}:=\Delta=\sum_{j=1}^{n}\frac{\partial^2}{\partial x_j^2},
\end{equation}
together with the boundary conditions (BC):
\begin{equation}\label{EQ:BCh}
h_{j} f(x)|_{x_{j}=0}=f(x)|_{x_{j}=1},\quad
h_{j} \frac{\partial f}{\partial x_j}(x)|_{x_{j}=0}=\frac{\partial f}{\partial x_j}(x)|_{x_{j}=1},
\quad j=1,\ldots,n,
\end{equation}
and the domain
\[{\rm Dom}({{\rm O}_h^{(n)}})=\{f\in L^{2}(\Omega) : \Delta f\in L^{2}(\Omega):\;
f \textrm{ satisfies \eqref{EQ:BCh}} \}.\]

In order to describe the corresponding biorthogonal system, we first note that since $b^0=1$ for all $b>0$, we can define $0^0=1$. In particular
 we write 
 $$h^x=h_1^{x_1}\cdots h_n^{x_n}=\prod\limits_{j=1}^{n}h_{j}^{x_{j}}$$ 
 for $x\in [0,1]^n$. 
Then, with $\ind=\mathbb{Z}^{n}$,
the system of
eigenfunctions of the operator $L_h$ is
$$\{u_{\xi}(x)=h^{x}e^{2\pi i x\cdot\xi}, \;\xi\in
\mathbb{Z}^{n}\}$$ 
and the conjugate system is
$$\{v_{\xi}(x)=h^{-x}e^{ 2\pi i x\cdot\xi },\; \xi\in
\mathbb{Z}^{n}\},$$ 
where $x\cdot\xi =x_{1}\xi_{1}+ \ldots + x_{n}\xi_{n}$. Note that
$u_{\xi}(x)=\prod\limits_{j=1}^{n}u_{\xi_{j}}(x_{j}),$ where
$u_{\xi_{j}}(x_{j})=h_{j}^{x_{j}}e^{{ 2\pi i x_{j}\xi_{j} }}$.
\end{example}

\begin{example}
Various sine and cosine systems appear as biorthogonal systems as well. One interesting
example is the collection of
$$
\sin(k-1/4)t, \quad k\in\mathbb N,
$$
which appears as a system of
eigenfunctions of the Sturm-Liouville problem after the separation of variables in the
Lavrent'ev-Bicadze equation with special boundary conditions, see
Ponomarev \cite{Ponomarev-DAN-1979}.
Shkalikov \cite{Shkalikov-sins}
showed that this system yields a Riesz basis in $L^2(0,\pi).$
See also Sedletskii \cite[p. 146]{Sedletskii-1} for more perspective on this problem.
\end{example}

\begin{example}
Let ${\rm O}^{(m)}$ be an ordinary differential operator in
$L^{2}(0, 1)$ of order $m$ generated by the
differential expression
\begin{equation}
l(y)\equiv y^{(m)}(x)+\sum_{k=0}^{m-1}p_{k}(x)y^{(k)}(x), \quad
0<x<1, \label{EQ1A}
\end{equation}
with coefficients
$$
p_{k}\in C^{k}[0,1], \,\,\, k=0,1,\ldots,m-1,
$$
and boundary conditions
\begin{equation}
U_{j}(y)\equiv
V_{j}(y)+\sum\limits_{s=0}^{k_{j}}\int\limits_{0}^{1}y^{(s)}(t)\rho_{js}(t)dt=0,
\quad  j=1, \ldots, m, \label{EQ2A}
\end{equation}
where
$$
V_{j}(y)\equiv
\sum\limits_{s=0}^{j}[\alpha_{js}y^{(k_{s})}(0)+\beta_{js}y^{(k_{s})}(1)],
$$
with $\alpha_{js}$ and $\beta_{js}$ some real numbers, and $\rho_{js}\in
L^{2}(0, 1)$ for all $j$ and $s$.

Furthermore, we suppose that the boundary conditions (\ref{EQ2A})
are normed and strong regular in the sense considered by Shkalikov in \cite{Shkalikov}.
Then it can be shown that the eigenvalues have the same algebraic and
geometric multiplicities and, after a suitable adaption for our case, we have

\begin{theorem}[\cite{Shkalikov}]
The eigenfunctions of the operator ${\rm O}^{(m)}$ with strong regular
boundary conditions (\ref{EQ2A}) form a Riesz basis in $L^{2}(0, 1)$.
\label{TH: Shkalikov}
\end{theorem}

In the monograph of Naimark \cite{Naimark} the spectral
properties of differential operators generated by the differential
expression (\ref{EQ1A}) with the boundary conditions (\ref{EQ2A})
without integral terms were considered. The statement as in Theorem
\ref{TH: Shkalikov} was established in this setting, with the asymptotic formula
for the Weyl eigenvalue counting function $N(\lambda)$ in the form
\begin{equation}
\label{EQ: DistrAsymp_4} N(\lambda)\sim C \lambda^{1/m} \,\,\,\,
\hbox{as} \,\,\,\, \lambda\rightarrow+\infty.
\end{equation}
\end{example}

\begin{example}
Let ${\rm E}_{s}$ be a realisation in $L^{2}(\Omega)$ of a
regular elliptic boundary value problem, i.e. such that
the underlying differential
operator is uniformly elliptic and has smooth coefficients on
an open bounded set $\Omega\subset \mathbb R^{n}$, and that the
boundary conditions determining ${\rm E}_{s}$ are also regular in
some sense. Suppose that ${\rm E}_{s}$ is a self-adjoint elliptic
operator, so that ${\rm E}_{s}$ has a basis of eigenfunctions
in $L^{2}(\Omega)$.

The earliest results on the
asymptotic form of the eigenvalue counting function $N(\lambda)$ were obtained in
1911 by Weyl
\cite{Weyl} for the case of the negative Laplacian $-\Delta$ in
two dimensions. Using the theory of integral equations, Weyl
derived the formula
\begin{equation}
\label{EQ: DistrAsymp_1} N(\lambda)\sim
\frac{\mu_{2}(\Omega)}{4\pi} \lambda \,\,\,\, \hbox{as} \,\,\,\,
\lambda\rightarrow+\infty,
\end{equation}
where $\mu_{2}(\Omega)$ denotes the area of $\Omega$. In three
dimensions, this becomes
\begin{equation}
\label{EQ: DistrAsymp_2} N(\lambda)\sim
\frac{\mu_{3}(\Omega)}{6\pi^{2}} \lambda^{3/2} \,\,\,\, \hbox{as}
\,\,\,\, \lambda\rightarrow+\infty.
\end{equation}
The problem was then developed by Courant
\cite{Courant_1, Courant_2, Courant_3} and \cite{Courant_4}, who extended the
formulae of Weyl to further settings. In 1934, Carleman \cite{Carleman} introduced
Tauberian methods reminiscent of analytic number theory into the
study of Weyl asymptotic formulae. Using Carleman's results
Clark \cite{Clark} provided a rather general
asymptotic formula
\begin{equation}
\label{EQ: DistrAsymp_3} N(\lambda)\sim C \lambda^{n/m} \,\,\,\,
\hbox{as} \,\,\,\, \lambda\rightarrow+\infty,
\end{equation}
for the operator ${\rm E}_{s}$, where $m$ is order of ${\rm E}_{s}$.
The second term of the spectral asymptotic was obtained by Duistermaat and
Guillemin \cite{DG}. This has been extended further to elliptic systems, see the book
\cite{SV-book} by Safarov and Vassiliev for a survey, as well as to systems with
multiplicities, see Kamotski and Ruzhansky \cite{Kamotski-Ruzhansky:CPDE} and references therein.
\end{example}

\section{Global distributions generated by the boundary value problem}
\label{SEC:TD}

In this section we describe the spaces of distributions generated by the boundary value problem
${\rm L}_\Omega$ and by its adjoint ${\rm L}_\Omega^*$ and the related global Fourier analysis.
The more far-reaching aim of this analysis is to establish a version of the Schwartz kernel
theorem for the appearing spaces of distributions equipped with the corresponding boundary
conditions.
We first define the space $C_{{\rm L}}^{\infty}(\overline{\Omega})$ of test functions.

\begin{defn}\label{TestFunSp}
The space
$C_{{\rm L}}^{\infty}(\overline{\Omega}):={\rm Dom}({\rm L}_\Omega^{\infty})$ is called the
space of test functions for ${\rm L}_\Omega$. Here we define
$$
{\rm Dom}({\rm L}_\Omega^{\infty}):=\bigcap_{k=1}^{\infty}{\rm Dom}({\rm
L}_\Omega^{k}),
$$
where ${\rm Dom}({\rm L}_\Omega^{k})$, or just ${\rm Dom}({\rm L}^{k})$ for simplicity, is the domain of the
operator ${\rm L}^{k}$, in turn defined as
$$
{\rm Dom}({\rm L}^{k}):=\{f\in L^{2}(\Omega): \,\,\, {\rm
L}^{j}f\in {\rm Dom}({\rm L}), \,\,\, j=0, \,1, \, 2, \ldots,
k-1\}.
$$
We note that in this way all operators ${\rm L}^{k}$, $k\in\mathbb N$, are being equipped with the same boundary conditions (BC).
The Fr\'echet topology of $C_{{\rm L}}^{\infty}(\overline{\Omega})$ is given by the family of norms
\begin{equation}\label{EQ:L-top}
\|\varphi\|_{C^{k}_{{\rm L}}}:=\max_{j\leq k}
\|{\rm L}^{j}\varphi\|_{L^2(\Omega)}, \quad k\in\mathbb N_0,
\; \varphi\in C_{{\rm L}}^{\infty}(\overline{\Omega}).
\end{equation}

Analogously to the ${\rm L}$-case, we introduce the space $C_{{\rm
L^{\ast}}}^{\infty}(\overline{\Omega})$ corresponding to the adjoint operator ${\rm L}_\Omega^*$ by
$$
C_{{\rm L^{\ast}}}^{\infty}(\overline{\Omega}):=
{\rm Dom}(({\rm L^{\ast}})^{\infty})=\bigcap_{k=1}^{\infty}{\rm
Dom}(({\rm L^{\ast}})^{k}),
$$
where ${\rm Dom}(({\rm L^{\ast}})^{k})$ is the domain of the
operator $({\rm L^{\ast}})^{k}$,
$$
{\rm Dom}(({\rm L^{\ast}})^{k}):=\{f\in L^{2}(\Omega): \,\,\, ({\rm
L^{\ast}})^{j}f\in {\rm Dom}({\rm L^{\ast}}), \,\,\, j=0, \ldots, k-1\},
$$
which satisfy the adjoint boundary conditions corresponding to the operator
${\rm L}_\Omega^*$. The Fr\'echet topology of $C_{{\rm L}^*}^{\infty}(\overline{\Omega})$ is given by the family of norms
\begin{equation}\label{EQ:L-top-adj}
\|\psi\|_{C^{k}_{{\rm L}^*}}:=\max_{j\leq k}
\|({\rm L}^*)^{j}\psi\|_{L^2(\Omega)}, \quad k\in\mathbb N_0,
\; \psi\in C_{{\rm L}^*}^{\infty}(\overline{\Omega}).
\end{equation}

Since we have $u_\xi\in C^\infty_{{\rm L}}(\overline\Omega)$ and
$v_\xi\in C^\infty_{{\rm L}^*}(\overline\Omega)$ for all $\xi\in\ind$, we observe that
Assumption \ref{Assumption_3} implies that the spaces
$C^\infty_{{\rm L}}(\overline\Omega)$ and $C^\infty_{{\rm L}^*}(\overline\Omega)$
are dense in $L^2(\Omega)$.
\end{defn}

We note that if ${\rm L}_\Omega$ is self-adjoint, i.e. if ${\rm L}_\Omega^*={\rm L}_\Omega$
with the equality of domains, then
$C_{{\rm L^{\ast}}}^{\infty}(\overline{\Omega})=C_{{\rm L}}^{\infty}(\overline{\Omega}).$

In general, for functions $f\in C_{{\rm L}}^{\infty}(\overline{\Omega})$ and
$g\in C_{{\rm L}^*}^{\infty}(\overline{\Omega})$, the $L^2$-duality makes sense in view
of the formula
\begin{equation}\label{EQ:duality}
({\rm L}f, g)_{L^2(\Omega)}=(f,{\rm L}^*g)_{L^2(\Omega)}.
\end{equation}
Therefore, in view of the formula \eqref{EQ:duality},
it makes sense to define the distributions $\mathcal D'_{{\rm L}}(\Omega)$
as the space which is dual to $C_{{\rm L}^*}^{\infty}(\overline{\Omega})$.
Note that the the respective boundary conditions of ${\rm L}_\Omega$ and ${\rm L}_\Omega^*$
are satisfied by the choice of $f$ and $g$ in corresponding domains.

\begin{defn}\label{DistrSp}
The space $$\mathcal D'_{{\rm
L}}(\Omega):=\mathcal L(C_{{\rm L}^*}^{\infty}(\overline{\Omega}),
\mathbb C)$$ of linear continuous functionals on
$C_{{\rm L}^*}^{\infty}(\overline{\Omega})$ is called the space of
${\rm L}$-distributions.
We can understand the continuity here either in terms of the topology
\eqref{EQ:L-top-adj} or in terms of sequences, see
Proposition \ref{TH: UniBdd}.
For
$w\in\mathcal D'_{{\rm L}}(\Omega)$ and $\varphi\in C_{{\rm L}^*}^{\infty}(\overline{\Omega})$,
we shall write
$$
w(\varphi)=\langle w, \varphi\rangle.
$$
For any $\psi\in C_{{\rm L}}^{\infty}(\overline{\Omega})$,
$$
C_{{\rm L}^*}^{\infty}(\overline{\Omega})\ni \varphi\mapsto\int_{\Omega}{\psi(x)} \, \varphi(x)\, dx
$$
is an ${\rm L}$-distribution, which gives an embedding $\psi\in
C_{{\rm L}}^{\infty}(\overline{\Omega})\hookrightarrow\mathcal D'_{{\rm
L}}(\Omega)$.
We note that in the distributional notation formula \eqref{EQ:duality} becomes
\begin{equation}\label{EQ:duality-dist}
\langle{\rm L}\psi, \varphi\rangle=\langle \psi,\overline{{\rm L}^* \overline{\varphi}}\rangle.
\end{equation}
\end{defn}

With the topology on $C_{{\rm L}}^{\infty}(\overline{\Omega})$
defined by \eqref{EQ:L-top},
the space $$\mathcal
D'_{{\rm L^{\ast}}}(\Omega):=\mathcal L(C_{{\rm L}}^{\infty}(\overline{\Omega}), \mathbb C)$$
of linear continuous functionals on $C_{{\rm L}}^{\infty}(\overline{\Omega})$
is called the
space of ${\rm L^{\ast}}$-distributions.

\begin{prop}\label{TH: UniBdd}
A linear functional $w$ on
$C_{{\rm L}^*}^{\infty}(\overline{\Omega})$ belongs to $\mathcal D'_{{\rm
L}}(\Omega)$ if and only if there exists a constant $c>0$ and a
number $k\in\mathbb N_0$ with the property
\begin{equation}
\label{EQ: UnifBdd-s1} |w(\varphi)|\leq
c \|\varphi\|_{C^{k}_{{\rm L}^*}} \quad \textrm{ for all } \; \varphi\in C_{{\rm
L}^*}^{\infty}(\overline{\Omega}).
\end{equation}
\end{prop}

\begin{proof}
$\Leftarrow$. If $w$ satisfies (\ref{EQ: UnifBdd-s1}), it then
follows that $w(\varphi_{j})-w(\varphi)=w(\varphi_{j}-\varphi)$
converges to 0 as $j\rightarrow\infty.$

$\Rightarrow$. Now suppose that $w$ does not satisfy condition
(\ref{EQ: UnifBdd-s1}). This means that for every $c>0$ and every $k\in\mathbb N_0$,
there is a $\varphi_{c, k}\in C_{{\rm
L}^*}^{\infty}(\overline{\Omega})$ for which $$|w(\varphi_{c, k})|>c
\|\varphi_{c, k}\|_{C^{k}_{{\rm L}^*}}.$$ This implies
$$
\|\psi_{c, k}\|_{C^{k}_{{\rm L}^*}}<\frac{1}{c} \,\,\,\,\,
\hbox{and} \,\,\,\,\, |w(\psi_{c, k})|=1,
$$
if we take $\psi_{c, k}=\lambda\varphi_{c, k}$ and
$\lambda=\frac{1}{|w(\varphi_{c, k})|}$. The sequence $\{\psi_{k,
k}\}_{k\in\mathbb N}$ converges to zero in $C_{{\rm
L}^*}^{\infty}(\overline{\Omega})$, while $w(\psi_{k, k})$ does not
converge to zero. Therefore, $w$ is not a distribution, which gives a
contradiction.
\end{proof}

The space $\mathcal D'_{{\rm L}}(\Omega)$ has many similarities with the
usual spaces of distributions. For example, suppose that for a linear continuous operator
$D:C_{{\rm L}}^{\infty}(\overline{\Omega})\to C_{{\rm L}}^{\infty}(\overline{\Omega})$
its adjoint $D^*$
preserves the adjoint boundary conditions (domain) of ${\rm L}_\Omega^*$
and is continuous on the space
$C_{{\rm L}^*}^{\infty}(\overline{\Omega})$, i.e.
that the operator
$D^*:C_{{\rm L}^*}^{\infty}(\overline{\Omega})\to C_{{\rm L}^*}^{\infty}(\overline{\Omega})$
is continuous.
Then we can extend $D$ to $\mathcal D'_{{\rm L}}(\Omega)$ by
$$
\langle Dw,{\varphi}\rangle := \langle w, \overline{D^* \overline{\varphi}}\rangle \quad
(w\in \mathcal D'_{{\rm L}}(\Omega),\;  \varphi\in C_{{\rm
L}^*}^{\infty}(\overline{\Omega})).
$$
This extends \eqref{EQ:duality-dist} from L to other operators.
The convergence in the linear space
$\mathcal D'_{{\rm L}}(\Omega)$ is the usual weak convergence with respect to
the space $C_{{\rm L}^*}^{\infty}(\overline{\Omega})$.
The following principle of uniform boundedness is based on the
Banach--Steinhaus Theorem applied to the Fr\'echet space $C_{{\rm
L}^*}^{\infty}(\overline{\Omega})$.

\begin{lemma} \label{LEM: UniformBoundedness}
Let $\{w_{j}\}_{j\in\mathbb N}$ be a sequence in $\mathcal
D'_{{\rm L}}(\Omega)$ with the property that for every $\varphi\in
C_{{\rm L}^*}^{\infty}(\overline{\Omega})$, the sequence
$\{w_{j}(\varphi)\}_{j\in\mathbb N}$ in $\mathbb C$ is bounded.
Then there exist constants $c>0$ and $k\in\mathbb N_0$ such that
\begin{equation}
\label{EQ: UniformBoundedness} |w_{j}(\varphi)|\leq c
\|\varphi\|_{C^{k}_{{\rm L}^*}} \quad \textrm{ for all } \; j\in\mathbb N, \,\,
\varphi\in C_{{\rm L}^*}^{\infty}(\overline{\Omega}).
\end{equation}
\end{lemma}

The lemma above leads to the following property of completeness of
the space of ${\rm L}$-distributions. 

\begin{theorem} \label{TH: Com-nessDistr}
Let $\{w_{j}\}_{j\in\mathbb N}$ be a
sequence in $\mathcal D'_{{\rm L}}(\Omega)$ with the property that
for every $\varphi\in C_{{\rm L}^*}^{\infty}(\overline{\Omega})$ the
sequence $\{w_{j}(\varphi)\}_{j\in\mathbb N}$ converges in
$\mathbb C$ as $j\rightarrow\infty$. Denote the limit by
$w(\varphi)$.

{\rm (i)} Then $w:\varphi\mapsto w(\varphi)$ defines an ${\rm
L}$-distribution on $\Omega$. Furthermore,
$$
\lim_{j\rightarrow\infty}w_{j}=w \,\,\,\,\,\,\,\, \hbox{in}
\,\,\,\,\,\,\, \mathcal D'_{{\rm L}}(\Omega).
$$

{\rm (ii)} If $\varphi_{j}\rightarrow\varphi$ in $\in C_{{\rm
L}^*}^{\infty}(\overline{\Omega})$, then
$$
\lim_{j\rightarrow\infty}w_{j}(\varphi_{j})=w(\varphi) \,\,\,\,\,\,\,\,
\hbox{in} \,\,\,\,\,\,\, \mathbb C.
$$
\end{theorem}
\begin{proof} 
(i) Writing out the definitions, we find that $w$
defines a linear functional on $C_{{\rm
L}^*}^{\infty}(\overline{\Omega})$. From the starting assumption it
follows that the sequence $\{w_{j}(\varphi)\}_{j\in\mathbb N}$ is bounded
for every $\varphi\in C_{{\rm L}^*}^{\infty}(\overline{\Omega})$,
and thus we obtain an estimate of the form (\ref{EQ:
UniformBoundedness}). Taking the limit in
$$
|w(\varphi)|\leq |w(\varphi)-w_{j}(\varphi)|+|w_{j}(\varphi)|\leq
|w(\varphi)-w_{j}(\varphi)|+c \|\varphi\|_{C^{k}_{{\rm L}^*}}
$$
as $j\rightarrow\infty$, we get $$|w(\varphi)|\leq c
\|\varphi\|_{C^{k}_{{\rm L}^*}}$$ for all $\varphi\in C_{{\rm
L}^*}^{\infty}(\overline{\Omega})$. According to Proposition \ref{TH:
UniBdd} this proves that $w\in\mathcal D'_{{\rm L}}(\Omega)$, and
$w_{j}\rightarrow w$ in $\mathcal D'_{{\rm L}}(\Omega)$ now holds
by definition.

(ii) Regarding the last assertion we observe that if
$\varphi_{j}\rightarrow\varphi$ in $C_{{\rm
L}^*}^{\infty}(\overline{\Omega})$, then by applying Lemma \ref{LEM:
UniformBoundedness} once again, we obtain
$$
|w_{j}(\varphi_{j})-w(\varphi)|\leq
|w_{j}(\varphi_{j}-\varphi)|+|w_{j}(\varphi)-w(\varphi)|\leq c
\|\varphi_{j}-\varphi\|_{C^{k}_{{\rm
L}^*}}+|w_{j}(\varphi)-w(\varphi)|,
$$
which converges to zero as $j\rightarrow\infty$.
\end{proof}

The main tool in the proof of Theorem \ref{TH: Com-nessDistr} was
Lemma \ref{LEM: UniformBoundedness}, which is based on the
principle of uniform boundedness. It may be instructive to give
another proof of Part (i) of Theorem \ref{TH: Com-nessDistr}
based on the method of the gliding hump.
\begin{proof}
Suppose that $w$ does not belongs to $\mathcal D'_{{\rm
L}}(\Omega)$. Then there exists a sequence
$\{\varphi_{j}\}_{j\in\mathbb N}$ in $C_{{\rm
L}^*}^{\infty}(\overline{\Omega})$ that converges to zero in
$C_{{\rm L}^*}^{\infty}(\overline{\Omega})$, while
$\{w(\varphi_{j})\}_{j\in\mathbb N}$ does not converge to zero as
$j\rightarrow\infty$. Hence, by passing to a subsequence if
necessary, we can arrange that there exists $c>0$ such that
$|w(\varphi_{j})|\geq c$. We can assume that
$\|\varphi_{j}\|_{C^{j}_{{\rm L}^*}}\leq\frac{1}{4^{j}}$ if we
replace $\{\varphi_{j}\}_{j\in\mathbb N}$ by a suitable
subsequence if necessary. Accordingly, upon writing $\varphi_{j}$
for $2^{j}\varphi_{j}$, we obtain that $\varphi_{j}\rightarrow0$
in $C_{{\rm L}^*}^{\infty}(\overline{\Omega})$, while
$|w(\varphi_{j})|\rightarrow\infty$ as $j\rightarrow\infty$.

Next, we define a subsequence of $\{\varphi_{j}\}_{j\in\mathbb
N}$, say $\{\psi_{j}\}_{j\in\mathbb N}$ in $C_{{\rm
L}^*}^{\infty}(\overline{\Omega})$, and a subsequence of
$\{w_{j}\}_{j\in\mathbb N}$, say $\{v_{j}\}_{j\in\mathbb N}$ in
$\mathcal D'_{{\rm L}}(\Omega)$, as follows. Select $\psi_{1}$
such that $|w(\psi_{1})|>2$. As $w_{j}(\psi_{1})\rightarrow
w(\psi_{1})$, we may choose $v_{1}$ such that
$|v_{1}(\psi_{1})|>2$. Now proceed by induction on
$j$. Thus, assume that $\psi_{k}$ and $v_{k}$ have been chosen,
for $1\leq k<j$. Then select $\psi_{j}$ from the sequence
$\{\varphi_{j}\}_{j\in\mathbb N}$ such that
\begin{equation}
\label{EQ: Estimates} \begin{split} &{\rm (a)} \,\,\,\,
\|\psi_{j}\|_{C^{j}_{{\rm L}^*}}<\frac{1}{2^{j}}, \,\,\,\,\,
\,\,\,\,\,\,\,\,\,\,\,\,\,\,\,  {}\\
&{\rm (b)} \,\,\,\, |v_{k}(\psi_{j})|<\frac{1}{2^{j-k}}, \,\,\,\,\,
(1\leq
k<j),{}\\
&{\rm (c)} \,\,\,\, |w(\psi_{j})|>\sum\limits_{1\leq
k<j}|w(\psi_{k})|+j+1.
\end{split}
\end{equation}
Condition (a) can be satisfied because of the properties of the
$\varphi_{i}$; and (b) because of $\varphi_{j}\rightarrow0$ in
$C_{{\rm L}^*}^{\infty}(\overline{\Omega})$ and all $v_{k}$ belong
to $\mathcal D'_{{\rm L}}(\Omega)$, for $1\leq k<j$; whereas (c)
holds because $|w(\varphi_{j})|\rightarrow\infty$ as
$j\rightarrow\infty$. In addition, since
$\lim\limits_{j\rightarrow\infty}w_{j}(\psi)=w(\psi)$, for all
$\psi\in C_{{\rm L}^*}^{\infty}(\overline{\Omega})$, condition (c)
implies that we may select $v_{j}$ from the sequence
$\{w_{j}\}_{j\in\mathbb N}$ such that
\begin{equation}
\label{EQ: Estimates2} |v_{j}(\psi_{j})|>\sum\limits_{1\leq
k<j}|v_{j}(\psi_{k})|+j+1.
\end{equation}

Now, set $\psi:=\sum\limits_{k\in\mathbb N}\psi_{k}$. According to
(a) the series on the right-hand side
converges in $C_{{\rm L}^*}^{\infty}(\overline{\Omega})$, which
leads to $\psi\in C_{{\rm L}^*}^{\infty}(\overline{\Omega})$.
Obviously, for any $j$,
$$
v_{j}(\psi)=\sum\limits_{1\leq
k<j}v_{j}(\psi_{k})+v_{j}(\psi_{j})+\sum\limits_{j<k}v_{j}(\psi_{k}),
$$
hence
$$
|v_{j}(\psi)|\geq |v_{j}(\psi_{j})|-\sum\limits_{1\leq
k<j}|v_{j}(\psi_{k})|-\sum\limits_{j<k}|v_{j}(\psi_{k})|>j+1-1=j,
$$
on account of (\ref{EQ: Estimates2}) and (b). On the other hand, $\{v_{j}\}_{j\in\mathbb N}$
being a subsequence of $\{w_{j}\}_{j\in\mathbb N}$ implies
$\lim\limits_{j\rightarrow\infty}v_{j}(\psi)=w(\psi)$.
Summarising these properties, we have arrived at a contradiction.
\end{proof}

Similarly to the previous case, we have analogues of
Proposition \ref{TH: UniBdd} and Theorem \ref{TH: Com-nessDistr}
for ${\rm L^{\ast}}$-distributions.

\section{${\rm L}$-Fourier transform}
\label{SEC:FT}

In this section we define the ${\rm L}$-Fourier transform generated by
our boundary value problem ${\rm L}_\Omega$ and its main properties.
The main difference between the self-adjoint and non-self-adjoint
problems ${\rm L}_\Omega$
is that in the latter case we have to make sure that we use the right
functions from the available biorthogonal families of $u_\xi$ and $v_\xi$.
We start by defining the spaces that we will obtain on the Fourier transform side.

From now on, we will assume that the boundary conditions are closed under taking limits in the
strong uniform topology to ensure that the strongly convergent series preserve the boundary conditions.
More precisely, from now on:

\MyQuote{assume that, with ${\rm L}_0$ denoting ${\rm L}$ or ${\rm L}^*$, if
$f_j\in C_{{\rm L}_0}^{\infty}(\overline{\Omega})$ satisfies $f_j\to f$ in $C_{{\rm L}_0}^{\infty}(\overline{\Omega})$,
then $f\in C_{{\rm L}_0}^{\infty}(\overline{\Omega})$.}{BC+}


\medskip
Let $\mathcal S(\ind)$ denote the space of rapidly decaying
functions $\varphi:\ind\rightarrow\mathbb C$. That is,
$\varphi\in\mathcal S(\ind)$ if for any $M<\infty$ there
exists a constant $C_{\varphi, M}$ such that
$$
|\varphi(\xi)|\leq C_{\varphi, M}\langle\xi\rangle^{-M}
$$
holds for all $\xi\in\ind$.
Here $\langle\xi\rangle$ is already adapted to our boundary value problem
since it is defined by \eqref{EQ:angle}.

The topology on $\mathcal
S(\ind)$ is given by the seminorms $p_{k}$, where
$k\in\mathbb N_{0}$ and $$p_{k}(\varphi):=\sup_{\xi\in\ind}\langle\xi\rangle^{k}|\varphi(\xi)|.$$
Continuous linear functionals on $\mathcal S(\ind)$ are of
the form
$$
\varphi\mapsto\langle u, \varphi\rangle:=\sum_{\xi\in\ind}u(\xi)\varphi(\xi),
$$
where functions $u:\ind \rightarrow \mathbb C$ grow at most
polynomially at infinity, i.e. there exist constants $M<\infty$
and $C_{u, M}$ such that
$$
|u(\xi)|\leq C_{u, M}\langle\xi\rangle^{M}
$$
holds for all $\xi\in\ind$. Such distributions $u:\ind
\rightarrow \mathbb C$ form the space of distributions which we denote by
$\mathcal S'(\ind)$.
We now define the $L$-Fourier transform on $C_{{\rm L}}^{\infty}(\overline{\Omega})$.

\begin{defn} \label{FT}
We define the ${\rm L}$-Fourier transform
$$
(\mathcal F_{{\rm L}}f)(\xi)=(f\mapsto\widehat{f}):
C_{{\rm L}}^{\infty}(\overline{\Omega})\rightarrow\mathcal S(\ind)
$$
by
\begin{equation}
\label{FourierTr}
\widehat{f}(\xi):=(\mathcal F_{{\rm L}}f)(\xi)=\int_{\Omega}f(x)\overline{v_{\xi}(x)}dx.
\end{equation}
Analogously, we define the ${\rm L}^{\ast}$-Fourier
transform
$$
(\mathcal F_{{\rm L}^{\ast}}f)(\xi)=(f\mapsto\widehat{f}_{\ast}):
C_{{\rm L}^{\ast}}^{\infty}(\overline{\Omega})\rightarrow\mathcal
S(\ind)
$$
by
\begin{equation}\label{ConjFourierTr}
\widehat{f}_{\ast}(\xi):=(\mathcal F_{{\rm
L}^{\ast}}f)(\xi)=\int_{\Omega}f(x)\overline{u_{\xi}(x)}dx.
\end{equation}
\end{defn}

The expressions \eqref{FourierTr} and \eqref{ConjFourierTr}
are well-defined by the Cauchy-Schwarz inequality, for example,
\begin{equation}
\label{EQ: Ineq1}
|\widehat{f}(\xi)|=\left|\int_{\Omega}f(x)\overline{v_{\xi}(x)}dx\right|\leq\|f\|_{L^{2}}
\|v_{\xi}\|_{L^{2}}=\|f\|_{L^{2}}<\infty.
\end{equation}
Moreover, we have

\begin{prop}\label{LEM: FTinS}
The ${\rm L}$-Fourier transform
$\mathcal F_{{\rm L}}$ is a bijective homeomorphism from $C_{{\rm
L}}^{\infty}(\overline{\Omega})$ to $\mathcal S(\ind)$.
Its inverse  $$\mathcal F_{{\rm L}}^{-1}: \mathcal S(\ind)
\rightarrow C_{{\rm L}}^{\infty}(\overline{\Omega})$$ is given by
\begin{equation}
\label{InvFourierTr} (\mathcal F^{-1}_{{\rm
L}}h)(x)=\sum_{\xi\in\ind}h(\xi)u_{\xi}(x),\quad h\in\mathcal S(\ind),
\end{equation}
so that the Fourier inversion formula becomes
\begin{equation}
\label{InvFourierTr0}
f(x)=\sum_{\xi\in\ind}\widehat{f}(\xi)u_{\xi}(x)
\quad \textrm{ for all } f\in C_{{\rm
L}}^{\infty}(\overline{\Omega}).
\end{equation}
Similarly,  $\mathcal F_{{\rm L}^{\ast}}:C_{{\rm L}^{\ast}}^{\infty}(\overline{\Omega})\to \mathcal S(\ind)$
is a bijective homeomorphism and its inverse
$$\mathcal F_{{\rm L}^{\ast}}^{-1}: \mathcal S(\ind)\rightarrow
C_{{\rm L}^{\ast}}^{\infty}(\overline{\Omega})$$ is given by
\begin{equation}
\label{ConjInvFourierTr} (\mathcal F^{-1}_{{\rm
L}^{\ast}}h)(x):=\sum_{\xi\in\ind}h(\xi)v_{\xi}(x), \quad h\in\mathcal S(\ind),
\end{equation}
so that the conjugate Fourier inversion formula becomes
\begin{equation}
\label{ConjInvFourierTr0} f(x)=\sum_{\xi\in\ind}\widehat{f}_{\ast}(\xi)v_{\xi}(x)\quad \textrm{ for all } f\in C_{{\rm
L^*}}^{\infty}(\overline{\Omega}).
\end{equation}
\end{prop}

\begin{proof}[Proof of Proposition \ref{LEM: FTinS}]
The proof is largely similar to the standard case, so we only indicate a few key points
due to biorthogonality.
We show first that
for any $f\in C_{{\rm
L}}^{\infty}(\overline{\Omega})$ we have $\widehat{f}\in\mathcal
S(\ind)$, i.e. that
for any $M<\infty$
there exists a constant $C$ such that
$$
|\widehat{f}(\xi)|\leq C\langle\xi\rangle^{-M}
$$
holds for all $\xi\in\ind$.
Indeed, for any $M\in\mathbb N$ and
$\lambda_\xi\not=0$ we get
\begin{multline*}
|\widehat{f}(\xi)|=\left|\int_{\Omega}f(x)\overline{v_{\xi}(x)}dx\right|=\left|\int_{\Omega}f(x)\frac{\overline{({\rm
L^{\ast}})^{M}v_{\xi}(x)}}{\lambda^{M}_{\xi}}dx\right|
\\
=\left|\frac{1}{\lambda^{M}_{\xi}}\int_{\Omega}
{\rm L}^{M}f(x)\overline{v_{\xi}(x)}dx\right|\leq C\|{\rm L}^{M}f\|_{L^2(\Omega)}\langle\xi\rangle^{-mM}
\end{multline*}
by the Cauchy-Schwarz inequality.
In view of \eqref{EQ:L-top}, this also shows that $\mathcal F_{{\rm L}}$ is continuous from
$C_{{\rm L}}^{\infty}(\overline{\Omega})$ to $\mathcal S(\ind)$.

Now, in view of (BC+), for any $h\in\mathcal S(\ind)$ the formula \eqref{InvFourierTr}
defines a function $\mathcal F^{-1}_{{\rm L}}h\in C_{{\rm L}}^{\infty}(\overline{\Omega})$
with Fourier coefficients $h(\xi)$ due to biorthogonality relations \eqref{BiorthProp}.
If two function $f_{1}, \, f_{2}\in C_{{\rm
L}}^{\infty}(\overline{\Omega})$
have the same Fourier coefficients
$\widehat{f_{1}}(\xi)=\widehat{f_{2}}(\xi)$
for all $\xi\in\ind$, since the linear span
$\{u_{\xi}\}_{\xi\in\ind}$ is dense in $C_{{\rm
L}}^{\infty}(\overline{\Omega})$, we have
$$
f_{1}(x)=\sum_{\xi\in\ind}\widehat{f_{1}}(\xi)u_{\xi}(x)=\sum_{\xi\in\ind}\widehat{f_{2}}(\xi)u_{\xi}(x)=f_{2}(x).
$$
The continuity of $\mathcal F_{{\rm L}}^{-1}: \mathcal S(\ind)
\rightarrow C_{{\rm L}}^{\infty}(\overline{\Omega})$ readily follows as well.
The properties of the conjugate Fourier transform
$\mathcal F_{{\rm L}^{\ast}}$ can be seen in an analogous way.
\end{proof}

By dualising the inverse ${\rm L}$-Fourier
transform $\mathcal F_{{\rm L}}^{-1}: \mathcal S(\ind)
\rightarrow C_{{\rm L}}^{\infty}(\overline{\Omega})$, the
${\rm L}$-Fourier transform extends uniquely to the mapping
$$\mathcal F_{{\rm L}}: \mathcal D'_{{\rm L}}(\Omega)\rightarrow
\mathcal S'(\ind)$$ by the formula
\begin{equation}\label{EQ: FTofDistr}
\langle\mathcal F_{{\rm L}}w, \varphi\rangle:=
\langle w,\overline{\mathcal F_{{\rm L}^*}^{-1}\overline{\varphi}}\rangle,
\quad\textrm{ with } w\in\mathcal D'_{{\rm L}}(\Omega),\; \varphi\in\mathcal
S(\ind).
\end{equation}
It can be readily seen that if $w\in\mathcal D'_{{\rm
L}}(\Omega)$ then $\widehat{w}\in\mathcal S'(\ind)$.
The reason for taking complex conjugates in \eqref{EQ: FTofDistr}
is that, if $w\in C_{{\rm L}}^{\infty}(\overline{\Omega})$, we have
the equality
\begin{multline*}
\langle \widehat{w},\varphi\rangle =
\sum_{\xi\in\ind} \widehat{w}(\xi) \varphi(\xi)=
\sum_{\xi\in\ind} \left( \int_\Omega w(x) \overline{v_\xi(x)}dx\right) \varphi(\xi)\\
=
\int_\Omega w(x) \overline{\left( \sum_{\xi\in\ind} \overline{\varphi(\xi)} v_\xi(x)\right)} dx
=
\int_\Omega w(x) \overline{\left( \mathcal F_{{\rm L}^*}^{-1} \overline{\varphi} \right)} dx
=\langle w,\overline{\mathcal F_{{\rm L}^*}^{-1}\overline{\varphi}}\rangle.
\end{multline*}
Analogously, we have the mapping
$$\mathcal F_{{\rm L}^*}: \mathcal D'_{{\rm L}^*}(\Omega)\rightarrow
\mathcal S'(\ind)$$ defined by the formula
\begin{equation}\label{EQ: FTofDistr}
\langle\mathcal F_{{\rm L}^*}w, \varphi\rangle:=
\langle w,\overline{\mathcal F_{{\rm L}}^{-1}\overline{\varphi}}\rangle,
\quad\textrm{ with } w\in\mathcal D'_{{\rm L}^*}(\Omega),\; \varphi\in\mathcal
S(\ind).
\end{equation}
It can be also seen that if $w\in\mathcal D'_{{\rm
L}^*}(\Omega)$ then $\widehat{w}\in\mathcal S'(\ind)$.

We note that since systems of  $u_\xi$ and of $v_\xi$ are Riesz bases,
we can also compare $L^2$-norms of functions with sums of squares of Fourier coefficients.
The following statement follows from the work of Bari \cite[Theorem 9]{bari}:

\begin{lemma}\label{LEM: FTl2}
There exist constants $k,K,m,M>0$ such that for every $f\in L^{2}(\Omega)$
we have
$$
m^2\|f\|_{L^{2}}^2 \leq \sum_{\xi\in\ind} |\widehat{f}(\xi)|^2\leq M^2\|f\|_{L^{2}}^2
$$
and
$$
k^2\|f\|_{L^{2}}^2 \leq \sum_{\xi\in\ind} |\widehat{f}_*(\xi)|^2\leq K^2\|f\|_{L^{2}}^2.
$$
\end{lemma}

However, we note that the Plancherel identity can be also achieved in suitably
defined $l^2$-spaces of Fourier coefficients, see Proposition \ref{PlanchId}.

\section{${\rm L}$-Convolution}
\label{SEC:conv}

Let us introduce a notion of the  ${\rm L}$-convolution, an analogue of the convolution adapted to the
boundary problem ${\rm L}_\Omega$.

\begin{defn} (${\rm L}$-Convolution) \label{Convolution}
For $f, g\in C_{{\rm L}}^{\infty}(\overline{\Omega})$ define their ${\rm L}$-convolution
by
\begin{equation}\label{EQ: CONV1}
(f\sL g)(x):=\sum_{\xi\in\ind}\widehat{f}(\xi)\widehat{g}(\xi)u_{\xi}(x).
\end{equation}
By Proposition \ref{LEM: FTinS}  it is well-defined and we have
$f\sL g\in C_{{\rm L}}^{\infty}(\overline{\Omega}).$

Moreover, due to the rapid decay of L-Fourier coefficients of
functions in $C_{{\rm L}}^{\infty}(\overline{\Omega})$ compared to
a fixed polynomial growth of elements of $\mathcal S'(\ind)$, the
definition \eqref{EQ: CONV1}  still makes sense if $f\in \mathcal
D^\prime_{\rm L}(\Omega)$ and $g\in C_{{\rm
L}}^{\infty}(\overline{\Omega})$, with $f\sL g\in C_{{\rm
L}}^{\infty}(\overline{\Omega}).$
\end{defn}

Analogously to the ${\rm L}$-convolution, we can introduce the ${\rm L}^*$-convolution.
Thus, for $f, g\in C_{{\rm
L^{\ast}}}^{\infty}(\overline{\Omega})$, we define the ${\rm
L^{\ast}}$-convolution using the ${\rm L}^*$-Fourier transform by
\begin{equation}\label{EQ: CONV2}
(f\sLs g)(x):=\sum_{\xi\in\ind}\widehat{f}_{\ast}(\xi)\widehat{g}_{\ast}(\xi)v_{\xi}(x).
\end{equation}
Its properties are similar to those of the ${\rm L}$-convolution, so we may
formulate and prove only the latter.

\begin{rem}
Informally, expanding the definitions of the Fourier transforms in \eqref{EQ: CONV1},
we can also write
\begin{equation}
\label{CONV} (f\sL
g)(x):=\int_{\Omega}\int_{\Omega}F(x,y,z)f(y)g(z)dydz,
\end{equation}
where
$$
F(x,y,z)=\sum_{\xi\in\ind}u_{\xi}(x) \ \overline{v_{\xi}(y)}
\ \overline{v_{\xi}(z)}.
$$
The latter series should be understood in the sense of distributions.

\medskip
In the case of operator ${\rm L}={\rm O}_{1}^{(1)}$ generated by the operator of differentiation with
periodic boundary condition on the interval $(0,1)$,
see the case $h=1$ in Example \ref{Example1}
as in \cite{Ruzhansky-Turunen-JFAA-torus}, we have
$$
F(x,y,z)=\delta(x-y-z).
$$
For any $h>0$, it can be shown that the convolution generated by the operator ${\rm O}_{h}^{(1)}$
from Example \ref{Example1} has
also the following integral form:
$$
(f\star_{{\rm O}_{h}^{(1)}}
g)(x)=\int^{x}_{0}f(x-t)g(t)dt+\frac{1}{h}\int^{1}_{x}f(1+x-t)g(t)dt,
$$
see \cite{Kanguzhin_Tokmagambetov} and \cite{Kanguzhin_Tokmagambetov_Tulenov}.
\end{rem}

\begin{prop}\label{ConvProp} For any $f, g\in C_{{\rm
L}}^{\infty}(\overline{\Omega})$ we have
$$\widehat{f\sL g}=\widehat{f}\,\widehat{g}.$$
The convolution is commutative and associative.
If $g \in C_{{\rm
L}}^{\infty}(\overline{\Omega}),$ then for all
$f\in \mathcal D^\prime_{\rm L}(\Omega)$ we have
\begin{equation}\label{EQ:conv1}
f\sL g\in C_{{\rm L}}^{\infty}(\overline{\Omega}).
\end{equation}
If $f,g\in  L^{2}(\Omega)$, then $f\sL g\in L^{1}(\Omega)$ with
$$\|f\sL g\|_{L^1}\leq C|\Omega|^{1/2} \|f\|_{L^2}\|g\|_{L^2},$$
where $|\Omega|$ is the volume of $\Omega$, with $C$ independent
of $f,g,\Omega$.
\end{prop}

\begin{proof}
By direct calculation, we get
\begin{align*}
\mathcal F_{{\rm L}}(f\sL
g)(\xi)&=\int_{\Omega}\sum_{\eta\in\ind}
\widehat{f}(\eta)\widehat{g}(\eta)u_{\eta}(x)\overline{v_{\xi}(x)}dx
\\
&=\sum_{\eta\in\ind}
\widehat{f}(\eta)\widehat{g}(\eta)\int_{\Omega}u_{\eta}(x)\overline{v_{\xi}(x)}dx
\\
&=\widehat{f}(\xi)\widehat{g}(\xi).
\end{align*}
This also implies the commutativity of the convolution in view of the bijectivity of the Fourier transform.
For the associativity, let $f, g, h\in C_{{\rm
L}}^{\infty}(\overline{\Omega})$. We can argue similarly using the Fourier transform or,
by the definition and direct calculations, we have
\begin{align*}
((f\sL g) \sL h)(x) & = \sum_{\xi\in\ind}\left[\int_{\Omega}\left(\sum_{\eta\in\ind}
\widehat{f}(\eta)\widehat{g}(\eta)u_{\eta}(y)\right)\overline{v_{\xi}(y)}dy\right]\widehat{h}(\xi)u_{\xi}(x)
\\
&=\sum_{\xi\in\ind}\left[\sum_{\eta\in\ind}
\widehat{f}(\eta)\widehat{g}(\eta)\int_{\Omega}u_{\eta}(y)\overline{v_{\xi}(y)}dy\right]\widehat{h}(\xi)u_{\xi}(x)
\\
&=\sum_{\xi\in\ind}\widehat{f}(\xi)\widehat{g}(\xi)\widehat{h}(\xi)u_{\xi}(x)
\\
&=\sum_{\xi\in\ind}\widehat{f}(\xi)\left[\sum_{\eta\in\ind}\widehat{g}(\eta)\widehat{h}(\eta)\int_{\Omega}u_{\eta}(y)\overline{v_{\xi}(y)}dy\right]u_{\xi}(x)
\\
&=\sum_{\xi\in\ind}\widehat{f}(\xi)\left[\int_{\Omega}\left(\sum_{\eta\in\ind}
\widehat{g}(\eta)\widehat{h}(\eta)u_{\eta}(y)\right)\overline{v_{\xi}(y)}dy\right]u_{\xi}(x)
\\
&=(f\sL (g \sL h))(x).
\end{align*}
The associativity is proved.
For \eqref{EQ:conv1}, we notice that
$$
{\rm L}^k (f\sL g)(x)=\sum_{\xi\in\ind}\widehat{f}(\xi)\widehat{g}(\xi) \lambda_\xi^k u_{\xi}(x),
$$
and the series converges absolutely since $\widehat{g}\in\mathcal
S(\ind).$ By (BC+), the boundary conditions are also satisfied since they
are satisfied by $u_\xi$. This shows that $f\sL g\in C_{{\rm
L}}^{\infty}(\overline{\Omega}).$

For the last statement, by simple calculations we get
\begin{align*}
\int_{\Omega}|(f\sL
g)(x)|dx&\leq\int_{\Omega}\sum\limits_{\xi\in\ind}|\widehat{f}(\xi)\widehat{g}(\xi)|\,|u_{\xi}(x)|dx\\
&\leq\sum\limits_{\xi\in\ind}|\widehat{f}(\xi)|\,|\widehat{g}(\xi)|\, \|u_{\xi}\|_{L^{1}}\\
&\leq C\|f\|_{L^{2}}\,\|g\|_{L^{2}}\sup\limits_{\xi\in\ind}\|u_{\xi}\|_{L^{1}},
\end{align*}
the latter estimate by Lemma \ref{LEM: FTl2}.
Since $\Omega$ is a bounded set, by the Cauchy-Schwarz inequality we have
$$
\|u_{\xi}\|_{L^{1}}\leq |\Omega|^{1/2} \|u_{\xi}\|_{L^{2}}=|\Omega|^{1/2}
$$
for all $\xi\in\ind$, where $|\Omega|$ is the volume of $\Omega$.
This inequality implies the statement.
\end{proof}

\section{Plancherel formula, Sobolev spaces $\mathcal H^{s}_{{\rm L}}(\Omega)$, and their Fourier images}
\label{SEC:Sobolev}

In this section we discuss Sobolev spaces adapted to ${\rm L}_\Omega$ and their images
under the L-Fourier transform. We start with the $L^2$-setting, where we can recall inequalities
between $L^2$-norms of functions and sums of squares of their Fourier coefficients, see
Lemma \ref{LEM: FTl2}. However, below we show that we actually have the Plancherel
identity in a suitably defined space $l^{2}_{{\rm L}}$ and its conjugate $l^{2}_{{\rm L}^{*}}$.

\medskip
Let us denote by $$l^{2}_{{\rm L}}=l^2({\rm L})$$ 
the linear space of complex-valued functions $a$
on $\ind$ such that $\mathcal F^{-1}_{{\rm L}}a\in
L^{2}(\Omega)$, i.e. if there exists $f\in L^{2}(\Omega)$ such that $\mathcal F_{{\rm L}}f=a$.
Then the space of sequences $l^{2}_{{\rm L}}$ is a
Hilbert space with the inner product
\begin{equation}\label{EQ: InnerProd SpSeq-s}
(a,\ b)_{l^{2}_{{\rm
L}}}:=\sum_{\xi\in\ind}a(\xi)\ \overline{(\mathcal F_{{\rm
L^{\ast}}}\circ\mathcal F^{-1}_{{\rm L}}b)(\xi)}
\end{equation}
for arbitrary $a,\,b\in l^{2}_{{\rm L}}$.
The reason for this choice of the definition is the following formal calculation:
\begin{align}\label{EQ:PL-prelim} \nonumber
(a,\ b)_{l^{2}_{{\rm L}}}&
=\sum_{\xi\in\ind}a(\xi)\ \overline{(\mathcal F_{{\rm L^{\ast}}}\circ\mathcal F^{-1}_{{\rm L}}b)(\xi)}\\ \nonumber
&=\sum\limits_{\xi\in\ind
}a(\xi)\int_{\Omega}\overline{(\mathcal F^{-1}_{{\rm L}}b)(x)}u_{\xi}(x)dx\\ \nonumber
&=\int_{\Omega}\left[\sum\limits_{\xi\in\ind}a(\xi)u_{\xi}(x)\right]\overline{(\mathcal F^{-1}_{{\rm
L}}b)(x)}dx\\ \nonumber
&=\int_{\Omega}(\mathcal F^{-1}_{{\rm L}}a)(x)\overline{(\mathcal F^{-1}_{{\rm L}}b)(x)} dx\\
&=(\mathcal F^{-1}_{{\rm L}}a,\,\mathcal F^{-1}_{{\rm L}}b)_{L^{2}},
\end{align}
which implies the Hilbert space properties of the space of sequences
$l^{2}_{{\rm L}}$. The norm of $l^{2}_{{\rm L}}$ is then given by the
formula
\begin{equation}\label{EQ:l2norm}
\|a\|_{l^{2}_{{\rm L}}}=\left(\sum_{\xi\in\ind}a(\xi)\
\overline{(\mathcal F_{{\rm L^{\ast}}}\circ\mathcal F^{-1}_{{\rm
L}}a)(\xi)}\right)^{1/2}, \quad \textrm{ for all } \; a\in l^{2}_{{\rm L}}.
\end{equation}
We note that individual terms in this sum may be complex-valued but the
whole sum is real and nonnegative due to formula \eqref{EQ:PL-prelim}.

Analogously, we introduce the
Hilbert space $$l^{2}_{{\rm L^{\ast}}}=l^{2}({\rm L^{\ast}})$$
as the space of functions $a$ on $\ind$
such that $\mathcal F^{-1}_{{\rm L^{\ast}}}a\in L^{2}(\Omega)$,
with the inner product
\begin{equation}
\label{EQ: InnerProd SpSeq-s_2} (a,\ b)_{l^{2}_{{\rm
L^{\ast}}}}:=\sum_{\xi\in\ind}a(\xi)\ \overline{(\mathcal
F_{{\rm L}}\circ\mathcal F^{-1}_{{\rm L^{\ast}}}b)(\xi)}
\end{equation}
for arbitrary $a,\,b\in l^{2}_{{\rm L^{\ast}}}$. The norm of
$l^{2}_{{\rm L^{\ast}}}$ is given by the formula
$$
\|a\|_{l^{2}_{{\rm L^{\ast}}}}=\left(\sum_{\xi\in\ind}a(\xi)\
\overline{(\mathcal F_{{\rm L}}\circ\mathcal F^{-1}_{{\rm
L^{\ast}}}a)(\xi)}\right)^{1/2}
$$
for all $a\in l^{2}_{{\rm L^{\ast}}}$. The spaces of sequences
$l^{2}_{{\rm L}}$ and
$l^{2}_{{\rm L^{\ast}}}$ are thus generated by biorthogonal systems
$\{u_{\xi}\}_{\xi\in\ind}$ and $\{v_{\xi}\}_{\xi\in\ind}$.
The reason for their definition in the above forms becomes clear again
in view of the following Plancherel identity:

\begin{prop} {\rm(Plancherel's identity)}\label{PlanchId}
If $f,\,g\in L^{2}(\Omega)$ then
$\widehat{f},\,\widehat{g}\in l^{2}_{{\rm L}}, \,\,\,
\widehat{f}_{\ast},\, \widehat{g}_{\ast}\in l^{2}_{{\rm
L^{\ast}}}$, and the inner products {\rm(\ref{EQ: InnerProd SpSeq-s}),
(\ref{EQ: InnerProd SpSeq-s_2})} take the form
$$
(\widehat{f},\ \widehat{g})_{l^{2}_{{\rm L}}}=\sum_{\xi\in\ind}\widehat{f}(\xi)\ \overline{\widehat{g}_{\ast}(\xi)}
$$
and
$$
(\widehat{f}_{\ast},\ \widehat{g}_{\ast})_{l^{2}_{{\rm
L^{\ast}}}}=\sum_{\xi\in\ind}\widehat{f}_{\ast}(\xi)\
\overline{\widehat{g}(\xi)}.
$$
In particular, we have
$$
\overline{(\widehat{f},\ \widehat{g})_{l^{2}_{{\rm L}}}}=
(\widehat{g}_{\ast},\ \widehat{f}_{\ast})_{l^{2}_{{\rm
L^{\ast}}}}.
$$
The Parseval identity takes the form
\begin{equation}\label{Parseval}
(f,g)_{L^{2}}=(\widehat{f},\widehat{g})_{l^{2}_{{\rm
L}}}=\sum_{\xi\in\ind}\widehat{f}(\xi)\ \overline{\widehat{g}_{\ast}(\xi)}.
\end{equation}
Furthermore, for any $f\in L^{2}(\Omega)$, we have
$\widehat{f}\in l^{2}_{{\rm L}}$, $\widehat{f}_{\ast}\in l^{2}_{{\rm
L^{\ast}}}$, and
\begin{equation}
\label{Planch} \|f\|_{L^{2}}=\|\widehat{f}\|_{l^{2}_{{\rm
L}}}=\|\widehat{f}_{\ast}\|_{l^{2}_{{\rm L^{\ast}}}}.
\end{equation}
\end{prop}

\begin{proof}
By the definition we get
\begin{align*}
(\mathcal F_{{\rm L^{\ast}}}\circ\mathcal F^{-1}_{{\rm
L}}\widehat{g})(\xi)=\left(\mathcal
F_{{\rm L^{\ast}}}g\right)(\xi)=\widehat{g}_{\ast}(\xi)
\end{align*}
and
\begin{align*}
(\mathcal F_{{\rm L}}\circ\mathcal F^{-1}_{{\rm
L^{\ast}}}\widehat{g}_{\ast})(\xi)=\left(\mathcal
F_{{\rm L}}g\right)(\xi)=\widehat{g}(\xi).
\end{align*}
Hence it follows that
$$
(\widehat{f},\ \widehat{g})_{l^{2}_{{\rm L}}}=\sum_{\xi\in\ind}\widehat{f}(\xi)\ \overline{(\mathcal F_{{\rm
L^{\ast}}}\circ\mathcal F^{-1}_{{\rm
L}}\widehat{g})(\xi)}=\sum_{\xi\in\ind}\widehat{f}(\xi)\
\overline{\widehat{g}_{\ast}(\xi)}
$$
and
$$
(\widehat{f}_{\ast},\ \widehat{g}_{\ast})_{l^{2}_{{\rm
L^{\ast}}}}=\sum_{\xi\in\ind}\widehat{f}_{\ast}(\xi)\
\overline{(\mathcal F_{{\rm L}}\circ\mathcal F^{-1}_{{\rm
L^{\ast}}}\widehat{g}_{\ast})(\xi)}=\sum_{\xi\in\ind}\widehat{f}_{\ast}(\xi)\ \overline{\widehat{g}(\xi)}.
$$
To show Parseval's identity \eqref{Parseval}, using these properties and the
biorthogonality of $u_\xi$'s to $v_\eta$'s,
we can write
\begin{multline*}
(f,g)_{L^{2}}
=\left(\sum_{\xi\in\ind}\widehat{f}(\xi)u_{\xi} \ , \
\sum_{\eta\in\ind}\widehat{g}_{\ast}(\eta)v_{\eta}\right)\\
=\sum_{\xi\in\ind}\sum_{\eta\in\ind}\widehat{f}(\xi)\overline{\widehat{g}_{\ast}(\eta)}\left(u_{\xi},
\ v_{\eta}\right)_{L^{2}}
=\sum_{\xi\in\ind}\widehat{f}(\xi)\overline{\widehat{g}_{\ast}(\xi)}=(\widehat{f},\widehat{g})_{l^{2}_{{\rm
L}}},
\end{multline*}
proving \eqref{Parseval}.
Taking $f=g$, we get
\begin{equation*}
\|f\|_{L^{2}}^{2}=(f,f)_{L^{2}}=
\sum_{\xi\in\ind}\widehat{f}(\xi)\overline{\widehat{f}_{\ast}(\xi)}=(\widehat{f},\widehat{f})_{l^{2}_{{\rm
L}}}=\|\widehat{f}\|_{l^{2}_{{\rm L}}}^{2},
\end{equation*}
proving the first equality in \eqref{Planch}.
Then, by checking that
\begin{align*}
(f,f)_{L^{2}}=\overline{(f,f)}_{L^{2}}&=\sum_{\xi\in\ind}
\overline{\widehat{f}(\xi)}\widehat{f}_{\ast}(\xi)=\sum_{\xi\in\ind}
\widehat{f}_{\ast}(\xi)\overline{\widehat{f}(\xi)}=(\widehat{f}_{\ast},\widehat{f}_{\ast})_{l^{2}_{{\rm
L^{\ast}}}}=\|\widehat{f}_{\ast}\|_{l^{2}_{{\rm L^{\ast}}}}^{2},
\end{align*}
the proofs of \eqref{Planch} and of Proposition \ref{PlanchId} are complete.
\end{proof}

Now we introduce Sobolev spaces generated by the operator ${\rm L}_{\Omega}$:

\begin{defn}[Sobolev spaces $\mathcal H^{s}_{{\rm L}}(\Omega)$] \label{SobSp}
For $f\in\mathcal D'_{{\rm L}}(\Omega)\cap \mathcal D'_{{\rm L}^{*}}(\Omega)$ and $s\in\mathbb R$, we say that
$$f\in\mathcal H^{s}_{{\rm L}}(\Omega)\; \textrm{ if and only if }\;
\langle\xi\rangle^{s}\widehat{f}(\xi)\in l^{2}_{{\rm L}}.$$
We define the norm on $\mathcal H^{s}_{{\rm L}}(\Omega)$ by
\begin{equation}\label{SobNorm}
\|f\|_{\mathcal H^{s}_{{\rm
L}}(\Omega)}:=\left(\sum_{\xi\in\ind}
\langle\xi\rangle^{2s}\widehat{f}(\xi)\overline{\widehat{f}_{\ast}(\xi)}\right)^{1/2}.
\end{equation}
The Sobolev space $\mathcal H^{s}_{{\rm L}}(\Omega)$ is then the
space of ${\rm L}$-distributions $f$ for which we have
$\|f\|_{\mathcal H^{s}_{{\rm L}}(\Omega)}<\infty$. Similarly,
we can define the
space $\mathcal H^{s}_{{\rm L^{\ast}}}(\Omega)$ by the
condition
\begin{equation}\label{SobNorm2}
\|f\|_{\mathcal H^{s}_{{\rm
L^{\ast}}}(\Omega)}:=\left(\sum_{\xi\in\ind}\langle\xi\rangle^{2s}\widehat{f}_{\ast}(\xi)\overline{\widehat{f}(\xi)}\right)^{1/2}<\infty.
\end{equation}
\end{defn}
We note that the expressions in \eqref{SobNorm} and 
\eqref{SobNorm2} are well-defined since the sum
$$
\sum_{\xi\in\ind}
\langle\xi\rangle^{2s}\widehat{f}(\xi)\overline{\widehat{f}_{\ast}(\xi)}=
(\langle\xi\rangle^{s}\widehat{f}(\xi),\langle\xi\rangle^{s}\widehat{f}(\xi))_{l^{2}_{\rm L}}\geq 0
$$
is real and non-negative. 
Consequently, since we can write the sum in \eqref{SobNorm2} as the
complex conjugate of that in  \eqref{SobNorm}, and with both being real,
we see that the spaces $\mathcal H^{s}_{{\rm L}}(\Omega)$ and 
$\mathcal H^{s}_{{\rm L^{\ast}}}(\Omega)$ coincide as sets. Moreover, we have

\begin{prop}\label{SobHilSpace}
For every $s\in\mathbb R$, the Sobolev space
$\mathcal H^{s}_{{\rm L}}(\Omega)$ is a Hilbert space with the
inner product
$$
(f,\ g)_{\mathcal H^{s}_{{\rm L}}(\Omega)}:=\sum_{\xi\in\ind
}\langle\xi\rangle^{2s}\widehat{f}(\xi)\overline{\widehat{g}_{\ast}(\xi)}.
$$
Similarly,
the Sobolev space
$\mathcal H^{s}_{{\rm L^{\ast}}}(\Omega)$ is a Hilbert space with
the inner product
$$
(f,\ g)_{\mathcal H^{s}_{{\rm
L^{\ast}}}(\Omega)}:=\sum_{\xi\in\ind}\langle\xi\rangle^{2s}\widehat{f}_{\ast}(\xi)\overline{\widehat{g}(\xi)}.
$$
For every $s\in\mathbb R$, the Sobolev spaces $\mathcal
H^{s}(\Omega)$, $\mathcal H^{s}_{{\rm L}}(\Omega)$,
and $\mathcal H^{s}_{{\rm L}^*}(\Omega)$ are
isometrically isomorphic.
\end{prop}

\begin{proof}
The spaces $\mathcal H^{0}_{{\rm L}}(\Omega)$ and $\mathcal
H^{s}_{{\rm L}}(\Omega)$ are isometrically isomorphic by the
canonical isomorphism $$\varphi_{s}:\mathcal H^{0}_{{\rm
L}}(\Omega)\rightarrow\mathcal H^{s}_{{\rm L}}(\Omega),$$ defined
by
$$
\varphi_{s}f(x):=\sum_{\xi\in\ind}\langle\xi\rangle^{-s}\widehat{f}(\xi)u_{\xi}(x).
$$
Indeed, $\varphi_{s}$ is a linear isometry between $\mathcal
H^{t}_{{\rm L}}(\Omega)$ and $\mathcal H^{t+s}_{{\rm L}}(\Omega)$
for every $s\in\mathbb R$, and it is true that
$$\varphi_{s_{1}}\varphi_{s_{2}}=\varphi_{s_{1}+s_{2}} \; \textrm{ and }\;
\varphi_{s}^{-1}=\varphi_{-s}.$$ Then the completeness of
$L^{2}(\Omega)=\mathcal H^{0}_{{\rm L}}(\Omega)$ is transferred to
that of $\mathcal H^{s}_{{\rm L}}(\Omega)$ for every $s\in\mathbb R$.

As $L^{2}(\Omega)=\mathcal H^{0}_{{\rm L}}(\Omega)$, the
spaces $L^{2}(\Omega)$ and $\mathcal H^{s}_{{\rm L}}(\Omega)$ are
isometrically isomorphic for every $s\in\mathbb R$. Hence the
Sobolev spaces $\mathcal H^{s}(\Omega)$ and $\mathcal H^{s}_{{\rm
L}}(\Omega)$ are also isometrically isomorphic for every
$s\in\mathbb R$. The arguments for the space $\mathcal H^{s}_{{\rm
L}^*}(\Omega)$ are all similar.
\end{proof}

\section{Spaces $l^{p}({\rm L})$ and $l^{p}({\rm L}^*)$}
\label{SEC:lp}

In this section we describe the $p$-Lebesgue versions of the spaces of Fourier coefficients.
These spaces can be considered as the extension of the usual $l^p$ spaces on the discrete set
$\ind$ adapted to the fact that we are dealing with biorthogonal systems.

\begin{defn}
Thus, we introduce the spaces $l^{p}_{\rm L}=l^{p}({\rm L})$ as the spaces of all
$a\in\mathcal S'(\ind)$ such that
\begin{equation}\label{EQ:norm1}
\|a\|_{l^{p}({\rm L})}:=\left(\sum_{\xi\in\ind}| a(\xi)|^{p}
\|u_{\xi}\|^{2-p}_{L^{\infty}(\Omega)} \right)^{1/p}<\infty,\quad \textrm{ for }\; 1\leq p\leq2,
\end{equation}
and
\begin{equation}\label{EQ:norm2}
\|a\|_{l^{p}({\rm L})}:=\left(\sum_{\xi\in\ind}| a(\xi)|^{p}
\|v_{\xi}\|^{2-p}_{L^{\infty}(\Omega)} \right)^{1/p}<\infty,\quad \textrm{ for }\; 2\leq p<\infty,
\end{equation}
and, for $p=\infty$,
$$
\|a\|_{l^{\infty}({\rm L})}:=\sup_{\xi\in\ind}\left( |a(\xi)|\cdot
\|v_{\xi}\|^{-1}_{L^{\infty}(\Omega)}\right)<\infty.
$$
\end{defn}

\begin{rem}\label{REM:lps}
We note that in the case of $p=2$, we have already defined the space $l^{2}({\rm L})$
by the norm \eqref{EQ:l2norm}. There is no problem with this since the norms
\eqref{EQ:norm1}-\eqref{EQ:norm2} with $p=2$ are equivalent to that in
\eqref{EQ:l2norm}. Indeed, by Lemma \ref{LEM: FTl2} the first one gives a 
homeomorphism between $l^{p}({\rm L})$ with $p=2$ just defined and $L^{2}(\Omega)$ while
the space $l^{2}({\rm L})$ defined by \eqref{EQ:l2norm} is isometrically isomorphic
to $L^{2}(\Omega)$ by the Plancherel identity in Proposition \ref{PlanchId}.
Therefore, both norms lead to the same space which we denote by $l^{2}({\rm L})$.
The norms \eqref{EQ:norm1}-\eqref{EQ:norm2} with $p=2$ and the one
in \eqref{EQ:l2norm} are equivalent, but there are advantages in using both of them.
Thus, the norms \eqref{EQ:norm1}-\eqref{EQ:norm2} allow us to view $l^{2}({\rm L})$
as a member of the scale of spaces $l^{p}({\rm L})$ for $1\leq p\leq \infty$ with
subsequent functional analytic properties, while the norm \eqref{EQ:l2norm} is the one
for which the Plancherel identity \eqref{Planch} holds. 
\end{rem}

Analogously, we also introduce spaces $l^{p}_{{\rm L^{\ast}}}=l^{p}({\rm L^{\ast}})$ as the spaces of
all $b\in\mathcal S'(\ind)$ such that the following norms are finite:
$$
\|b\|_{l^{p}({\rm L^{\ast}})}=\left(\sum_{\xi\in\ind}|
b(\xi)|^{p} \|v_{\xi}\|^{2-p}_{L^{\infty}(\Omega)} \right)^{1/p},\quad \textrm{ for }\; 1\leq p\leq2,
$$
$$
\|b\|_{l^{p}({\rm L^{\ast}})}=\left(\sum_{\xi\in\ind}|
b(\xi)|^{p} \|u_{\xi}\|^{2-p}_{L^{\infty}(\Omega)}
\right)^{1/p},\quad \textrm{ for }\; 2\leq p<\infty,
$$
$$
\|b\|_{l^{\infty}({\rm L^{\ast}})}=\sup_{\xi\in\ind}\left(|b(\xi)|\cdot \|u_{\xi}\|^{-1}_{L^{\infty}(\Omega)}\right).
$$
Before we discuss several basic properties of the spaces $l^{p}({\rm L})$, we
recall a useful fact on the interpolation of weighted spaces from Bergh and L\"ofstr\"om
\cite[Theorem 5.5.1]{Bergh-Lofstrom:BOOK-Interpolation-spaces}:

\begin{theorem}[Interpolation of weighted spaces] \label{TH: IWS}
Let us write
$d\mu_{0}(x)=\omega_{0}(x)d\mu(x),$
$d\mu_{1}(x)=\omega_{1}(x)d\mu(x),$ and write
$L^{p}(\omega)=L^{p}(\omega d\mu)$ for the weight $\omega$.
Suppose that $0<p_{0}, p_{1}<\infty$. Then
$$
(L^{p_{0}}(\omega_{0}), L^{p_{1}}(\omega_{1}))_{\theta,
p}=L^{p}(\omega),
$$
where $0<\theta<1$, $\frac{1}{p}=\frac{1-\theta}{p_{0}}+\frac{\theta}{p_{1}}$, and
$\omega=\omega_{0}^{\frac{p(1-\theta)}{p_{0}}}\omega_{1}^{\frac{p\theta}{p_{1}}}$.
\end{theorem}

From this it is easy to check that we obtain:

\begin{corollary}[Interpolation of $l^{p}({\rm L})$ and $l^{p}({\rm
L}^{\ast})$ spaces]
For $1\leq p\leq2$, we have
$$
(l^{1}({\rm L}), l^{2}({\rm L}))_{\theta,p}=l^{p}({\rm L}),
$$
$$
(l^{1}({\rm L}^{\ast}), l^{2}({\rm
L}^{\ast}))_{\theta,p}=l^{p}({\rm L}^{\ast}),
$$
where $0<\theta<1$ and $p=\frac{2}{2-\theta}$.
\end{corollary}

\begin{rem}
The reason that the interpolation above is restricted to $1\leq p\leq2$ is that the definition of
$l^p$-spaces changes when we pass $p=2$, in the sense that we use different families of
biorthogonal systems $u_\xi$ and $v_\xi$ for $p<2$ and for $p>2$. We note that if the boundary value problem
${\rm L}_\Omega={\rm L}_\Omega^*$ is self-adjoint, so that we can take
$u_\xi=v_\xi$ for all $\xi\in\ind$, then the scales $l^{p}({\rm L})$ and $l^{p}({\rm
L}^{\ast})$ coincide and satisfy interpolation properties for all $1\leq p<\infty$.
\end{rem}

Using these interpolation properties we can establish further properties of the Fourier transform
and its inverse:

\begin{theorem}[Hausdorff-Young inequality] \label{TH: HY}
Let $1\leq
p\leq2$ and $\frac{1}{p}+\frac{1}{p'}=1$. There is a constant $C_{p}\geq 1$ such that
for all $f\in L^{p}(\Omega)$
and $a\in l^{p}({\rm L})$ we have
\begin{equation}\label{EQ:HY}
\|\widehat{f}\|_{l^{p'}({\rm
L})}\leq C_{p}\|f\|_{L^{p}(\Omega)}\quad \textrm{ and }\quad \|\mathcal F_{{\rm
L}}^{-1}a\|_{L^{p'}(\Omega)}\leq C_{p}\|a\|_{l^{p}({\rm L})}.
\end{equation}
Similarly, we also have
\begin{equation}\label{EQ:HYast}
\|\widehat{f}_*\|_{l^{p'}({\rm L}^*)}\leq C_{p}\|f\|_{L^{p}(\Omega)}\quad \textrm{ and }
\quad \|\mathcal F_{{\rm
L}^{*}}^{-1}b\|_{L^{p'}(\Omega)}\leq C_{p}\|b\|_{l^{p}({\rm L}^*)},
\end{equation}
for all $b\in l^{p}({\rm L}^*)$.
\end{theorem}

It follows from the proof that 
if ${\rm L}_{\Omega}$ is self-adjoint, then the $l^{2}_{L}$-norms discussed in
Remark \ref{REM:lps} coincide, and so we can put $C_{p}=1$
in inequalities \eqref{EQ:HY} and \eqref{EQ:HYast}.
If ${\rm L}_{\Omega}$ is not self-adjoint, $C_{p}$ may in principle depend
on ${\rm L}$ and its domain
through constants from inequalities in Lemma \ref{LEM: FTl2}.

\begin{proof}[Proof of Theorem \ref{TH: HY}]
First we note that the proofs of \eqref{EQ:HY} and \eqref{EQ:HYast} are similar, so it
suffices to prove only \eqref{EQ:HY}. Then we observe that
\eqref{EQ:HY} would follow from the $L^{1}(\Omega)\rightarrow
l^{\infty}({\rm L})$ and $l^{1}({\rm L})\rightarrow L^{\infty}(\Omega)$ boundedness
in view of the Plancherel identity in Proposition \ref {PlanchId}
by interpolation, see e.g. Bergh and L\"ofstr\"om
\cite[Corollary 5.5.4]{Bergh-Lofstrom:BOOK-Interpolation-spaces}.
We note that in view of the discussion in Remark \ref{REM:lps} we 
write constants $C_{p}$ in inequalities \eqref{EQ:HY} and \eqref{EQ:HYast}.
In particular, if ${\rm L}_{\Omega}$ is self-adjoint, we can put $C_{p}=1$.

Thus, we can assume that $p=1$.
Using the definition
$\widehat{f}(\xi)=\int_{\Omega}f(x)\overline{v_{\xi}(x)}dx$
we get
$$
|\widehat{f}(\xi)|\leq\int_{\Omega}|f(x)|\,|\overline{v_{\xi}(x)}|dx\leq\|\overline{v_{\xi}}\|_{L^{\infty}}\|f\|_{L^{1}}.
$$
Therefore,
$$
\|\widehat{f}\|_{l^{\infty}({\rm L})}=\sup_{\xi\in\ind}|\widehat{f}(\xi)|
\|v_{\xi}\|^{-1}_{L^{\infty}}\leq\|f\|_{L^{1}},
$$
which gives the first inequality in \eqref{EQ:HY} for $p=1$.
For the second one, using 
$$(\mathcal F_{{\rm L}}^{-1}a)(x)=\sum\limits_{\xi\in\ind}a(\xi)u_{\xi}(x)$$ we
have
$$
|(\mathcal F_{{\rm L}}^{-1}a)(x)|\leq\sum\limits_{\xi\in\ind}|a(\xi)||u_{\xi}(x)|
\leq\sum\limits_{\xi\in\ind}|a(\xi)|\,\|u_{\xi}\|_{L^{\infty}}
=\|a\|_{l^{1}({\rm L})},
$$
from which we get
$$
\|\mathcal F_{{\rm L}}^{-1}a\|_{L^{\infty}}\leq\|a\|_{l^{1}({\rm L})},
$$
completing the proof.
\end{proof}

We now turn to the duality between spaces $l^{p}({\rm L})$ and
$l^{q}({\rm L}^{\ast})$:

\begin{theorem}[Duality of $l^{p}({\rm L})$ and $l^{p'}({\rm L}^{\ast})$] \label{TH:Duality lp}
Let $1\leq p<\infty$ and
$\frac{1}{p}+\frac{1}{p'}=1$. Then 
$$\left(l^{p}({\rm
L})\right)'=l^{p'}({\rm L}^{\ast}) \quad \textrm{ and }\quad \left(l^{p}({\rm
L}^{\ast})\right)'=l^{p'}({\rm L}).$$
\end{theorem}
\begin{proof} 
The proof is rather standard but we give some details for clarity.
The duality can be given by the form
$$
(\sigma_{1}, \sigma_{2})=\sum\limits_{\xi\in\ind
}\sigma_{1}(\xi){\sigma_{2}(\xi)}
$$
for $\sigma_{1}\in l^{p}({\rm L})$ and $\sigma_{2}\in l^{p'}({\rm
L}^{\ast})$.
Assume that $1<p\leq2$. Then, if $\sigma_{1}\in l^{p}({\rm L})$
and $\sigma_{2}\in l^{p'}({\rm L}^{\ast})$, we have
\begin{align*}
|(\sigma_{1}, \sigma_{2})|&= \left|\sum_{\xi\in\ind}
\sigma_{1}(\xi)\sigma_{2}(\xi)\right|\\
&=\left|\sum_{\xi\in\ind}
\sigma_{1}(\xi)\|u_{\xi}\|_{L^{\infty}}^{\frac{2}{p}-1}\|u_{\xi}\|_{L^{\infty}}^{-(\frac{2}{p}-1)}\sigma_{2}(\xi)\right|\\
&\leq\left(\sum_{\xi\in\ind}
|\sigma_{1}(\xi)|^{p}\|u_{\xi}\|_{L^{\infty}}^{p(\frac{2}{p}-1)}\right)^{p}\left(\sum_{\xi\in\ind}
|\sigma_{2}(\xi)|^{p'}\|u_{\xi}\|_{L^{\infty}}^{-p'(\frac{2}{p}-1)}\right)^{\frac{1}{p'}}\\
&=\|\sigma_{1}\|_{l^{p}({\rm L})}\|\sigma_{2}\|_{l^{p'}({\rm
L}^{\ast})},
\end{align*}
where we used that $2\leq p'<\infty$ and that
$\frac{2}{p}-1=1-\frac{2}{p'}$ in the last line.
Assume now that $2<p<\infty$. Then, if $\sigma_{1}\in l^{p}({\rm L})$
and $\sigma_{2}\in l^{p'}({\rm L}^{\ast})$, we have
\begin{align*}
|(\sigma_{1}, \sigma_{2})|&= \left|\sum_{\xi\in\ind}
\sigma_{1}(\xi)\sigma_{2}(\xi)\right|\\
&=\left|\sum_{\xi\in\ind}
\sigma_{1}(\xi)\|v_{\xi}\|_{L^{\infty}}^{\frac{2}{p}-1}\|v_{\xi}\|_{L^{\infty}}^{-(\frac{2}{p}-1)}\sigma_{2}(\xi)\right|\\
&\leq\left(\sum_{\xi\in\ind}
|\sigma_{1}(\xi)|^{p}\|v_{\xi}\|_{L^{\infty}}^{p(\frac{2}{p}-1)}\right)^{p}\left(\sum_{\xi\in\ind}
|\sigma_{2}(\xi)|^{p'}\|v_{\xi}\|_{L^{\infty}}^{-p'(\frac{2}{p}-1)}\right)^{\frac{1}{p'}}\\
&=\|\sigma_{1}\|_{l^{p}({\rm L})}\|\sigma_{2}\|_{l^{p'}({\rm
L}^{\ast})}.
\end{align*}

Let now $p=1$. In this case we get
\begin{align*}
|(\sigma_{1}, \sigma_{2})|&= \left|\sum_{\xi\in\ind}
\sigma_{1}(\xi)\sigma_{2}(\xi)\right|\\
&=\left|\sum_{\xi\in\ind}
\sigma_{1}(\xi)\|u_{\xi}\|_{L^{\infty}}\|u_{\xi}\|_{L^{\infty}}^{-1}\sigma_{2}(\xi)\right|\\
&\leq\left(\sum_{\xi\in\ind}
|\sigma_{1}(\xi)|\,\|u_{\xi}\|_{L^{\infty}}\right)\sup_{\xi\in\ind}|\sigma_{2}(\xi)|\,\|u_{\xi}\|^{-1}_{L^{\infty}}\\
&=\|\sigma_{1}\|_{l^{1}({\rm L})}\|\sigma_{2}\|_{l^{\infty}({\rm
L}^{\ast})}.
\end{align*}
The proofs for the adjoint spaces are similar.
\end{proof}

\section{Schwartz' kernel theorem}
\label{SEC:Schwartz}

This section is devoted to establishing the Schwartz kernel theorem in the spaces
of distributions $\mathcal D'_{{\rm L}}(\Omega)$.
In this analysis as  well as in establishing further estimates for the integral kernels in
Section \ref{SEC:kernels}, we will need the following assumption which may be
also regarded as the definition of the number $s_{0}$. So, from now on we will make the
following:

\begin{assump}
\label{Assumption_4}
Assume that the number
$s_0\in\mathbb R$ is such that we have
$$\sum_{\xi\in\ind} \langle\xi\rangle^{-s_0}<\infty.$$
\end{assump}
Recalling the operator ${\rm L}^{\circ}$ in \eqref{EQ:Lo-def}
the assumption \eqref{Assumption_4} is equivalent to assuming that
the operator $({\rm I}+{\rm L^\circ L})^{-\frac{s_0}{4m}}$ is Hilbert-Schmidt on $L^2(\Omega)$.
Indeed, recalling the definition of $\langle\xi\rangle$ in \eqref{EQ:angle},
namely that $\langle\xi\rangle$ are the eigenvalues of $({\rm I}+{\rm L^\circ L})^{-\frac{s_0}{2m}}$,
the condition
that the operator $({\rm I}+{\rm L^\circ L})^{-\frac{s_0}{4m}}$ is Hilbert-Schmidt is equivalent to
the condition that
\begin{equation}\label{EQ:HS-conv}
\|({\rm I}+{\rm L^\circ L})^{-\frac{s_0}{4m}}\|_{\tt HS}^2\cong \sum_{\xi\in\ind}
\langle\xi\rangle^{-s_0}<\infty.
\end{equation}
If L is elliptic, we may expect that we can take any $s_0>n$ but this depends on the
boundary conditions in general.
The order $s_0$ will enter the regularity properties of the Schwartz kernels.

\medskip
We will use the notation
$$C^{\infty}_{{\rm L}}(\overline{\Omega}\times \overline{\Omega}):=
C^{\infty}_{{\rm L}}(\overline{\Omega})\otimes C^{\infty}_{{\rm L}}(\overline{\Omega}),$$
and for the corresponding dual space we write
$$\mathcal D'_{{\rm L}}(\Omega\times\Omega):=
\left(C^{\infty}_{{\rm L}}(\overline{\Omega}\times \overline{\Omega})\right)^\prime.$$
The purpose of the subsequent discussion is to show that
for a continuous linear operator $$A:C^{\infty}_{{\rm
L}}(\overline{\Omega})\rightarrow\mathcal D'_{{\rm L}}(\Omega)$$
there exists the kernel $K\in \mathcal D'_{{\rm L}}(\Omega\times\Omega)$ such that
$$
\langle Af,g\rangle=\int_{\Omega}\int_{\Omega}K(x,y)f(x)g(y)dxdy,
$$
and, using the notion of the ${\rm L}$-convolution in Section \ref{SEC:conv},
the convolution kernel $k_{A}(x)\in\mathcal D'_{{\rm L}}(\Omega)$, such that
$$
Af(x)=(k_{A}(x)\sL f)(x).
$$
Here as usual, we identify an integrable function $w$ in, e.g., $C^{\infty}_{{\rm
L}}(\overline{\Omega})$, with the distribution
$$
C^{\infty}_{{\rm
L}}(\overline{\Omega})\ni \varphi\mapsto\langle w,
\varphi\rangle=\int_{\Omega}w(x)\varphi(x)dx,
$$
and we shall use the integral as a notation for the value $\langle
w, \varphi\rangle$ of $w$ at $\varphi$ also when $w$ is an
arbitrary distribution in $\mathcal D'_{{\rm L}}(\Omega)$.

\medskip
Consider the space $\mathcal A$ of all separately continuous
bilinear functionals $A$ on $C^{\infty}_{{\rm
L}}(\overline{\Omega})\times C^{\infty}_{{\rm
L}}(\overline{\Omega})$ with the topology of uniform convergence
on products of bounded sets in $C^{\infty}_{{\rm
L}}(\overline{\Omega})$. Any distribution $w$ in $\mathcal
D'_{{\rm L}}(\Omega\times\Omega)$ gives rise to such a functional
$A\in \mathcal A$ by specialisation to products of functions
\begin{equation}\label{LinFunc-l}
(\Lambda w)(f, g):=\langle w, f(x)g(y)\rangle=:A(f,g).
\end{equation}
The kernel theorem says that the mapping
$$
\Lambda: w\mapsto A
$$
is a linear homeomorphism between $\mathcal D'_{{\rm
L}}(\Omega\times\Omega)$ and $\mathcal A$. In particular, for every
$A\in \mathcal A$ there is precisely one `kernel' $K\in \mathcal D'_{{\rm L}}(\Omega\times\Omega)$
such that
$$
A(f,g)=\int_{\Omega}\int_{\Omega}K(x,y)f(x)g(y)dxdy.
$$
Such theorem was proved by Schwartz \cite{Schwartz} for
standard distributions, but then much simplified proofs have been given,
for instance, by Ehrenpreis \cite{Ehrenpreis} and by Gask
\cite{Gask}.

Any function $h$ in $C_{{\rm
L}}^{\infty}(\overline{\Omega}\times\overline{\Omega})$ can be
expanded in an ${\rm L}$-Fourier series
\begin{equation}\label{EQ: FS}
h(x,y)=\sum_{\xi,\eta\in\ind}a_{\xi\eta}u_{\xi}(x)u_{\eta}(y).
\end{equation}
The coefficients in (\ref{EQ: FS}) are given by
$$
a_{\xi\eta}=\int_{\Omega}\int_{\Omega}h(x,y)\overline{v_{\xi}(x)}\,\overline{v_{\eta}(y)}dxdy.
$$
Integration by parts in these formulae and the Cauchy-Schwarz inequality
yield the estimates
\begin{equation}\label{EQ: FSE}
|a_{\xi\eta}|\leq
C_{j_1,j_2}|h|_{j_1+j_2}(1+\langle\xi\rangle)^{-m j_1}(1+\langle\eta\rangle)^{-m j_2},
\end{equation}
where $C_{j_1,j_2}$ is a constant independent of $h$, $j_1,j_2\in\mathbb N_0$,
and
$$
|h|_{j}:=\sum_{k_1+k_2\leq j}\|{\rm L}_x^{k_1}{\rm L}_y^{k_2}h(x,y)\|_{L^2(\Omega\times\Omega)}.
$$
Rescaling the coefficients $a_{\xi\eta}$ we can write (\ref{EQ: FS}) in the form
\begin{equation}\label{EQ: FS_2}
h(x,y)=\sum_{\xi,\eta\in\ind}b_{\xi\eta}f_{\xi}(x)g_{\eta}(y)
\end{equation}
with $f_{\xi}(x)$ and $g_{\eta}(y)$ proportional to $u_{\xi}(x)$
and $u_{\eta}(y)$, with new coefficients $b_{\xi\eta}$.
The proportionality factors shall be chosen in
a suitable way, expressed in the following discussion.

\begin{lemma}\label{LEM: 6.1}
Let $h$ be a function in $C_{{\rm
L}}^{\infty}(\overline{\Omega}\times\overline{\Omega})$ and $k$
and $l$ given positive integers. Then $f_{\xi}$ and $g_{\eta}$ in
(\ref{EQ: FS_2}) can be chosen such that
$$
|f_{\xi}|_{k}\leq1
\quad
\textrm{ and }
\quad
|g_{\eta}|_{l}\leq1
$$
for all $\xi$ and $\eta$, and
$$
\sum_{\xi, \eta\in\ind}|b_{\xi\eta}|\leq C|h|_{k+l+2s_{0}},
$$
with a constant $C$ independent of $h$, and the number $s_{0}$ is the
one from Assumption \eqref{Assumption_4}.
\end{lemma}
\begin{proof}
We write (\ref{EQ: FS}) as
$$
h(x,y)=\sum_{\xi,\eta\in\ind}
a_{\xi\eta}(1+\langle\xi\rangle)^{m k}(1+\langle\eta\rangle)^{m l}
[(1+\langle\xi\rangle)^{-m k}u_{\xi}(x)][(1+\langle\eta\rangle)^{-m l}u_{\eta}(y)]
$$
and choose the functions in square brackets for $f_{\xi}$ and
$g_{\eta}$. The estimates (\ref{EQ: FSE}) and some
straightforward calculations then give the lemma.
\end{proof}

From Lemma \ref{LEM: 6.1} we readily obtain the following corollary
that expresses the fact that if $h$ is in some
bounded set in $C_{{\rm L}}^{\infty}(\overline{\Omega}\times\overline{\Omega})$, the
expansion \eqref{EQ: FS_2}
can be made such that (\ref{EQ: FS_3}) holds with
$f_{\xi}$ and $g_{\eta}$ in fixed bounded sets in $C_{{\rm
L}}^{\infty}(\overline{\Omega})$.

\begin{corollary}\label{LEM: 6.2}
Let $\{r_{\nu}\}_{\nu=1}^{\infty}$ be a sequence of positive real numbers.
Then there exists another sequence $\{s_{\nu}\}_{1}^{\infty}$ of positive real numbers
such that for every $h\in C_{{\rm
L}}^{\infty}(\overline{\Omega}\times\overline{\Omega})$ satisfying
$$
|h|_{\nu}\leq r_{\nu}, \quad \nu=1, 2, \ldots,
$$
we can choose $f_{\xi}$ and $g_{\eta}$ in
\eqref{EQ: FS_2} so that we have
$$
|f_{\xi}|_{\nu}\leq s_{\nu}, \quad |g_{\eta}|_{\nu}\leq s_{\nu},
\quad \nu=1, 2, \ldots,
$$
for all $\xi$ and $\eta$, and also
\begin{equation}\label{EQ: FS_3}
\sum_{\xi,\eta\in\ind}|b_{\xi\eta}|\leq1.
\end{equation}
\end{corollary}

Since $A\in\mathcal A$ is continuous, there exist a constant
$C$ and integers $k$ and $l$ (depending on $A$) for
which
\begin{equation}\label{EQ: 6.7}
|A(f, g)|\leq C|f|_{k}\, |g|_{l}, \,\,\, f, g\in C^{\infty}_{{\rm
L}}(\overline{\Omega}).
\end{equation}
As stated above, there is a mapping $\Lambda$ of $\mathcal
D'_{{\rm L}}(\overline{\Omega}\times\overline{\Omega})$ into
$\mathcal A$, defined by (\ref{LinFunc-l}). We shall now first
prove that the range of this mapping is the whole of $\mathcal A$
and that it is one-to-one.

\begin{theorem}\label{TH: 6.1}
For any separately continuous functional $A$ on $C_{{\rm
L}}^{\infty}(\overline{\Omega})\times C_{{\rm
L}}^{\infty}(\overline{\Omega})$ there exists precisely one
distribution $u$ in $\mathcal D'_{{\rm
L}}({\Omega}\times{\Omega})$ such that
\begin{equation}\label{EQ: 6.8}
(\Lambda u)(f, g):=\langle u, f(x)g(y)\rangle=A(f,g)
\end{equation}
holds for all $(f, g)$ in $C_{{\rm L}}^{\infty}(\overline{\Omega})\times
C_{{\rm L}}^{\infty}(\overline{\Omega})$.
The mapping $\Lambda$ defined by \eqref{EQ: 6.8} is a linear
homeomorphism.
\end{theorem}
\begin{proof}
Let us write an arbitrary $h$ in $C_{{\rm
L}}^{\infty}(\overline{\Omega}\times\overline{\Omega})$ in the
form given by Lemma \ref{LEM: 6.1}. If $k$ and $l$ are integers
such that (\ref{EQ: 6.7}) holds for our given $A$ we find
\begin{equation}\label{EQ: 6.9}
\sum_{\xi,\eta\in\ind}|b_{\xi\eta}|\,|A(f_{\xi},g_{\eta})|\leq C\, |h|_{k+l+2s_{0}},
\end{equation}
with $C$ independent of $h$. We define $u$ by
\begin{equation}\label{EQ: 6.10}
\langle u, h\rangle:=\sum_{\xi,\eta\in\ind}b_{\xi\eta}\,A(f_{\xi},g_{\eta})
\end{equation}
and conclude from (\ref{EQ: 6.9}) that $u$ is an ${\rm
L}$-distribution on $\Omega\times\Omega$ of order $k+l+2s_{0}$. It
is clear that (\ref{EQ: 6.8}) holds for this $u$ and also that $u$
is uniquely determined by $A$: indeed, if $A$ vanishes it follows from
\eqref{EQ: 6.10} that $u(h)=0$ on all finite sums
$h=\sum_{\xi,\eta\in\ind}b_{\xi\eta}\,f_{\xi}\,g_{\eta}$, and
as the set of such sums is dense in $C_{{\rm
L}}^{\infty}(\overline{\Omega}\times\overline{\Omega})$ the ${\rm
L}$-distribution $u$ must vanish.

Let us now show that the mapping $\Lambda$ defined by \eqref{EQ: 6.8} is a linear
homeomorphism.
In view of Proposition \ref{TH: UniBdd} and Lemma \ref{LEM: UniformBoundedness}
the topologies on $\mathcal A$ and $\mathcal D'_{{\rm
L}}(\overline{\Omega}\times\overline{\Omega})$ can be defined by
the seminorms, which can be also expressed as
\begin{align*}
\rho_{B_{x}B_{y}}(A)&=\sup|A(f,g)|, \,\,\,\,\, f\in B_{x}, \,\,
g\in
B_{y},\\
\varrho_{B_{xy}}(u)&=\sup|\langle u,h\rangle|, \,\,\,\,\,\,\,\,
h\in B_{xy},
\end{align*}
where $B_{x}$, $B_{y}$ are bounded sets in $C_{{\rm
L}}^{\infty}(\overline{\Omega})$ and $B_{xy}$ is a bounded set in
$C_{{\rm L}}^{\infty}(\overline{\Omega}\times\overline{\Omega})$.

It is clear that $\Lambda$ is linear. Let us show that
$\Lambda$ and $\Lambda^{-1}$ are both continuous.

Let $\rho_{B_{x}B_{y}}$ be an arbitrary seminorm on $\mathcal A$.
Then
$$
\rho_{B_{x}B_{y}}(\Lambda
u)=\rho_{B_{x}B_{y}}(A)=\sup\limits_{f\in B_{x}, g\in
B_{y}}|A(f,g)|=\sup\limits_{f\in B_{x}, g\in B_{y}}|\langle u,
fg\rangle|.
$$
It is easy to see that for any bounded sets $B_{x}\subset C_{{\rm
L}}^{\infty}(\overline{\Omega})$ and  $B_{y}\subset C_{{\rm
L}}^{\infty}(\overline{\Omega})$ there exists a bounded set
$B_{xy}\subset C_{{\rm
L}}^{\infty}(\overline{\Omega}\times\overline{\Omega})$ such that
all products $fg$ are in $B_{xy}$ whenever $f$ is in $B_{x}$ and
$g$ is in $B_{y}$. Hence
$$
\sup\limits_{f\in B_{x}, g\in B_{y}}|\langle u,
fg\rangle|\leq\sup\limits_{h\in B_{xy}}|\langle u,
h\rangle|=\varrho_{B_{xy}}(u),
$$
and so $\Lambda$ is continuous. Conversely, if $\varrho_{B_{xy}}$
is a seminorm on $C_{{\rm
L}}^{\infty}(\overline{\Omega}\times\overline{\Omega})$ we find
$$
\varrho_{B_{xy}}(\Lambda^{-1}A)=\varrho_{B_{xy}}(u)=\sup\limits_{h\in
B_{xy}}|\langle u, h\rangle|=\sup\limits_{h\in
B_{xy}}|\sum_{\xi,\eta\in\ind}b_{\xi\eta}\,A(f_{\xi},g_{\eta})|,
$$
where $h$ has been expanded as in Corollary \ref{LEM: 6.2}, and thus
$$
\sup\limits_{h\in
B_{xy}}|A(f_{\xi},g_{\eta})|\leq\sup\limits_{f\in B_{x}, g\in
B_{y}}|A(f,g)|=\rho_{B_{x}B_{y}}(A),
$$
if $B_{x}$ and $B_{y}$ are those bounded sets in $C_{{\rm
L}}^{\infty}(\overline{\Omega})$ which contain all $f_{\xi}$ and
$g_{\eta}$ according to Lemma \ref{LEM: 6.2}. From (\ref{EQ:
FS_3}) we now conclude that
$$
\varrho_{B_{xy}}(\Lambda^{-1}A)\leq\sup\limits_{h\in
B_{xy}}|A(f_{\xi},g_{\eta})|\sum_{\xi,\eta\in\ind}|b_{\xi\eta}|\leq\rho_{B_{x}B_{y}}(A),
$$
and thus $\Lambda^{-1}$ is also continuous. This completes the proof of the theorem.
\end{proof}

Summarising what we have proved, for any linear continuous operator
$$A:C^{\infty}_{{\rm L}}(\overline{\Omega})\rightarrow \mathcal D'_{{\rm L}}(\Omega)$$
there exists a kernel $K_{A}\in \mathcal D'_{{\rm
L}}(\Omega\times\Omega)$ such that for all $f\in C^{\infty}_{{\rm
L}}(\overline{\Omega})$, we can write, in the sense of distributions,
\begin{equation}\label{EQ:int1}
Af(x)=\int_{\Omega}K_{A}(x,y)f(y)dy.
\end{equation}
As usual, $K_{A}$ is called the Schwartz kernel of $A$.
For $f\in C^{\infty}_{{\rm L}}(\overline{\Omega})$, using the Fourier series formula
$$
f(y)=\sum\limits_{\eta\in\ind}\widehat{f}(\eta) u_{\eta}(y),
$$
we can also write
\begin{equation}\label{EQ:int2}
Af(x)=\sum\limits_{\eta\in\ind}\widehat{f}(\eta)\int_{\Omega}K_{A}(x,y)u_{\eta}(y)dy.
\end{equation}

Suppose now that $\{u_{\xi}: \,\,\, \xi\in\ind\}$ is a ${\rm WZ}$-system
in the sense of Definition \ref{DEF: WZ-system}.
Let us introduce the ${\rm L}$-distribution $k_{A}\in\mathcal D'_{{\rm L}}(\Omega\times\Omega)$
by the formula
\begin{equation} \label{EQ: KernelConv}
k_{A}(x,z):=k_A(x)(z):=\sum\limits_{\eta\in\ind}u_{\eta}^{-1}(x)
\int_{\Omega}K_{A}(x,y)u_{\eta}(y)dy \,
u_{\eta}(z).
\end{equation}
Since for some $C>0$ and $N\geq0$ we have by Definition \ref{DEF: WZ-system}
$$
\inf\limits_{x\in\overline{\Omega}}|u_{\eta}(x)|\geq
C\langle\eta\rangle^{-N},
$$
the series in (\ref{EQ: KernelConv}) is converges in the sense of
${\rm L}$-distributions.
Formula \eqref{EQ: KernelConv} means that the Fourier transform of $k_A$ in the
second variable satisfies
\begin{equation}\label{EQ:int3}
\widehat{k_{A}}(x,\eta)u_{\eta}(x)=
\int_{\Omega}K_{A}(x,y)u_{\eta}(y)dy.
\end{equation}
Combining this and \eqref{EQ:int2} we get
\begin{equation*}
Af(x)=\sum\limits_{\eta\in\ind}
\widehat{f}(\eta)\int_{\Omega}K_{A}(x,y)u_{\eta}(y)dy
=\sum\limits_{\eta\in\ind}\widehat{f}(\eta)\widehat{k}_{A}(x,\eta)u_{\eta}(x)
=(f\sL k_{A}(x))(x),
\end{equation*}
where in the last equality we used the notion of the L-convolution in Definition \ref{Convolution}.
Summarising this calculation as well as an analogous argument for the adjoint operator ${\rm L}^*$,
we record

\begin{prop}\label{PROP:conv-kernel}
Suppose that $\{u_{\xi}: \,\,\, \xi\in\ind\}$ is a ${\rm WZ}$-system
in the sense of Definition \ref{DEF: WZ-system}.
Then for a linear continuous operator $$A:C^{\infty}_{{\rm
L}}(\overline{\Omega})\rightarrow \mathcal D'_{{\rm L}}(\Omega)$$
there exists the convolution kernel $k_{A}\in\mathcal D'_{{\rm
L}}(\Omega\times\Omega)$ such that
$$Af(x)=(f\sL k_{A}(x))(x),\quad f\in C^{\infty}_{{\rm
L}}(\overline{\Omega}),$$ where we write $$k_{A}(x)(y):=k_{A}(x,y)$$
in the sense of distributions. The convolution kernel $k_A$ and
the Schwartz kernel $K_A$ of an operator $A$ are related by
formulae \eqref{EQ:int1}--\eqref{EQ:int3}.

Also, for any linear continuous operator $$A:C^{\infty}_{{\rm
L^{\ast}}}(\overline{\Omega})\rightarrow \mathcal D'_{{\rm
L^{\ast}}}(\Omega)$$ there exists a kernel $\widetilde{K}_{A}\in
\mathcal D'_{{\rm L^{\ast}}}(\Omega\times\Omega)$ such that for
all $f\in C^{\infty}_{{\rm L^{\ast}}}(\overline{\Omega})$ we have
$$
Af(x)=\int_{\Omega}\widetilde{K}_{A}(x,y)f(y)dy.
$$
If, in addition, $\{v_{\xi}: \,\,\, \xi\in\ind\}$ is a ${\rm
WZ}$-system, then for a linear continuous operator
$A:C^{\infty}_{{\rm L^{\ast}}}(\overline{\Omega})\rightarrow
\mathcal D'_{{\rm L}^*}(\Omega)$ there exists the
convolution kernel $\widetilde{k}_{A}\in\mathcal D'_{{\rm
L^{\ast}}}(\Omega\times\Omega)$, such that
$$Af(x)=(f\sLs \widetilde{k}_{A}(x))(x),\quad f\in C^{\infty}_{{\rm
L}^*}(\overline{\Omega}),$$ where we write $$\widetilde
k_{A}(x)(y):=\widetilde k_{A}(x,y)$$ in the sense of distributions.
\end{prop}
In the last formula we refer to \eqref{EQ: CONV2} for the
definition of the ${\rm L}^*$-convolution $\sLs$.

\section{${\rm L}$-Quantization and and full symbols}
\label{SEC:quantization}

In this section we describe the ${\rm L}$-quantization induced by the boundary value problem
${\rm L}_\Omega$. From now on we will assume
that the system of functions $\{u_{\xi}:\; \xi\in\ind\}$ is a ${\rm WZ}$-system
in the sense of Definition \ref{DEF: WZ-system}. Later, we will make some remarks
on what happens when this assumption is not satisfied.

\begin{defn}[${\rm L}$-Symbols of operators on $\Omega$] \label{$L$--Symbols}
The ${\rm L}$-symbol of a linear continuous
operator $$A:C^{\infty}_{{\rm L}}(\overline{\Omega})\rightarrow
\mathcal D'_{{\rm L}}(\Omega)$$ at $x\in\Omega$ and
$\xi\in\ind$ is defined by
$$\sigma_{A}(x, \xi):=\widehat{k_{A}(x)}(\xi)=\mathcal F_{{\rm L}}(k_{A}(x))(\xi).$$
Hence, we can also write
$$\sigma_{A}(x, \xi)=\int_{\Omega}k_{A}(x,y)\overline{v_{\xi}(y)}dy=
\langle k_{A}(x),\overline{v_{\xi}}\rangle.$$
\end{defn}

By the ${\rm L}$-Fourier inversion formula the convolution kernel
can be regained from the symbol:
\begin{equation}
\label{Kernel} k_{A}(x, y)=\sum_{\xi\in\ind}\sigma_{A}(x,
\xi)u_{\xi}(y),
\end{equation}
all in the sense of ${\rm L}$-distributions. We now show that an
operator $A$ can be represented by its symbol:

\begin{theorem}[${\rm L}$--quantization] \label{QuanOper}
Let $$A:C^{\infty}_{{\rm L}}(\overline{\Omega})\rightarrow
C^{\infty}_{{\rm L}}(\overline{\Omega})$$ be a continuous linear
operator with {\rm L}-symbol $\sigma_{A}$. Then
\begin{equation}\label{Quantization}
Af(x)=\sum_{\xi\in\ind}
u_{\xi}(x)\sigma_{A}(x, \xi) \widehat{f}(\xi)
\end{equation}
for every $f\in C^{\infty}_{{\rm L}}(\overline{\Omega})$ and
$x\in\Omega$.
The {\rm L}-symbol $\sigma_{A}$ satisfies
\begin{equation}\label{FormSymb}
\sigma_{A}(x,\xi)=u_{\xi}(x)^{-1}(Au_{\xi})(x)
\end{equation}
for all $x\in\Omega$ and $\xi\in\ind$.
\end{theorem}

\begin{proof}
Let us define a convolution operator $A_{x_{0}}\in\mathcal
L(C^{\infty}_{{\rm L}}(\overline\Omega))$ by the kernel
$$k_{x_{0}}(x):=k_{A}(x_{0},x),$$ 
i.e. by
$$A_{x_{0}}f(x):=(f\sL k_{x_{0}})(x),$$
with the usual distributional interpretation of the appearing quantities.
Thus
$$\sigma_{A_{x_{0}}}(x,\xi)=\widehat{k_{x_{0}}}(\xi)=\sigma_{A}(x_{0},\xi),$$
so that we have
$$
A_{x_{0}}f(x)=\sum_{\xi\in\ind}\widehat{A_{x_{0}}f}(\xi)
 u_{\xi}(x)
=\sum_{\xi\in\ind}\widehat{f}(\xi) \sigma_{A}(x_{0}, \xi)
 u_{\xi}(x),
$$
where we used that $\widehat{f\sL k_{x_{0}}
}=\widehat{f}\widehat{k_{x_{0}}}$ by the same calculations as in
Lemma \ref{ConvProp}. This implies \eqref{Quantization} because
$$Af(x)=A_{x}f(x).$$ For \eqref{FormSymb} , we can then calculate
$$
u_{\xi}(x)^{-1}(Au_{\xi})(x)
=u_{\xi}(x)^{-1}\sum_{\eta\in\ind} u_{\eta}(x) \sigma_{A}(x, \eta)
\widehat{u_{\xi}}(\eta)
 =\sigma_{A}(x, \xi),
$$
completing the proof.
\end{proof}

As a consequence of the proof and of various formulae for kernels and convolutions,
we can collect several formulae for the symbol under the assumption that
the biorthogonal system $u_\xi$ is a WZ-system:

\begin{corollary}\label{COR: SymFor}
We have the following equivalent formulae for {\rm L}-symbols:
\begin{align*}
{\rm (i)} \,\,\,\,\, \sigma_{A}(x, \xi)
&=\int_{\Omega}k_{A}(x,y)\overline{v_{\xi}(y)}dy;\\
{\rm (ii)} \,\,\,\,\, \sigma_{A}(x, \xi)&=u_{\xi}^{-1}(x)(Au_{\xi})(x);\\
{\rm (iii)} \,\,\,\,\,
\sigma_{A}(x,\xi)&=u_{\xi}^{-1}(x)\int_{\Omega}K_{A}(x,y)u_{\xi}(y)dy;\\
{\rm (iv)} \,\,\,\,\, \sigma_{A}(x,
\xi)&=u_{\xi}^{-1}(x)\int_{\Omega}\int_{\Omega}F(x,y,z)k_{A}(x,y)
u_{\xi}(z)dydz.
\end{align*}
Here and in the sequel we write $u_{\xi}^{-1}(x)=u_{\xi}(x)^{-1}.$
Formula {\rm (iii)} also implies
\begin{align*}
{\rm (v)} \,\,\,\,\, K_A(x,y)&=
\sum_{\xi\in\ind}  u_\xi(x) \sigma_A(x,\xi) \overline{v_\xi(y)}.
\end{align*}
\end{corollary}

In the case when $\{u_{\xi}: \; \xi\in\ind\}$ is not a
WZ-system, we can still understand the ${\rm L}$-symbol
$\sigma_{A}$ of the operator $A$ as a function on
$\overline{\Omega}\times\ind$, for which the equality
$$
u_{\xi}(x)\sigma_{A}(x,\xi)=\int_{\Omega}K_{A}(x,y)u_{\xi}(y)dy
$$
holds for all $\xi$ in $\ind$ and for $x\in\overline{\Omega}$.
Of course, this implies certain restrictions on the zeros of the
Schwartz kernel $K_A$. Such restrictions may be considered
natural from the point of view of the scope of problems that
can be treated by our approach in the case when the eigenfunctions
$u_\xi(x)$ may vanish at some points $x$.

\vspace{3mm}

Similarly, we can introduce an analogous notion of the ${\rm L^{\ast}}$-quantization.

\begin{defn}[${\rm L^{\ast}}$-Symbols of operators on $\Omega$] \label{$L$--Symbols_star}
The ${\rm L^{\ast}}$-symbol of a linear
continuous operator $$A:C^{\infty}_{{\rm
L^{\ast}}}(\overline{\Omega})\rightarrow \mathcal D'_{{\rm L}^*}(\Omega)$$ 
at $x\in{\Omega}$ and
$\xi\in\ind$ is defined by
$$\tau_{A}(x, \xi):=\widehat{\widetilde{k}_{A}(x)}_{\ast}(\xi)=\mathcal F_{{\rm L^{\ast}}}(\widetilde{k}_{A}(x))(\xi).$$
We can also write
$$\tau_{A}(x, \xi)=\int_{\Omega}\widetilde{k}_{A}(x,y)\overline{u_{\xi}(y)}dy=
\langle\widetilde{k}_{A}(x),\overline{u_{\xi}}\rangle.$$
\end{defn}

By the ${\rm L^{\ast}}$-Fourier inversion formula the convolution
kernel can be regained from the symbol:
\begin{equation}
\label{Kernel_star} \widetilde{k}_{A}(x, y)=\sum_{\xi\in\ind}\tau_{A}(x, \xi)v_{\xi}(y)
\end{equation}
in the sense of ${\rm L^{\ast}}$--distributions.
Analogously to the L-quantization, we have:

\begin{corollary}[${\rm L^{\ast}}$-quantization] \label{QuanOper_star}
Let $\tau_{A}$ be the ${\rm L}^*$-symbol of a
continuous linear operator $$A:C^{\infty}_{{\rm
L^{\ast}}}(\overline{\Omega})\rightarrow C^{\infty}_{{\rm
L^{\ast}}}(\overline{\Omega}).$$ Then
\begin{equation}
\label{Quantization_star} Af(x)=\sum_{\xi\in\ind}
v_{\xi}(x) \tau_{A}(x, \xi)  \widehat{f}_{\ast}(\xi)
\end{equation}
for every $f\in C^{\infty}_{{\rm L^{\ast}}}(\overline{\Omega})$
and $x\in\Omega$.
For all
$x\in\Omega$ and $\xi\in\ind$, we have
\begin{equation}
\label{FormSymb_star} \tau_{A}(x,
\xi)=v_{\xi}(x)^{-1}(Av_{\xi})(x).
\end{equation}
We also have the following equivalent formulae for the ${\rm L^{\ast}}$-symbol:
\begin{align*}
{\rm (i)} \,\,\,\,\, \tau_{A}(x, \xi)
&=\int_{\Omega}\widetilde{k}_{A}(x,y)\overline{u_{\xi}(y)}dy;\\
{\rm (ii)} \,\,\,\,\,
\tau_{A}(x,\xi)&=v_{\xi}^{-1}(x)\int_{\Omega}\widetilde{K}_{A}(x,y)v_{\xi}(y)dy.\\
\end{align*}
\end{corollary}

We now briefly describe the 
notion of Fourier multipliers which is a natural name for operators with 
symbols independent of $x$. In \cite{Delgado-Ruzhansky-Togmagambetov:nuclear}
the analysis of this paper is applied to investigate the spectral properties of
such operators, so we can be brief here.

\begin{defn}\label{Lfm} 
Let $A:C_L^{\infty}\omp\rightarrow C_L^{\infty}\omp$ 
be a continuous linear operator. 
We will say
 that $A$ is an $L$-Fourier multiplier if it satisfies
\[
\efel (Af)(\xi)=\sigma(\xi)\efel (f)(\xi),\; f\in C_{L}^{\infty}\omp,
\] 
for some $\sigma:\ind\rightarrow\mathbb C$.
Analogously we define $L^*$-Fourier multipliers:
Let $B:C_{L^*}^{\infty}\omp\rightarrow C_{L^*}^{\infty}\omp$ 
be a continuous linear operator. We will say
 that $B$ is an $L^*$-Fourier multiplier if it satisfies
\[
\efela (Bf)(\xi)=\tau(\xi)\efela (f)(\xi),\, f\in C_{L^*}^{\infty}\omp,
\]
for some $\tau:\ind\rightarrow\mathbb C$.
\end{defn}

As used in \cite{Delgado-Ruzhansky-Togmagambetov:nuclear}, we have the
following simple relation between the symbols of
an operator and its adjoint.

\begin{prop}\label{admu} 
The operator $A$ is an $L$-Fourier multiplier by $\sigma(\xi)$ if and only if
$A^*$ is an $L^*$-Fourier multiplier by $\overline{\sigma(\xi)}$.
\end{prop}

\begin{proof} 
It is enough to prove the `only if' implication.
First, by the Parceval identity
$$
(Af,g)_{L^2}=\sum\limits_{\xi\in\ind}\widehat{Af}(\xi)\overline{\widehat{g_{\ast}}(\xi)}
=\sum\limits_{\xi\in\ind}\sigma(\xi)\widehat{f}(\xi)\overline{\widehat{g_{\ast}}(\xi)}
=\sum\limits_{\xi\in\ind}\widehat{f}(\xi)\overline{\overline{\sigma(\xi)}\widehat{g_{\ast}}(\xi)}.
$$
At the same time
\[
(Af,g)_{L^2}=(f,A^*g)_{L^2}=\sum\limits_{\xi\in\ind}\widehat{f}(\xi)\overline{\widehat{A^*g_{\ast}}(\xi)}.\]
Consequently,
\[
\widehat{A^*g_{\ast}}(\xi)=\overline{\sigma(\xi)}\widehat{g_{\ast}}(\xi),
\]
i.e. $A^*$ is an $L^*$-multiplier by  $\overline{\sigma(\xi)}$.
\end{proof}

\section{Difference operators and symbolic calculus}
\label{SEC:differences}

In this section we discuss difference operators that will be instrumental in defining symbol
classes for the symbolic calculus of operators. An interesting new feature of these operators
compared to previous settings
is that they will be also dependent on a point $x\in\Omega$.

Let $q_{j}\in C^{\infty}({\Omega}\times{\Omega})$, $j=1,\ldots,l$, be a given family
of smooth functions.
We will call the collection of $q_j$'s {\em {\rm L}-strongly admissible} if the following properties hold:
\begin{itemize}
\item For every $x\in\Omega$, the multiplication by $q_{j}(x,\cdot)$
is a continuous linear mapping on
 $C^{\infty}_{{\rm L}}(\overline{\Omega})$, for all $j=1,\ldots,l$;
\item $q_{j}(x,x)=0$ for all $j=1,\ldots,l$;
\item $
{\rm rank}(\nabla_{y}q_{1}(x,y), \ldots, \nabla_{y}q_{l}(x,y))|_{y=x}=n;
$
\item  the diagonal in $\Omega\times\Omega$ is the only set when all of
$q_j$'s vanish:
$$
\bigcap_{j=1}^l \left\{(x,y)\in\Omega\times\Omega: \, q_j(x,y)=0\right\}=\{(x,x):\, x\in\Omega\}.
$$
\end{itemize}

We note that the first property above implies that for every $x\in\Omega$, the multiplication by
$q_{j}(x,\cdot)$ is also well-defined and extends to a continuous linear mapping on
$\mathcal D'_{{\rm L}}(\Omega)$. Also, the last property above contains the second one
but we chose to still give it explicitly for the clarity of the exposition.

The collection of $q_j$'s with the above properties generalises the notion of a strongly
admissible collection of functions for difference operators introduced in
\cite{Ruzhansky-Turunen-Wirth:JFAA} in the context of compact Lie groups.
We will use the multi-index notation
$$
q^{\alpha}(x,y):=q^{\alpha_1}_{1}(x,y)\cdots q^{\alpha_l}_{l}(x,y).
$$
Analogously, the notion of an ${\rm L}^{*}$-strongly admissible collection suitable for the
conjugate problem is that of a family
$\widetilde{q}_{j}\in C^{\infty}({\Omega}\times{\Omega})$, $j=1,\ldots,l$, satisfying the properties:
\begin{itemize}
\item For every $x\in\Omega$, the multiplication by $\widetilde{q}_{j}(x,\cdot)$
is a continuous linear mapping on
 $C^{\infty}_{{\rm L}^{*}}(\overline{\Omega})$, for all $j=1,\ldots,l$;
\item $\widetilde{q}_{j}(x,x)=0$ for all $j=1,\ldots,l$;
\item $
{\rm rank}(\nabla_{y}\widetilde{q}_{1}(x,y), \ldots, \nabla_{y}\widetilde{q}_{l}(x,y))|_{y=x}=n;
$
\item  the diagonal in $\Omega\times\Omega$ is the only set when all of
$\widetilde{q}_j$'s vanish:
$$
\bigcap_{j=1}^l \left\{(x,y)\in\Omega\times\Omega: \, \widetilde{q}_j(x,y)=0\right\}=\{(x,x):\, x\in\Omega\}.
$$
\end{itemize}
We also write
$$
\widetilde{q}^{\alpha}(x,y):=\widetilde{q}^{\alpha_1}_{1}(x,y)\cdots
\widetilde{q}^{\alpha_l}_{l}(x,y).
$$
We now record the Taylor expansion formula with respect to a family of $q_j$'s,
which follows from expansions of functions $g$ and
$q^{\alpha}(e,\cdot)$ by the common Taylor series:

\begin{prop}\label{TaylorExp}
Any smooth function $g\in C^{\infty}({\Omega})$ can be
approximated by Taylor polynomial type expansions, i.e. for $e\in\Omega$, we have
$$g(x)=\sum_{|\alpha|<
N}\frac{1}{\alpha!}D^{(\alpha)}_{x}g(x)|_{x=e}\, q^{\alpha}(e,x)+\sum_{|\alpha|=
N}\frac{1}{\alpha!}q^{\alpha}(e,x)g_{N}(x)
$$
\begin{equation}
\sim\sum_{\alpha\geq
0}\frac{1}{\alpha!}D^{(\alpha)}_{x}g(x)|_{x=e}\, q^{\alpha}(e,x)
\label{TaylorExpFormula}
\end{equation}
in a neighborhood of $e\in\Omega$, where $g_{N}\in
C^{\infty}({\Omega})$ and
$D^{(\alpha)}_{x}g(x)|_{x=e}$ can be found from the recurrent formulae:
$D^{(0,\cdots,0)}_{x}:=I$ and for $\alpha\in\mathbb N_0^l$,
$$
\mathsf
\partial^{\beta}_{x}g(x)|_{x=e}=\sum_{|\alpha|\leq|\beta|}\frac{1}{\alpha!}
\left[\mathsf
\partial^{\beta}_{x}q^{\alpha}(e,x)\right]\Big|_{x=e}D^{(\alpha)}_{x}g(x)|_{x=e},
$$
where $\beta=(\beta_1, \ldots, \beta_n)$ and
$
\partial^{\beta}_{x}=\frac{\partial^{\beta_{1}}}{\partial x_{1}^{\beta_{1}}}\cdots
\frac{\partial^{\beta_{n}}}{\partial x_{n}^{\beta_{n}}}.
$
\end{prop}

Analogously, any function $g\in C^{\infty}({\Omega})$
can be approximated by Taylor polynomial type expansions corresponding to
the adjoint problem, i.e. we
have
$$g(x)=\sum_{|\alpha|<
N}\frac{1}{\alpha!}\widetilde{D}^{(\alpha)}_{x}g(x)|_{x=e}\, \widetilde{q}^{\alpha}(e,x)+\sum_{|\alpha|=
N}\frac{1}{\alpha!}\widetilde{q}^{\alpha}(e,x)g_{N}(x)
$$
\begin{equation}
\sim\sum_{\alpha\geq
0}\frac{1}{\alpha!}\widetilde{D}^{(\alpha)}_{x}g(x)|_{x=e}\, \widetilde{q}^{\alpha}(e,x)
\label{TaylorExpFormula}
\end{equation}
in a neighborhood of $e\in\Omega$, where $g_{N}\in
C^{\infty}({\Omega})$ and
$\widetilde{D}^{(\alpha)}_{x}g(x)|_{x=e}$ are found from the
recurrent formula: $\widetilde{D}^{(0,\cdots,0)}:=I$ and for
$\alpha\in\mathbb N_{0}^{l}$,
$$
\partial^{\beta}_{x}g(x)|_{x=e}=\sum_{|\alpha|\leq|\beta|}\frac{1}{\alpha!}
\left[
\partial^{k}_{x}\widetilde{q}^{\alpha}(e,x)\right]\Big|_{x=e}\widetilde{D}^{(\alpha)}_{x}g(x)|_{x=e},
$$
where $\beta=(\beta_1, \ldots, \beta_n)$, and $\partial^{\beta}$ is defined as in
Proposition \ref{TaylorExp}.

It can be seen that operators $D^{(\alpha)}$ and
$\widetilde{D}^{(\alpha)}$ are differential operators of order
$|\alpha|$.
We now define difference operators acting on Fourier coefficients.
Since the problem in general may lack any invariance or symmetry structure,
the introduced difference operators will depend on a point $x$ where they
will be taken when applied to symbols.

\begin{defn}\label{DEF: DifferenceOper}\label{DEF: DifferenceOper_2}
For WZ-systems, we define difference operator $\Delta_{q,(x)}^{\alpha}$ acting
on Fourier coefficients by any of the following equal expressions
\begin{align*}
\Delta_{q,(x)}^{\alpha}\widehat{f}(\xi)& = u_{\xi}^{-1}(x)
\int_{\Omega}\Big[\int_{\Omega}q^{\alpha}(x,y)F(x,y,z)f(z)dz\Big]u_{\xi}(y)dy
\\
& = u_{\xi}^{-1}(x)
\sum_{\eta\in\ind}\mathcal
F_{L}\Big(q^{\alpha}(x,\cdot)u_{\xi}(\cdot)\Big)(\eta)\widehat{f}(\eta)u_{\eta}(x)
\\
& = u_{\xi}^{-1}(x) \left([q^{\alpha}(x,\cdot)u_{\xi}(\cdot)]\sL
f\right)(x).
\end{align*}
Analogously, we define the difference operator
$\widetilde{\Delta}_{q,(x)}^{\alpha}$ acting on adjoint Fourier
coefficients by
$$
\widetilde{\Delta}_{\widetilde{q},(x)}^{\alpha}\widehat{f}_{\ast}(\xi):=
v_{\xi}^{-1}(x)\sum_{\eta\in\ind}\mathcal F_{{\rm
L^{\ast}}}\Big(\widetilde{q}^{\alpha}(x,\cdot)v_{\xi}(\cdot)\Big)(\eta)\widehat{f}_{\ast}(\eta)v_{\eta}(x).
$$
\end{defn}
For simplicity, if there is no confusion, for a fixed collection
of $q_j$'s, instead of $\Delta_{q,(x)}$ and
$\widetilde{\Delta}_{\widetilde{q},(x)}$ we will often simply write
$\Delta_{(x)}$ and $\widetilde{\Delta}_{(x)}$.

\medskip
Recalling that the general philosophy behind the symbolic constructions and the definition of
the classes of symbols is that since the symbol is the Fourier transform of the (convolution)
kernel of the operator, the difference conditions correspond to the multiplication of the kernel
by functions vanishing on its singular support and, therefore, lead to the improved behaviour
reducing the strength of the singularity. Indeed, applying difference operators to a symbol
and using formulae from Section \ref{SEC:quantization}, we
obtain
\begin{align}\nonumber
\Delta_{(x)}^{\alpha}a(x,\xi)&=u_{\xi}^{-1}(x)\sum_{\eta\in\ind}\mathcal
F_{L}\Big(q^{\alpha}(x,\cdot)u_{\xi}(\cdot)\Big)(\eta)a(x,\eta)u_{\eta}(x)\\ \nonumber
&=u_{\xi}^{-1}(x)\sum_{\eta\in\ind}\mathcal
F_{L}\Big(q^{\alpha}(x,\cdot)u_{\xi}(\cdot)\Big)(\eta)\int_{\Omega}K(x,y)u_{\eta}(y)dy\\ \nonumber
&=u_{\xi}^{-1}(x)\int_{\Omega}K(x,y)\left[\sum_{\eta\in\ind}\mathcal
F_{L}\Big(q^{\alpha}(x,\cdot)u_{\xi}(\cdot)\Big)(\eta)u_{\eta}(y)\right]dy\\
&=u_{\xi}^{-1}(x)\int_{\Omega}q^{\alpha}(x,y)K(x,y)u_{\xi}(y)dy. \label{EQ:diff-symb}
\end{align}
In view of the first property of the strongly admissible collections, for each
$x\in\Omega$, the multiplication
by $q^{\alpha}(x,\cdot)$ is well defined on $\mathcal D'_{{\rm L}}(\Omega)$.
Therefore, we can write \eqref{EQ:diff-symb} also in the distributional form
$$
 \Delta_{(x)}^{\alpha}a(x,\xi) = u_{\xi}^{-1}(x)\,
 \langle q^{\alpha}(x,\cdot)K(x,\cdot),u_{\xi}\rangle,
$$
providing more light on the nature of the difference operators applied to symbols.
In view of the preceding discussion this and the latter formula \eqref{EQ:diff-symb}
yield indeed the justification of the definition of
difference operators as in Definition \ref{DEF: DifferenceOper}.

Plugging the expression (v) from Corollary \ref{COR: SymFor} for the kernel in terms of the symbol
into \eqref{EQ:diff-symb}, namely, using
$$
K(x,y)=
\sum_{\eta\in\ind}  u_\eta(x) a(x,\eta) \overline{v_\eta(y)},
$$
we record another useful form of \eqref{EQ:diff-symb} to be used later as
\begin{align}\nonumber
 \Delta_{(x)}^{\alpha}a(x,\xi) &=
 u_{\xi}^{-1}(x)\int_{\Omega}q^{\alpha}(x,y)\left[
 \sum_{\eta\in\ind}  u_\eta(x) a(x,\eta) \overline{v_\eta(y)} \right]u_{\xi}(y)dy \\
 & = u_{\xi}^{-1}(x)   \sum_{\eta\in\ind}u_\eta(x) a(x,\eta)
 \left[ \int_{\Omega}q^{\alpha}(x,y) \overline{v_\eta(y)} u_{\xi}(y)dy\right],
 \label{EQ:diff-symb-2}
\end{align}
with the usual distributional interpretation of all the steps. In the sequel we will also require
the ${\rm L^*}$-version of this formula, which we record now as
\begin{equation}\label{EQ:diff-symb-3}
  \widetilde\Delta_{(x)}^{\alpha}a(x,\xi)=
  v_{\xi}^{-1}(x)   \sum_{\eta\in\ind}v_\eta(x) a(x,\eta)
 \left[ \int_{\Omega}\widetilde q^{\alpha}(x,y) \overline{u_\eta(y)} v_{\xi}(y)dy\right].
\end{equation}

Using such difference operators and derivatives $D^{(\alpha)}$ from
Proposition \ref{TaylorExp}
we can now define classes of symbols.

\begin{defn}[Symbol class $S^m_{\rho,\delta}(\overline{\Omega}\times\ind)$]\label{DEF: SymClass}
Let $m\in\mathbb R$ and $0\leq\delta,\rho\leq 1$. The ${\rm L}$-symbol class
$S^m_{\rho,\delta}(\overline{\Omega}\times\ind)$ consists of
those functions $a(x,\xi)$ which are smooth in $x$ for all
$\xi\in\ind$, and which satisfy
\begin{equation}\label{EQ:symbol-class}
  \left|\Delta_{(x)}^\alpha D^{(\beta)}_{x} a(x,\xi) \right|
        \leq C_{a\alpha\beta m}
                \ \langle\xi\rangle^{m-\rho|\alpha|+\delta|\beta|}
\end{equation}
for all $x\in\overline{\Omega}$, for all $\alpha,\beta\geq 0$, and for all $\xi\in\ind$.
Here the operators $D^{(\beta)}_{x}$ are defined in Proposition
\ref{TaylorExp}. We will often denote them simply by $D^{(\beta)}$.

The class $S^m_{1,0}(\overline{\Omega}\times\ind)$ will be often
denoted by writing simply $S^m(\overline{\Omega}\times\ind)$.
In \eqref{EQ:symbol-class}, we assume that the inequality is satisfied for $x\in\Omega$ and
it extends to the closure $\overline\Omega$.
Furthermore, we define
$$
S^{\infty}_{\rho,\delta}(\overline{\Omega}\times\ind):=\bigcup\limits_{m\in\mathbb
R}S^{m}_{\rho,\delta}(\overline{\Omega}\times\ind)
$$
and
$$
S^{-\infty}(\overline{\Omega}\times\ind):=\bigcap\limits_{m\in\mathbb
R}S^{m}(\overline{\Omega}\times\ind).
$$
When we have two L-strongly admissible collections, expressing one in terms of
the other similarly to Proposition \ref{TaylorExp} and arguing similarly to
\cite{Ruzhansky-Turunen-Wirth:JFAA}, we can convince ourselves that for $\rho>\delta$ the
definition of the symbol class does not depend on the choice of an
 L-strongly admissible collection.

Analogously, we define the ${\rm L^{\ast}}$-symbol class
$\widetilde{S}^m_{\rho,\delta}(\overline{\Omega}\times\ind)$
as the space of those functions $a(x,\xi)$ which are smooth in $x$ for
all $\xi\in\ind$, and which satisfy
\begin{equation*}
  \left|\widetilde{\Delta}_{(x)}^\alpha \widetilde{D}^{(\beta)} a(x,\xi) \right|
        \leq C_{a\alpha\beta m}
                \ \langle\xi\rangle^{m-\rho|\alpha|+\delta|\beta|}
\end{equation*}
for all $x\in\overline{\Omega}$, for all $\alpha,\beta\geq 0$, and for all $\xi\in\ind$.
Similarly, we can define classes
$\widetilde{S}^{\infty}_{\rho,\delta}(\overline{\Omega}\times\ind)$
and $\widetilde{S}^{-\infty}(\overline{\Omega}\times\ind)$.
\end{defn}

If $a\in S^m_{\rho,\delta}(\overline{\Omega}\times\ind)$, it is convenient to
denote by $a(X,D)={\rm Op_L}(a)$  the corresponding ${\rm
L}$-pseudo-differential operator defined by
\begin{equation}\label{EQ: L-tor-pseudo-def}
  {\rm Op_L}(a)f(x)=a(X,D)f(x):=\sum_{\xi\in\ind} u_{\xi}(x)\ a(x,\xi)\widehat{f}(\xi).
\end{equation}
The set of operators ${\rm Op_L}(a)$ of the form
(\ref{EQ: L-tor-pseudo-def}) with $a\in
S^m_{\rho,\delta}(\overline{\Omega}\times\ind)$ will be denoted by
${\rm Op_L}(S^m_{\rho,\delta} (\overline{\Omega}\times\ind))$, or by
$\Psi^m_{\rho,\delta} (\overline{\Omega}\times\ind)$. If an
operator $A$ satisfies $A\in{\rm
Op_L}(S^m_{\rho,\delta}(\overline{\Omega}\times\ind))$, we denote
its ${\rm L}$-symbol by $\sigma_{A}=\sigma_{A}(x, \xi), \,\,
x\in\overline{\Omega}, \, \xi\in\ind$. Naturally,
$\sigma_{a(X,D)}(x,\xi)=a(x,\xi)$.

Analogously, if $a\in
\widetilde{S}^m_{\rho,\delta}(\overline{\Omega}\times\ind)$,
we denote by $a(X,D)={\rm Op_{L^*}}(a)$  the corresponding ${\rm
L^{\ast}}$-pseudo-differential operator defined by
\begin{equation}\label{EQ: L-tor-pseudo-def_2}
  {\rm Op_{L^*}}(a)f(x)=a(X,D)f(x):=\sum_{\xi\in\ind} v_{\xi}(x)\ a(x,\xi)\widehat{f}_{\ast}(\xi).
\end{equation}
The set of operators ${\rm Op_{L^*}}(a)$ of the form (\ref{EQ:
L-tor-pseudo-def_2}) with $a\in
\widetilde{S}^m_{\rho,\delta}(\overline{\Omega}\times\ind)$
will be denoted by ${\rm Op_{L^*}}(\widetilde{S}^m_{\rho,\delta}
(\overline{\Omega}\times\ind))$, or by
$\widetilde{\Psi}^m_{\rho,\delta} (\overline{\Omega}\times\ind)$.

\begin{rem}\label{REM: Topology of SymClass}
{\rm (Topology on $S^{m}_{\rho,
\delta}(\overline{\Omega}\times\ind)$ ($\widetilde{S}^{m}_{\rho,
\delta}(\overline{\Omega}\times\ind)$)).} The set $S^{m}_{\rho,
\delta}(\overline{\Omega}\times\ind)$ ($\widetilde{S}^{m}_{\rho,
\delta}(\overline{\Omega}\times\ind)$) of symbols has a natural
topology. Let us consider the functions $p_{\alpha\beta}^{l}:
S^{m}_{\rho,
\delta}(\overline{\Omega}\times\ind)\rightarrow\mathbb R$
($\widetilde{p}_{\alpha\beta}^{l}: \widetilde{S}^{m}_{\rho,
\delta}(\overline{\Omega}\times\ind)\rightarrow\mathbb R$) defined
by
$$
p_{\alpha\beta}^{l}(\sigma):={\rm
sup}\left[\frac{\left|\Delta_{(x)}^{\alpha}D^{(\beta)}\sigma(x,
\xi)\right|}{\langle\xi\rangle^{l-\rho|\alpha|+\delta|\beta|}}:\,\,
(x, \xi)\in\overline{\Omega}\times\ind\right]
$$
$$
\left(\widetilde{p}_{\alpha\beta}^{l}(\sigma):={\rm
sup}\left[\frac{\left|\widetilde{\Delta}_{(x)}^{\alpha}\widetilde{D}^{(\beta)}\sigma(x,
\xi)\right|}{\langle\xi\rangle^{l-\rho|\alpha|+\delta|\beta|}}:\,\,
(x, \xi)\in\overline{\Omega}\times\ind\right]\right).
$$
Now $\{p_{\alpha\beta}^{l}\}$
($\{\widetilde{p}_{\alpha\beta}^{l}\}$) is a countable family of seminorms,
and they define a
Fr\'echet topology on $S^{m}_{\rho,
\delta}(\overline{\Omega}\times\ind)$
($\widetilde{S}^{m}_{\rho, \delta}(\overline{\Omega}\times\mathbb
Z)$). Due to the bijective correspondence of ${\rm
Op_L}(S^{m}_{\rho, \delta}(\overline{\Omega}\times\ind))$ and
$S^{m}_{\rho, \delta}(\overline{\Omega}\times\ind)$ (${\rm
Op_{L^*}}(\widetilde{S}^{m}_{\rho,
\delta}(\overline{\Omega}\times\ind))$ and
$\widetilde{S}^{m}_{\rho, \delta}(\overline{\Omega}\times\mathbb
Z)$), this directly topologises also the set of operators. These spaces
are not normable, and the topologies have but a marginal role.
\end{rem}

The notion of a symbol can be naturally  extended to that of an amplitude.

\begin{defn}[${\rm L}$-amplitudes]\label{DEF: Amplitude}
The class $\mathcal A^m_{\rho,\delta}(\overline{\Omega})$ of
${\rm L}$-amplitudes consists of the functions
$a(x,y,\xi)$ which are smooth in $x$ and $y$ for all
$\xi\in\ind$ and which satisfy
\begin{equation}
  \left|\Delta_{(x)}^\alpha \Delta_{(y)}^{\alpha'} D^{(\beta)}_{x} D^{(\gamma)}_{y}
        a(x,y,\xi) \right|
        \leq C_{a\alpha\alpha'\beta\gamma m}
   \ \langle\xi\rangle^{m-\rho(|\alpha|+|\alpha'|)+\delta(|\beta|+|\gamma|)}
\end{equation}
for all $x,y\in\overline{\Omega}$, for all $\alpha,\alpha',
\beta,\gamma\geq 0$, and for all $\xi\in\ind$. Such a
function $a$ will be also called an ${\rm L}$-amplitude
of order $m\in\mathbb R$ of type $(\rho,\delta)$. Formally we may
also define
$$
  ({\rm Op_L}(a)f)(x) := \sum_{\xi\in\ind} \int_{\Omega}
    u_{\xi}(x)\ \overline{v_{\xi}(y)}\ a(x,y,\xi)\ f(y)\ dy
$$
for $f\in C_{{\rm L}}^\infty(\overline{\Omega})$. Sometimes we
may denote ${\rm Op_L}(a)$ by $a(X,Y,D).$
We also write $\mathcal
A^{m}(\overline{\Omega}):=\mathcal A^{m}_{1,
0}(\overline{\Omega})$ as well as
$$
\mathcal
A^{-\infty}(\overline{\Omega}):=\bigcap\limits_{m\in\mathbb
R}\mathcal A^{m}(\overline{\Omega}) \,\,\,\, \hbox{and} \,\,\,\,
\mathcal
A^{\infty}_{\rho,\delta}(\overline{\Omega}):=\bigcup\limits_{m\in\mathbb
R}\mathcal A^{m}_{\rho,\delta}(\overline{\Omega}).
$$
\end{defn}

Clearly we can regard the ${\rm L}$-symbols as a special class of
${\rm L}$-amplitudes, namely the ones
independent of the middle argument.
Analogously, the class $\widetilde{\mathcal
A}^m_{\rho,\delta}(\overline{\Omega})$ of ${\rm
L^{\ast}}$-amplitudes consists of the functions
$a(x,y,\xi)$ which are smooth in $x$ and $y$ for all
$\xi\in\ind$ and which satisfy
\begin{equation}
  \left|\widetilde{\Delta}_{(x)}^\alpha \widetilde{\Delta}_{(y)}^{\alpha'} \widetilde{D}^{(\beta)}_{x} \widetilde{D}^{(\gamma)}_{y}
        a(x,y,\xi) \right|
        \leq C_{a\alpha\beta\gamma m}
   \ \langle\xi\rangle^{m-\rho(|\alpha|+|\alpha'|)+\delta(|\beta|+|\gamma|)}
\end{equation}
for all $x,y\in\overline{\Omega}$, for all $\alpha, \alpha',
\beta,\gamma\geq 0$, and for all $\xi\in\ind$.
Formally we may also write
$$
  ({\rm Op_{L^*}}(a)f)(x) := \sum_{\xi\in\ind} \int_{\Omega}
    v_{\xi}(x)\ \overline{u_{\xi}(y)}\ a(x,y,\xi)\ f(y)\ dy
$$
for $f\in C_{{\rm L^{\ast}}}^\infty(\overline{\Omega})$.
We also write
$\widetilde{\mathcal
A}^{m}(\overline{\Omega}):=\widetilde{\mathcal A}^{m}_{1,
0}(\overline{\Omega})$ as well as
$
\widetilde{\mathcal
A}^{-\infty}(\overline{\Omega}):=\bigcap\limits_{m\in\mathbb
R}\widetilde{\mathcal A}^{m}(\overline{\Omega})$
and $\widetilde{\mathcal
A}^{\infty}_{\rho,\delta}(\overline{\Omega}):=\bigcup\limits_{m\in\mathbb
R}\widetilde{\mathcal A}^{m}_{\rho,\delta}(\overline{\Omega}).
$

\begin{defn}[Equivalence of amplitudes] \label{DEF: EquivAmplit}
We say that amplitudes $a, a'$ are $m(\rho,
\delta)$-equivalent $(m\in\mathbb R)$,
$a\stackrel{m,\rho,\delta}{\sim} a'$, if $a-a'\in\mathcal
A^{m}_{\rho,\delta}(\overline{\Omega})$; they are asymptotically
equivalent, $a\sim a'$ (or $a\stackrel{-\infty}{\sim} a'$ if we
need additional clarity), if $a-a'\in\mathcal
A^{-\infty}(\overline{\Omega})$. For the corresponding operators we also write
${\rm Op}(a)\stackrel{m,\rho,\delta}{\sim} {\rm Op}(a')$ and ${\rm
Op}(a)\sim {\rm Op}(a')$ (or ${\rm Op}(a)\stackrel{-\infty}{\sim}
{\rm Op}(a')$ if we need additional clarity), respectively. It is
obvious that $\stackrel{m,\rho,\delta}{\sim}$ and $\sim$ are
equivalence relations.
\end{defn}

From the algebraic point of view, we could handle the amplitudes,
symbols, and operators modulo the equivalence relation $\sim$,
because the ${\rm L}$-pseudo-differential operators form a
$\ast$-algebra with ${\rm
Op}(S^{-\infty}(\overline{\Omega}\times\ind))$ as a
subalgebra.

\vspace{3mm}

The next theorem is a prelude to asymptotic expansions, which are
the main tool in the symbolic analysis of ${\rm
L}$-pseudo-differential operators.

\begin{theorem}[Asymptotic sums of symbols] Let $(m_{j})_{j=0}^{\infty}\subset\mathbb
R$ be a sequence such that $m_{j}>m_{j+1}$, and
$m_{j}\rightarrow-\infty$ as $j\rightarrow\infty$, and
$\sigma_{j}\in
S^{m_{j}}_{\rho,\delta}(\overline{\Omega}\times\ind)$ for all
$j\in\ind$. Then there exists an ${\rm L}$-symbol $\sigma\in
S^{m_{0}}_{\rho,\delta}(\overline{\Omega}\times\ind)$ such that
for all $N\in\ind$,
$$
\sigma\stackrel{m_{N},\rho,\delta}{\sim}\sum_{j=0}^{N-1}\sigma_{j}.
$$
\end{theorem}

\begin{proof}
The proof is rather standard.
Choose a function $\chi\in C^{\infty}(\mathbb
R)$ satisfying $|\xi|\geq1\Rightarrow\chi(\xi)=1$ and
$|\xi|\leq\frac{1}{2}\Rightarrow\chi(\xi)=0$; otherwise
$0\leq\chi(\xi)\leq1$. Take a sequence
$(\varepsilon_{j})_{j=0}^{\infty}$ of positive real numbers such
that $\varepsilon_{j}>\varepsilon_{j+1}$, and
$\varepsilon_{j}\rightarrow0$ as $j\rightarrow\infty$, and define
$\chi_{j}\in C^{\infty}(\mathbb R)$ by
$\chi_{j}(\xi):=\chi(\varepsilon_{j}\xi).$ Since $\chi_{j}(\xi)=1$
for sufficiently large $\xi$, we get $\chi_{j}\sigma_{j}\in
S^{m_{j}}_{\rho,\delta}(\overline{\Omega}\times\ind)$ for each
$j$. For any fixed $\xi\in\ind$ the function
$\chi_{j}(\xi)\sigma_{j}(x,\xi)$ vanishes, when $j$ is large
enough. This justifies the definition
$$
\sigma(x,\xi):=\sum_{j=0}^{\infty}\chi_{j}(\xi)\sigma_{j}(x,\xi),
$$
and clearly $\sigma\in
S^{m_{0}}_{\rho,\delta}(\overline{\Omega}\times\ind)$.
Furthermore,
$$
\sigma(x,\xi)-\sum_{j=0}^{N-1}\sigma_{j}(x,\xi)=
\sum_{j=0}^{N-1}[\chi_{j}(\xi)-1]\sigma_{j}(x,\xi)+\sum_{j=N}^{\infty}\chi_{j}(\xi)\sigma_{j}(x,\xi).
$$
Recall that $\varepsilon_{j}>\varepsilon_{j+1}$, and
$\varepsilon_{j}\rightarrow0$ as $j\rightarrow\infty$, so that the
$\sum_{j=0}^{N-1}$ part of the sum vanishes, whenever $\xi$ is
large. This shows that
$$\sigma(x,\xi)-\sum_{j=0}^{N-1}\chi_{j}(\xi)\sigma_{j}(x,\xi)\in
S^{m_{N}}_{\rho,\delta}(\overline{\Omega}\times\ind)$$ finishing
the proof.
\end{proof}

We will now look at the formula for the symbol of the adjoint operator.
Let $A\in {\rm Op_L}
(S^m_{\rho,\delta}(\overline{\Omega}\times\ind))$. By the
definition of the adjoint operator we have
$$
(Au_{\xi}, v_{\eta})_{L^2}=(u_{\xi}, A^{*}v_{\eta})_{L^2}
$$
or
$$
\int_{\Omega}Au_{\xi}(x)\overline{v_{\eta}(x)}dx=\int_{\Omega}u_{\xi}(x)\overline{A^{\ast}v_{\eta}(x)}dx
$$
for $\xi, \eta\in\ind$.
Plugging in the integral expressions, we get
\begin{align*}
\int_{\Omega}{\left[\int_{\Omega}K_{A}(x,y)u_{\xi}(y)dy\right]}\overline{v_{\eta}(x)}dx & =
\int_{\Omega}{u_{\xi}(x)}\overline{\left[\int_{\Omega}K_{A^{\ast}}(x,y)v_{\eta}(y)dy\right]}dx \\
& = \int_{\Omega}{u_{\xi}(y)}{\left[\int_{\Omega}\overline{K_{A^{\ast}}(y,x)}
\overline{v_{\eta}(x)}dx\right]}dy
\end{align*}
for $\xi, \eta\in\ind$, where we swapped $x$ and $y$ in the last formula.
Consequently, we get the familiar property
$$
K_{A^{\ast}}(x,y)=\overline{K_{A}(y,x)}.
$$
Now, using this and formula (ii) in Corollary \ref{QuanOper_star},
and then formula (v) in Corollary \ref{COR: SymFor} and the Taylor expansion in Proposition \ref{TaylorExp},
we can write for the ${\rm L}^*$-symbol $\tau_{A^*}$ of $A^*$ that
\begin{align*}
v_\xi(x) \tau_{A^*}(x,\xi) & =
\int_\Omega K_{A^{\ast}}(x,y)v_\xi(y) dy \\
& =
\int_\Omega \overline{K_{A}(y,x)} v_\xi(y) dy \\
& =
\int_\Omega \sum_{\eta\in\ind}  \overline{u_\eta(y)  \sigma_A(y,\eta)} v_\eta(x)  v_\xi(y) dy \\
& \sim \int_\Omega \sum_{\eta\in\ind}  \overline{u_\eta(y)}
\sum_\alpha \frac{1}{\alpha!} \overline {D_x^{(\alpha)} \sigma_A(x,\eta)
q^\alpha(x,y)} v_\eta(x)  v_\xi(y) dy
\end{align*}
as an asymptotic sum. Formally regrouping terms for each $\alpha$, we obtain
$$
\tau_{A^*}(x,\xi) \sim v_\xi(x)^{-1}
\sum_\alpha \frac{1}{\alpha!}  \sum_{\eta\in\ind}  v_\eta(x) \overline {D_x^{(\alpha)} \sigma_A(x,\eta)}
 \int_\Omega \overline{q^\alpha(x,y)}  \overline{u_\eta(y)} v_\xi(y) dy.
$$
Using the ${\rm L^*}$-version of the  difference formula \eqref{EQ:diff-symb-3},
taking $$\widetilde{q}(x,y):=\overline{q(x,y)}$$
we can write this as
$$
\tau_{A^*}(x,\xi) \sim \sum_\alpha \frac{1}{\alpha!}
\widetilde \Delta_x^\alpha D_x^{(\alpha)}\overline{\sigma_A(x,\xi)}.
$$
Making rigorous estimates for the remainder in a routine way,
and assuming in the following theorem that for every $x\in\Omega$,
the multiplication by $q_{j}(x,\cdot)$ preserves both spaces
$C_{{\rm L}}^\infty(\overline{\Omega})$ and $C_{{\rm L}^{*}}^\infty(\overline{\Omega})$,
we have proved:

\begin{theorem}[Adjoint operators]
Let $0\leq\delta<\rho\leq 1$. Let $A\in {\rm Op_L}
(S^m_{\rho,\delta}(\overline{\Omega}\times\ind))$.
Assume that the conjugate symbol
class $\widetilde{S}^{m}_{\rho,\delta}(\overline{\Omega}\times\ind)$
is defined with strongly admissible
functions $\widetilde{q}_{j}(x,y):=\overline{q_{j}(x,y)}$ which are ${\rm L}^{*}$-strongly admissible.
Then the adjoint of $A$ satisfies
$A^{\ast}\in {\rm Op_{L^*}}(\widetilde{S}^{m}_{\rho,\delta}(\overline{\Omega}\times\ind))$,
with its ${\rm L}^*$-symbol
$\tau_{A^*}\in \widetilde{S}^{m}_{\rho,\delta}(\overline{\Omega}\times\ind)$
having the asymptotic expansion
$$
\tau_{A^*}(x,\xi) \sim \sum_\alpha \frac{1}{\alpha!}
\widetilde \Delta_x^\alpha D_x^{(\alpha)}\overline{\sigma_A(x,\xi)}.
$$
\end{theorem}

We now treat symbols of the amplitude operators.

\begin{theorem}[Amplitude symbols]
Let $0\leq\delta<\rho\leq 1$. For every amplitude $a\in \mathcal
A^m_{\rho,\delta}(\overline{\Omega})$ there exists a unique {\rm L}-symbol
$\sigma\in S^m_{\rho,\delta}(\overline{\Omega}\times\ind)$
satisfying ${\rm Op_L}(a)={\rm Op_L}(\sigma)$, where
\begin{equation}
  \sigma(x,\xi) \sim \sum_{\alpha\geq 0 } \frac{1}{\alpha!}
        \ \Delta_{(x)}^{\alpha}
        \ D_y^{(\alpha)} a(x,y,\xi)|_{y=x}.
\end{equation}
\end{theorem}
\begin{proof} As a linear operator on $C_{{\rm
L}}^{\infty}(\overline{\Omega})$, the operator ${\rm Op_L}(a)$ possesses the
unique L-symbol $\sigma=\sigma_{{\rm Op_L}(a)}$, but at the moment we
do not yet know whether $\sigma\in
S^m_{\rho,\delta}(\overline{\Omega}\times\ind)$. By Theorem
\ref{QuanOper}  the L-symbol is computed from
$$
\sigma(x,\xi)=u_{\xi}^{-1}(x)({\rm Op_L}(a)u_{\xi})(x)
=u_{\xi}^{-1}(x)\sum_{\eta\in\ind} \int_{\Omega}
    u_{\eta}(x)\ \overline{v_{\eta}(y)}\ a(x,y,\eta)\ u_{\xi}(y) dy.
$$
Let us denote
$$
k_{a}(x,y,z):= (\mathcal F^{-1}_{{\rm L}} a(x,y,\cdot))(z)
$$
i.e., we get
$$
a(x,y,\xi)=(\mathcal F_{{\rm L}} k_{a}(x,y,\cdot))(\xi).
$$
Then we have
\begin{align*}
\sigma(x,\xi)&=
u_{\xi}^{-1}(x)\sum_{\eta\in\ind} \int_{\Omega}
\int_{\Omega}    u_{\eta}(x)\ \overline{v_{\eta}(y)}\
\overline{v_{\eta}(z)} k_{a}(x,y,z)\ u_{\xi}(y)\ dy \ dz
\\
&=u_{\xi}^{-1}(x)\int_{\Omega} \int_{\Omega}
F(x,y,z) k_{a}(x,y,z)\ u_{\xi}(y)\ dy \ dz.
\end{align*}
Now we approximate
the function $k_{a}(x,\cdot,z)\in C_{{\rm
L}}^{\infty}(\overline{\Omega})$ by Taylor polynomial type
expansions, by using Proposition \ref{TaylorExp}, we have
\begin{align*}
\sigma(x,\xi)&\sim\sum_{\alpha\geq 0}\frac{1}{\alpha!}\int_{\Omega}
\int_{\Omega} F(x,y,z)
q^{\alpha}(x,y)\\
&\,\,\,\,\,\,\,\,\,\,\,\,\,\,\,\,\,\,\,\,\,\,\,\,\,\,\,\,\times\big[
D^{(\alpha)}_{y}k_{a}(x,y,z)\big]_{y=x}\
u_{\xi}(y)u_{\xi}^{-1}(x)dzdy
\\
&\sim\sum_{\alpha\geq 0} \frac{1}{\alpha!}
        \ \Delta_{(x)}^\alpha
        \ D_y^{(\alpha)} a(x,y,\xi)|_{y=x}.
\end{align*}
Omitting a routine verification of the properties of the remainder,
this yields the statement.
\end{proof}

We now formulate the composition formula.

\begin{theorem}\label{Composition}
Let $m_{1}, m_{2}\in\mathbb R$ and $\rho>\delta\geq0$. Let $A,
B:C_{{\rm L}}^{\infty}(\overline{\Omega})\rightarrow C_{{\rm
L}}^{\infty}(\overline{\Omega})$ be continuous and linear, and assume that
their {\rm L}-symbols satisfy
\begin{align*}
|\Delta_{(x)}^{\alpha}\sigma_{A}(x,\xi)|&\leq
C_{\alpha}\langle\xi\rangle^{m_{1}-\rho|\alpha|},\\
|D^{(\beta)}\sigma_{B}(x,\xi)|&\leq
C_{\beta}\langle\xi\rangle^{m_{2}+\delta|\beta|},
\end{align*}
for all $\alpha,\beta\geq 0$, uniformly in $x\in\overline{\Omega}$ and
$\xi\in\ind$.
Then
\begin{equation}
\sigma_{AB}(x,\xi)\sim\sum_{\alpha\geq 0}
\frac{1}{\alpha!}(\Delta_{(x)}^{\alpha}\sigma_{A}(x,\xi))D^{(\alpha)}\sigma_{B}(x,\xi),
\label{CompositionForm}
\end{equation}
where the asymptotic expansion means that for every $N\in\mathbb N$ we have
$$
|\sigma_{AB}(x,\xi)-\sum_{|\alpha|<N}\frac{1}{\alpha!}(\Delta_{(x)}^{\alpha}\sigma_{A}(x,\xi))D^{(\alpha)}\sigma_{B}(x,\xi)|\leq
C_{N}\langle\xi\rangle^{m_{1}+m_{2}-(\rho-\delta)N}.
$$
\end{theorem}

\begin{proof}
First, by the Schwartz kernel theorem from Section \ref{SEC:Schwartz}, we have
\begin{align*}
ABf(x)&=(k_{A}(x)\sL Bf)(x)\\
&=\int_{\Omega}\Big[\int_{\Omega}F(x,y,z)k_{A}(x,z)dz\Big](Bf)(y)dy
\\
&=\int_{\Omega}\Big(\Big[\int_{\Omega}F(x,y,z)k_{A}(x,z)dz\Big]\\
&\times\int_{\Omega}\Big[\int_{\Omega}F(y,s,t)k_{B}(y,t)dt\Big]f(s)ds\Big)dy.
\end{align*}
Hence
\begin{align*}
\sigma_{AB}(x,\xi)&=u_{\xi}^{-1}(x)(A(Bu_{\xi}))(x) \\
&=u_{\xi}^{-1}(x)\int_{\Omega}\Big(\Big[\int_{\Omega}F(x,y,z)k_{A}(x,z)dz\Big]
\\
&\times\int_{\Omega}\Big[\int_{\Omega}F(y,s,t)k_{B}(y,t)dt\Big]u_{\xi}(s)ds\Big)dy.
\end{align*}
Now we approximate the function $k_{B}(\cdot,t)\in C_{{\rm
L}}^{\infty}(\overline{\Omega})$ by Taylor polynomial type
expansions. By using Proposition \ref{TaylorExp}, we get
\begin{align*}
\sigma_{AB}(x,\xi)&\sim
u_{\xi}^{-1}(x)\int_{\Omega}\Big(\Big[\int_{\Omega}F(x,y,z)k_{A}(x,z)dz\Big]
\\
&\times\int_{\Omega}\Big[\int_{\Omega}F(y,s,t)\sum_{\alpha\geq 0}
\frac{1}{\alpha!}q^{\alpha}(x,y)D^{(\alpha)}_{x}k_{B}(x,t)dt\Big]u_{\xi}(s)ds\Big)dy
\\
&=\sum_{\alpha\geq 0}
\frac{1}{\alpha!}u_{\xi}^{-1}(x)\int_{\Omega}\Big(\Big[\int_{\Omega}F(x,y,z)q^{\alpha}(x,y)k_{A}(x,z)dz\Big]
\\
&\times\int_{\Omega}\Big[\int_{\Omega}F(y,s,t)D^{(\alpha)}_{x}k_{B}(x,t)dt\Big]u_{\xi}(s)ds\Big)dy.
\end{align*}
Since
\begin{align*}
u_{\xi}^{-1}(y)\int_{\Omega}\Big[\int_{\Omega}&F(y,s,t)D^{(\alpha)}_{x}k_{B}(x,t)dt\Big]
u_{\xi}(s)ds=\\
&=u_{\xi}^{-1}(y)\int_{\Omega}\Big[\int_{\Omega}\sum_{\eta\in\ind}u_{\eta}(y) \ \overline{v_{\eta}(s)} \
\overline{v_{\eta}(t)}D^{(\alpha)}_{x}k_{B}(x,t)dt\Big]
u_{\xi}(s)ds
\\
&=\sum_{\eta\in\ind} u_{\xi}^{-1}(y)u_{\eta}(y) \
\int_{\Omega}\Big[\int_{\Omega}
D^{(\alpha)}_{x}k_{B}(x,t)\ \overline{v_{\eta}(t)}dt\Big]
u_{\xi}(s) \overline{v_{\eta}(s)} ds
\\
&=\sum_{\eta\in\ind} u_{\xi}^{-1}(y)u_{\eta}(y) \ \Big[
\int_{\Omega} u_{\xi}(s) \overline{v_{\eta}(s)} ds
\Big]\times \Big[\int_{\Omega} D^{(\alpha)}_{x}k_{B}(x,t)\
\overline{v_{\eta}(t)}dt\Big]
\\
&=u_{\xi}^{-1}(y)u_{\xi}(y) \
D^{(\alpha)}_{x}\widehat{k_{B}}(x,\xi)
\\
&=D^{(\alpha)}_{x}\widehat{k_{B}}(x,\xi)
\\
&=D^{(\alpha)}_{x}\sigma_{B}(x,\xi),
\end{align*}
using Definition \ref{DEF: DifferenceOper}, we have
\begin{align*}
\sigma_{AB}(x,\xi)\sim \sum_{\alpha\geq 0}
\frac{1}{\alpha!}(\Delta_{(x)}^{\alpha}\sigma_{A}(x,\xi))D^{(\alpha)}_{x}\sigma_{B}(x,\xi).
\end{align*}
Omitting a routine treatment of the remainder, this completes the proof.
\end{proof}

\section{Properties of integral kernels}
\label{SEC:kernels}

We now establish some properties of Schwartz kernels of
pseudo-differential operators with symbols in the introduced
H\"ormander-type classes. In the following Theorem \ref{TH:
KernelofPDO}, let us make the  assumption on the growth of
$L^\infty$-norms of the eigenfunctions $u_\xi$. Finding estimates
for the norms $\|u_{\xi}\|_{L^{\infty}}$ in terms of the
corresponding eigenvalues of L is a challenging problem even for
self-adjoint operators L, see e.g. Sogge and Zelditch \cite{Sogge-Zelditch:max-ef-growth-Duke} and
references therein. Thus, on tori or, more generally, on
compact Lie groups, the eigenfunctions of the Laplacian can be
chosen to be uniformly bounded. However, even for the Laplacian,
on more general manifolds, such growth depends on the geometry of
the manifold. We refer to
\cite[Remark 8.9]{Delgado-Ruzhansky:invariant} for a more thorough discussion
of this topic as well as for a list of relevant references.

\begin{theorem}[Kernel of a pseudo--differential operator] \label{TH: KernelofPDO}
Let $\mu_0$ be a constant such that there is $C>0$ such that for all $\xi\in\ind$
we have
$$
\|u_{\xi}\|_{L^{\infty}}\leq C \langle\xi\rangle^{\mu_0}.
$$
Let $a\in S^{\mu}_{\rho,\delta}(\overline{\Omega}\times\ind)$, $\rho>0$. Then the kernel
$K(x,y)$ of the pseudo-differential operator ${\rm Op_L}a$ satisfies
\begin{equation}\label{EQ:ests-L}
|({\rm L}^{\ast}_{y})^{k}K(x,y)|\leq C_{Nk}|x-y|^{-N}
\end{equation}
for any $N>(\mu+mk+2\mu_0+s_0)/\rho$ and $x\neq y$,
where $m$ is the order of the differential operator ${\rm L}$
and $s_0$ is the constant from Assumption \ref{Assumption_4}.
\end{theorem}

In particular, if L is for example locally elliptic, \eqref{EQ:ests-L} implies that
for $x\neq y$, the kernel $K(x,y)$ is a smooth function.
And, if $a\in S^{-\infty}(\overline{\Omega}\times\ind)$, then the
integral kernel $K(x,y)$ of ${\rm Op_L}a$ is smooth in $x$ and $y$.

\begin{proof}
By Corollary \ref{COR: SymFor} we have
$$
a(x,\xi)=u_{\xi}^{-1}(x)\int_{\Omega}K(x,y)u_{\xi}(y)dy.
$$
By using Definition \ref{DEF: DifferenceOper} and by direct
calculations, recalling \eqref {EQ:diff-symb} we have
\begin{align*}
\Delta_{(x)}^{\alpha}a(x,\xi)&=u_{\xi}^{-1}(x)\sum_{\eta\in\ind}\mathcal
F_{L}\Big(q^{\alpha}(x,\cdot)u_{\xi}(\cdot)\Big)(\eta)a(x,\eta)u_{\eta}(x)\\
&=u_{\xi}^{-1}(x)\sum_{\eta\in\ind}\mathcal
F_{L}\Big(q^{\alpha}(x,\cdot)u_{\xi}(\cdot)\Big)(\eta)\int_{\Omega}K(x,y)u_{\eta}(y)dy\\
&=u_{\xi}^{-1}(x)\int_{\Omega}K(x,y)\left[\sum_{\eta\in\ind}\mathcal
F_{L}\Big(q^{\alpha}(x,\cdot)u_{\xi}(\cdot)\Big)(\eta)u_{\eta}(y)\right]dy\\
&=u_{\xi}^{-1}(x)\int_{\Omega}q^{\alpha}(x,y)K(x,y)u_{\xi}(y)dy,
\end{align*}
and also
\begin{multline*}
u_\xi(x)\lambda_\xi^k \Delta_{(x)}^{\alpha} a(x,\xi)=\int_{\Omega}q^{\alpha}(x,y)K(x,y)\lambda_{\xi}^{k}u_{\xi}(y)dy
\\ =
\int_{\Omega}q^{\alpha}(x,y)K(x,y){\rm L}_{y}^{k}u_{\xi}(y)dy
=\int_{\Omega}
({\rm L}^{\ast}_{y})^{k}(q^{\alpha}(x,y)K(x,y))u_{\xi}(y)dy.
\end{multline*}
This means that $$({\rm L}^{\ast}_{y})^{k}(q^{\alpha}(x,y)K(x,y))=
{\mathcal F}_{\rm L}^{-1} (u_\xi(x)\lambda_\xi^k \Delta_{(x)}^{\alpha} a(x,\xi))(y).$$
Since it follows from assumptions that
$$\lambda_{\xi}^{k}\Delta_{(x)}^{\alpha}  a(x,\xi)\in
S^{\mu+mk-\rho|\alpha|}(\overline{\Omega}\times\ind),$$ we have
$$
\lambda_{\xi}^{k} |\Delta_{(x)}^{\alpha}a(x,\xi)|\leq
C\langle\xi\rangle^{\mu+mk-\rho|\alpha|}.
$$
We recall now the norm $$\|a(x,\cdot)\|_{l^{1}({\rm L})}=\sum_{\xi\in\ind}| a(x,\xi)|
\|u_{\xi}\|_{L^{\infty}(\Omega)}$$ from Section \ref{SEC:lp}.
It follows that
$$
\|u_{\xi}(x)\lambda_{\xi}^{k}\Delta_{(x)}^{\alpha}a(x,\xi)\|_{l^{1}({\rm L})}\leq
C\sum_{\xi\in\ind}\langle\xi\rangle^{\mu+mk-\rho|\alpha|}
\|u_{\xi}\|_{L^{\infty}(\Omega)}^2\leq
C\sum_{\xi\in\ind}\langle\xi\rangle^{\mu+mk-\rho|\alpha|+2\mu_0}.
$$
Consequently, if
$$|\alpha|>(\mu+mk+2\mu_0+s_0)/\rho,$$
where $s_0$ is the constant from Assumption \ref{Assumption_4},
we have that
$u_{\xi}(x)\lambda_{\xi}^{k}\Delta_{(x)}^{\alpha}a(x,\xi)$ is in $l^{1}({\rm L})$
with respect to $\xi$, and hence
$({\rm L}^{\ast}_{y})^{k}(q^{\alpha}(x,y)K(x,y))$ is in $L^\infty$ by the Hausdorff-Young
inequality in Theorem \ref{TH: HY}.
Since ${\rm L}^{\ast}_{y}$ is a differential operator, we also have
$$
q^\alpha(x,y)({\rm L}^{\ast}_{y})^{k}K(x,y)\in
L^{\infty}(\Omega\times\Omega)
$$
for such $\alpha$.
By the properties of $q^\alpha$
it implies the statement of the theorem.
\end{proof}

The singular support of $w\in\mathcal D'_{{\rm L}}(\Omega)$ is
defined as the complement of the set where $w$ is smooth. Namely,
$x\notin {\rm sing\,supp}\,\, w$ if there is an open neighbourhood $U$ of
$x$ and a smooth function $f\in C_{{\rm
L}}^{\infty}(\overline{\Omega})$ such that $w(\varphi)=f(\varphi)$
for all $\varphi\in C_{{\rm L}}^{\infty}(\overline{\Omega})$ with
${\rm supp} \,\varphi\subset U$.
As an immediate consequence of Theorem \ref{TH: KernelofPDO} we
obtain the information on how the singular support is mapped by a
pseudo-differential operator:

\begin{corollary}\label{COR: SingularSupp}
Let $\sigma_{A}\in
S^{\mu}_{\rho,\delta}(\overline{\Omega}\times\ind)$, $1\geq \rho>\delta\geq 0$. Then for every
$w\in\mathcal D'_{{\rm L}}(\Omega)$ we have
\begin{equation*}\label{EQ: SingSupp}
{\rm sing\,supp}\,\, Aw\subset {\rm sing\,supp}\,\, w.
\end{equation*}
\end{corollary}
For elliptic operators,
in Corollary \ref{COR:SingularSupp-ell} we state also the inverse inclusion.

\section{${\rm L}$-elliptic pseudo--differential operators}
\label{SEC:elliptic}

In this section we discuss operators that are elliptic in the symbol classes
generated by L. For such operators we can obtain parametrix and then also
a-priori estimates by the properties of pseudo-differential operators in, for example,
Sobolev spaces, once they are established in Section \ref{SEC:L2}, see Theorem \ref{L2-Bs and El-ty}.
Thus, from the asymptotic expansion for the composition of
pseudo-differential operators, we get an expansion for a
parametrix of an elliptic operator:

\begin{theorem}[L-ellipticity]\label{El-ty}
Let $1\geq \rho>\delta\geq 0$.
Let $\sigma_A\in S^\mu_{\rho,\delta}(\overline{\Omega}\times\ind)$ be
elliptic in the sense that there exist constants $C_0>0$ and
$N_0\in\mathbb N$ such that
\begin{equation}\label{elliptic}
  |\sigma_A(x,\xi)|
  \geq C_0 \langle\xi\rangle^\mu
\end{equation}
for all $(x,\xi)\in\overline{\Omega}\times\ind$ for which
$\xi\geq N_0$; this is equivalent to assuming that there exists
$\sigma_B\in S^{-\mu}_{\rho,\delta}(\overline{\Omega}\times\ind)$ such that
$I-BA,I-AB$ are in ${\rm Op_L} S^{-\infty}$.
 Let $$A \sim \sum_{j=0}^\infty A_j,$$
 $\sigma_{A_j}\in S^{\mu-(\rho-\delta)j}_{\rho,\delta}(\overline{\Omega}\times\ind)$.
Then $$B \sim \sum_{k=0}^\infty B_k, $$ where $B_k\in
S^{-\mu-(\rho-\delta)k}_{\rho,\delta}(\overline{\Omega}\times\ind)$ is such that
$$\sigma_{B_0}(x,\xi)= 1/\sigma_{A_0}(x,\xi)$$ for large enough
$\xi$, and recursively
$$
  \sigma_{B_N}(x,\xi) = \frac{-1}{\sigma_{A_0}(x,\xi)}
  \sum_{k=0}^{N-1} \sum_{j=0}^{N-k}
  \sum_{|\alpha|=N-j-k}
        \frac{1}{\alpha!} \left[
          \Delta_{(x)}^{\alpha} \sigma_{A_j}(x,\xi) \right]
        D_x^{(\alpha)} \sigma_{B_k}(x,\xi).
$$
\end{theorem}

\begin{proof} Now $I\sim BA,$ so that by the composition Theorem
\ref{Composition} we have
\begin{align*}
1&\sim\sigma_{BA}(x, \xi)
\\
&\sim\sum_{\alpha\in\mathbb
N_0^{l}}\frac{1}{\alpha!}(\Delta_{(x)}^{\alpha}\sigma_{B}(x,\xi))D^{(\alpha)}_{x}\sigma_{A}(x,\xi)
\\
&\sim\sum_{\alpha\in\mathbb
N_0^{l}}\frac{1}{\alpha!}(\Delta_{(x)}^{\alpha}\sum_{k=0}^\infty
\sigma_{B_k}(x,\xi))D^{(\alpha)}_{x}\sum_{j=0}^\infty
\sigma_{A_j}(x,\xi),
\end{align*} where we want to solve it for
$\sigma_{B_k}$. Notice that $A_0$ is elliptic if and only if $A$
is elliptic. Moreover, without a loss of generality we may assume
that $\sigma_{A_0}$ does not vanish anywhere. Obviously, we can
demand that $1=\sigma_{B_0}(x,\xi)\sigma_{A_0}(x,\xi)$, and that
$$
0=\sum_{j+k+|\alpha|=N}\frac{1}{\alpha!}\Big[\Delta_{(x)}^{\alpha}
\sigma_{B_k}(x,\xi)\Big] D^{(\alpha)}_{x} \sigma_{A_j}(x,\xi).
$$
Then the trivial solution of these equations is the recursion of
the theorem. It is easy to check that $\sigma_{B_N}\in
S^{-\mu-(\rho-\delta)N}_{\rho,\delta}(\overline{\Omega}\times\ind)$. Thus
$B\sim\sum_{k=0}^{\infty}B_k$.
\end{proof}

Theorem \ref{TH: KernelofPDO} applied to the parametrix from in
Theorem \ref{El-ty}, implies the inverse inclusion to the singular supports from
Corollary \ref{COR: SingularSupp} for elliptic operators:

\begin{corollary}\label{COR:SingularSupp-ell}
Let $1\geq \rho>\delta\geq 0$ and assume that
$\sigma_{A}\in S^{\mu}_{\rho,\delta}(\overline{\Omega}\times\ind)$ is {\rm L}-elliptic. Then for every
$w\in\mathcal D'_{{\rm L}}(\Omega)$ we have
\begin{equation*}\label{EQ: SingSupp}
{\rm sing\,supp}\,\, Aw={\rm sing\,supp}\,\, w.
\end{equation*}
\end{corollary}

\section{Sobolev embedding theorem}
\label{SEC:embeddings}


In this section we prove an example of a Sobolev embedding theorem
for Sobolev spaces $\mathcal H^s_{\rm L}$ associated to L, considered in Section \ref{SEC:Sobolev}.
However, only limited conclusions
are possible in the abstract setting when no further specifics about L
are available.
Now, let $C({\Omega})$ be the Banach space
under the norm
$$
\|f\|_{C({\Omega})}:=\sup\limits_{x\in{\Omega}}
|f(x)|.
$$
We recall that we have a differential operator L of order $m$ with
smooth coefficients in the open set $\Omega\subset\mathbb R^n$,
and also the operator ${\rm L}^\circ$ from \eqref{EQ:Lo-def}.

\medskip
The following theorem is conditional to the local
regularity estimate \eqref{EQ:Sob-as}. It is satisfied with $\varkappa=1$
if, for example, L is locally elliptic, i.e. elliptic in the classical sense of $\mathbb R^n$.
However, if L is for example a sum of squares satisfying
H\"ormander's commutator condition, the number $\varkappa\geq 1$ may depend on the order
to which the H\"ormander condition is satisfied, see e.g.
\cite{Garetto-Ruzhansky:sum-JDE} in the context of compact Lie groups.

\begin{theorem}\label{TH:SETh}
Let $k$ be an integer such that $k>n/2$.
Let $\varkappa$ be such that the operators ${\rm L}$ and ${\rm L}^{\circ}$
satisfy the inequality
\begin{equation}\label{EQ:Sob-as}
\Big\|\frac{\partial^{\alpha}f}{\partial
x^{\alpha}}\Big\|_{L^{2}(\Omega)}\leq C\Big\|({\rm I}+{\rm L}^\circ{\rm
L})^{\frac{\varkappa k}{2m}}f\Big\|_{L^{2}(\Omega)}
\end{equation}
for all $f\in C^{\infty}({\Omega})$, for all
$\alpha\in\mathbb N_0^{n}$ with $|\alpha|\leq k$. Then we have the
continuous embedding
$$
\mathcal H^{\varkappa k}_{{\rm L}}(\Omega)\hookrightarrow C({\Omega}).
$$
\end{theorem}

\begin{proof}
By the local Sobolev embedding theorem at $x\in\Omega$, for
$|\alpha|\leq k$, we have
$$
|f(x)|\leq C\left(\sum_{|\alpha|\leq k}\int_{\Omega}\Big|\frac{\partial^{\alpha}}{\partial y^{\alpha}}f(y)\Big|^{2}dy\right)^{1/2}.
$$
Thus, considering $f\in C^\infty(\Omega)$ without loss of generality due to the density of
$C^\infty(\Omega)$, and for all
$x\in{\Omega}$, in view of the assumption \eqref{EQ:Sob-as}
we have
\begin{align*}
|f(x)|&\leq C\left(\sum_{|\alpha|\leq
k}\int_{\Omega}\Big|\frac{\partial^{\alpha}}{\partial
y^{\alpha}}f(y)\Big|^{2}dy\right)^{1/2}
\\
&\leq C\left(\int_{\Omega}|({\rm I}+{\rm L}^\circ{\rm
L})^{\frac{\varkappa k}{2m}}f(y)|^{2}dy\right)^{1/2}
\\
&\leq C\left(\sum_{\xi\in\ind}
\langle\xi\rangle^{2\varkappa k}\widehat{f}(\xi)\overline{\widehat{f}^{\ast}(\xi)}\right)^{1/2}=
C\|f\|_{\mathcal
H^{\varkappa k}_{{\rm L}}(\Omega)}.
\end{align*} is true. Hence
we obtain the statement of the theorem.
\end{proof}

\section{Conditions for $L^{2}$-boundedness}
\label{SEC:L2}

In this section we will discuss what conditions on the ${\rm L}$-symbol
$a$ guarantee the $L^{2}$-boundedness of the
corresponding pseudo-differential operator ${\rm
Op_L}(a):C^{\infty}_{{\rm L}}(\overline{\Omega})\rightarrow \mathcal
D'_{{\rm L}}(\Omega)$.

\begin{theorem}\label{L2-Bs}
Let $k$ be an integer $>n/2$.
Let $a:\overline{\Omega}\times\ind\rightarrow\mathbb C$ be such
that
\begin{equation}\label{EQ:torus-L2-pse}
  \left| \partial^{\alpha}_{x} a(x,\xi) \right| \leq
  C \quad\textrm{ for all } (x,\xi)\in\overline{\Omega}\times\ind,
\end{equation}
and all $|\alpha|\leq k$, all $x\in\Omega$ and $\xi\in\ind$. Then the operator
${\rm Op_L}(a)$ extends to a bounded operator from $L^{2}(\Omega)$
to $L^{2}(\Omega)$.
\end{theorem}


\begin{proof}
Let us define an operator $A_y$ by
$$
A_{y}f(x):=\sum_{\xi\in\ind}\int_{\Omega} u_{\xi}(x)\overline{v_{\xi}(z)}a(y,
\xi)f(z)dz,
$$
so that $A_{x}f(x)={\rm Op_L}(a)f(x).$ Then
$$
\|{\rm Op_L}(a)f\|_{L^{2}(\Omega)}^{2}=
\int_{\Omega}|A_{x}f(x)|^{2}dx\leq\int_{\Omega}\sup\limits_{y\in\Omega}|A_{y}f(x)|^{2}dx,
$$
and by an application of the local Sobolev embedding theorem we get
$$
\sup_{y\in{\Omega}}|A_{y}f(x)|^{2}\leq C\sum_{|\alpha|\leq
k}\int_{\Omega}|\partial_{y}^{\alpha}A_{y}f(x)|^{2}dy.
$$
Therefore, using the Fubini theorem to change the order of
integration, we obtain
\begin{align*}
\|{\rm Op_L}(a)f\|_{L^{2}(\Omega)}^{2}&\leq C\sum_{|\alpha|\leq
k}\int_{\Omega}\int_{\Omega}|\partial_{y}^{\alpha}A_{y}f(x)|^{2}dxdy
\\
&\leq C\sum_{|\alpha|\leq k}\sup_{y\in{\Omega}}\int_{\Omega}|
\partial_{y}^{\alpha}A_{y}f(x)|^{2}dx
\\
&=C\sum_{|\alpha|\leq k}\sup_{y\in{\Omega}}\|\partial_{y}^{\alpha}
A_{y}f\|_{L^{2}(\Omega)}^{2}
\\
&\leq C\sum_{|\alpha|\leq k}\sup_{y\in{\Omega}}\sup_{\xi\in\ind}\|\partial_{y}^{\alpha}
a(y,\xi)\|_{L^{2}(\Omega)}^{2}\|f\|_{L^{2}(\Omega)}^{2},
\end{align*}
using the $L^{2}$-boundedness of multipliers with bounded symbols following
from Lemma \ref{LEM: FTl2}, completing the proof.
\end{proof}

From a suitable adaption of the composition Theorem
\ref{Composition}, using that by Proposition \ref{TaylorExp}
the operators $\partial^{\alpha}_{x}$ and $D^{(\alpha)}_{x}$ can
be expressed in terms of each other as linear combinations with smooth
coefficients, we immediately obtain the result in Sobolev
spaces:

\begin{corollary}\label{Hs-Bs}
Let $k$ be an integer $>n/2$.
Let $\mu\in\mathbb R$ and let $a:\overline{\Omega}\times\mathbb
Z\rightarrow\mathbb C$ be such that
\begin{equation}\label{Hs-cond}
  \left| \partial^{\alpha}_{x} a(x,\xi) \right| \leq
  C \langle\xi\rangle^{\mu} \quad\textrm{ for all } (x,\xi)\in\overline{\Omega}\times\ind,
\end{equation}
and for all $\alpha$. Then operator ${\rm Op_L}(a)$ extends
to a bounded operator from $\mathcal H^{s}_{L}(\Omega)$ to
$\mathcal H^{s-\mu}_{L}(\Omega),$ for any $s\in\mathbb R.$
\end{corollary}

By using Theorem \ref{El-ty} and Corollary \ref{Hs-Bs}, we get

\begin{theorem}\label{L2-Bs and El-ty}
Let $A$ be an elliptic pseudo-differential operator with {\rm L}-symbol
$\sigma_A\in S^{\mu}(\overline{\Omega}\times\ind)$,
$\mu\in\mathbb R$, and let $Au=f$ in $\Omega$, $u\in \mathcal
H^{\infty}_{{\rm L}}(\Omega)$. Then we have the estimate
$$
\|u\|_{\mathcal H^{s+\mu}_{{\rm L}}(\Omega)}\leq C_{sN}
(\|f\|_{\mathcal H^{s}_{{\rm L}}(\Omega)}+
\|u\|_{\mathcal H^{-N}_{{\rm L}}(\Omega)}).
$$
for any $s, N\in\mathbb R$.
\end{theorem}

\begin{proof} Since $A$ is an elliptic pseudo--differential operator with L-symbol
$\sigma_A\in S^{\mu}(\overline{\Omega}\times\ind)$, by
Theorem \ref{El-ty} there exists a pseudo-differential operator
$A^{\sharp}$ with symbol $\sigma_{A^{\sharp}}\in
S^{-\mu}(\overline{\Omega}\times\ind)$, such that
$$
AA^{\sharp}=A^{\sharp}A=I+R,
$$
where $R\in {\rm Op_L}(S^{-\infty})$. Thus
$$A^{\sharp}f=A^{\sharp}Au=(I+R)u$$
and
$$
u=A^{\sharp}f-Ru.
$$
Now, for $f\in\mathcal
 H^{s}_{{\rm L}}(\Omega)$ we have $A^{\sharp}f\in\mathcal
 H^{s+\mu}_{{\rm L}}(\Omega).$
As $u\in \mathcal H^{\infty}_{{\rm L}}(\Omega)$ there exists
$s_{0}\in\mathbb R$ such that $u\in \mathcal H^{s_{0}}_{{\rm
L}}(\Omega)$. Then $Ru\in C_{{\rm
L}}^{\infty}(\overline{\Omega})$. Hence, we have
$$
u=(A^{\sharp}f-Ru)\in\mathcal
 H^{s+\mu}_{{\rm L}}(\Omega).
$$
By using Corollary \ref{Hs-Bs}, we complete the proof.
\end{proof}


\end{document}